\documentclass[reqno,final]{amsart}

\usepackage{natbib}   
\usepackage{fancyhdr} 
\usepackage{color}    
\usepackage{hyperref} 
\usepackage{graphicx} 
\usepackage{epstopdf} 
\usepackage{srcltx}   

\usepackage{mathtools}
\usepackage{amssymb}
\usepackage{mathrsfs}  
\usepackage{caption}   
\usepackage{float}     
\usepackage{booktabs}
\usepackage{lscape}    

\definecolor{aleacolor}{rgb}{0.16,0.59,0.78}

\hypersetup{
breaklinks,
colorlinks=true,
linkcolor=aleacolor,
urlcolor=aleacolor,
citecolor=aleacolor}

\renewcommand{\headheight}{11pt}
\renewcommand{\cite}{\citet}

\theoremstyle{plain}
\newtheorem{theorem}{Theorem}[section]
\newtheorem{proposition}[theorem]{Proposition}
\newtheorem{lemma}[theorem]{Lemma}
\newtheorem{corollary}[theorem]{Corollary}
\theoremstyle{definition}
\newtheorem{definition}[theorem]{Definition}
\theoremstyle{remark}
\newtheorem{remark}{Remark}[section]

\makeatletter \@addtoreset{equation}{section} \makeatother
\renewcommand\theequation{\thesection.\arabic{equation}}
\renewcommand\thefigure{\thesection.\arabic{figure}}
\renewcommand\thetable{\thesection.\arabic{table}}

\elogo{\parbox[c]{78pt}{\includegraphics[width=78pt]{logo_white.png}}} 

\newcommand{\Z}{\mathbb{Z}}

\newcommand{\R}{\mathbb{R}}
\newcommand{\N}{\mathbb{N}}
\newcommand{\PP}{\mathbb{P}}

\newcommand{\HH}{\mathcal{H}}
\newcommand{\F}{\mathcal{F}}
\newcommand{\nvert}[0]{~\vert~}
\newcommand{\bb}[1]{\boldsymbol{#1}}
\newcommand{\prob}[2]{\mathbb{P}_{#1}\hspace{-0.6mm}\left( #2 \right)}
\newcommand{\esp}[2]{\mathbb{E}_{#1}\hspace{-0.6mm}\left[ #2 \right]}
\newcommand{\var}[2]{\mathbb{V}_{#1}\hspace{-0.6mm}\left( #2 \right)}

\newcommand{\probw}[2]{\mathscr{P}_{#1}\hspace{-0.6mm}\left( #2 \right)}
\newcommand{\espw}[2]{\mathscr{E}_{#1}\hspace{-0.6mm}\left[ #2 \right]}
\newcommand{\para}{(\boldsymbol{\sigma},\boldsymbol{\lambda})}
\newcommand{\tabitem}{~~\llap{\textbullet}~~}

\begin{document}

\noindent
\fbox{
\begin{minipage}{34.5em}
    This is an electronic reprint of the original article : \\
    ALEA, Lat. Am. J. Probab. Math. Stat. 14 (2), 851-902 (2017). \\
    \url{http://alea.impa.br/articles/v14/14-38.pdf}. \\
    This reprint {\it may differ} from the original in typographic details.
\end{minipage}
}
\vspace{10mm}

\title[Scale-inhomogeneous Gaussian free field]{Geometry of the Gibbs measure for the discrete 2D Gaussian free field with scale-dependent variance\vspace{8mm}}
\author[F. Ouimet]{Fr\'ed\'eric Ouimet}
\address{Universit\'e de Montr\'eal, D\'epartement de Math\'ematiques et de Statistique, \newline 2920, chemin de la Tour, Montr\'eal, QC H3T 1J4, Canada.}
\email{ouimetfr@dms.umontreal.ca}
\urladdr{\url{https://sites.google.com/site/fouimet26}}


    \begin{abstract}
        We continue our study of the scale-inhomogeneous Gaussian free field introduced in \cite{MR3541850}.
        Firstly, we compute the limiting free energy on $V_N$ and adapt a technique of \cite{MR2070335} to determine the limiting two-overlap distribution.
        The adaptation was already successfully applied in the simpler case of \cite{MR3354619}, where the limiting free energy was computed for the field with two levels (in the center of $V_N$) and the limiting two-overlap distribution was determined in the homogeneous case.
        Our results agree with the analogous quantities for the Generalized Random Energy Model (GREM); see \cite{MR883541} and \cite{MR2070334}, respectively.
        Secondly, we show that the extended Ghirlanda-Guerra identities hold exactly in the limit.
        As a corollary, the limiting array of overlaps is ultrametric and the limiting Gibbs measure has the same law as a Ruelle probability cascade.
    \end{abstract}

    \maketitle

\section{The model}\label{sec:model}

    Let $(W_k)_{k\geq 0}$ be a simple random walk starting at $u\in \Z^2$ with law $\mathscr{P}_u$. For every finite box $B \subseteq \Z^2$, the Gaussian free field (GFF) on $B$ is a centered Gaussian field $\phi \circeq \{\phi_v\}_{v\in B}$ with covariance matrix
    \begin{equation}\label{eq:GFF.green.function}
        G_B(v,v') \circeq \frac{\pi}{2} \cdot \espw{v}{\sum_{k=0}^{\tau_{\partial B} - 1} \bb{1}_{\{W_k = v'\}}}, \quad v,v'\in B,
    \end{equation}
    where $\tau_{\partial B}$ is the first hitting time of $(W_k)_{k\geq 0}$ on the \textit{boundary of $B$},
    \begin{equation}
        \partial B \circeq \{v\in B : \exists z\in \Z^2 \backslash B ~~\text{such that } \|v - z\|_2 = 1\},
    \end{equation}
    and $\|\cdot\|_2$ denotes the Euclidean distance in $\Z^2$.
    With this definition, $B$ contains its boundary.
    We let $B^o \circeq B \backslash \partial B$.
    By convention, summations are zero when there are no indices, so $\phi$ is identically zero on $\partial B$. This is the {\it Dirichlet boundary condition}.
    The constant $\pi/2$ in \eqref{eq:GFF.green.function} is a convenient normalization for the variance.

    We build a family of Gaussian fields constructed from the GFF $\{\phi_v^N\}_{v\in V_N}$ on the square box $V_N \circeq \{0,1,...,N\}^2$.
    For $\lambda \in (0,1)$ and $v=(v_1,v_2)\in V_N$, consider the closed neighborhood $[v]_\lambda$ in $V_N$ consisting of the square box of width $N^{1-\lambda}$ centered at $v$ that has been cut off by the boundary of $V_N$ :
    \begin{equation}
        [v]_{\lambda} \circeq \left((v_1,v_2) + \Big[-\frac{1}{2}N^{1-\lambda},\frac{1}{2}N^{1-\lambda}\Big]^2\right) \bigcap V_N\ .
    \end{equation}
    By convention, define $[v]_0 \circeq V_N$ and $[v]_1 \circeq \{v\}$.
    Let $\F_{\partial [v]_{\lambda} \cup [v]_{\lambda}^c} \circeq \sigma(\{\phi^N_v, v\notin [v]_{\lambda}^o\})$ be the $\sigma$-algebra generated by the variables on the boundary of the box $[v]_{\lambda}$ and those outside of it.
    Since the neighborhoods are shrinking when $\lambda$ increases, for any $v\in V_N$, the collection $\mathbb{F}_v \circeq \{\F_{\partial [v]_{\lambda} \cup [v]_{\lambda}^c}\}_{\lambda\in [0,1]}$ is a filtration.
    In particular, if we let
    \begin{equation}
        \phi^N_v(\lambda) \circeq \mathbb{E}\Big[\phi^N_v \nvert \F_{\partial [v]_{\lambda} \cup [v]_{\lambda}^c}\Big],
    \end{equation}
    then
    \begin{equation*}
        \text{for every $v\in V_N$, the process $(\phi^N_v(\lambda))_{\lambda\in[0,1]}$ is a $\mathbb{F}_v$-martingale.}
    \end{equation*}
    It is also a Gaussian field, therefore disjoint increments of the form $\phi^N_v(\lambda') - \phi^N_v(\lambda)$ are independent.
    These observations motivate the definition of {\it scale-inhomogeneous Gaussian free field}, which can be seen as a martingale-transform of $(\phi^N_v(\lambda))_{\lambda\in[0,1]}$ applied simultaneously for every $v\in V_N$.

    Fix $M\in \N$ and consider the parameters
    \begin{align}
        &\boldsymbol{\sigma} \circeq (\sigma_1,\sigma_2, ..., \sigma_M)\in (0,\infty)^M, ~& \text{(variance parameters)} \\
        &\boldsymbol{\lambda} \circeq (\lambda_1,\lambda_2, ..., \lambda_M)\in (0,1]^M, ~& \text{(scale parameters)}
    \end{align}
    where
    \begin{equation}
        0 \circeq \lambda_0 < \lambda_1 < ... < \lambda_M \circeq 1.
    \end{equation}
    We write $\nabla_i$ for the difference operator with respect to the index $i$.
    When the index variable is obvious, we omit the subscript.
    For example,
    \begin{equation}
        \nabla \phi^N_v(\lambda_i) \circeq \phi^N_v(\lambda_i) - \phi^N_v(\lambda_{i-1}).
    \end{equation}

    \begin{definition}[Scale-inhomogeneous Gaussian free field]\label{def:IGFF}
        Let $\{\phi^N_v\}_{v\in V_N}$ be the GFF on $V_N$.
        The $(\boldsymbol{\sigma},\boldsymbol{\lambda})$-GFF on $V_N$ is a Gaussian field $\{\psi^N_v\}_{v\in V_N}$ defined by
        \begin{equation}\label{eq:def.IGFF}
            \psi^N_v \circeq \sum_{i=1}^M \sigma_i \nabla \phi^N_v(\lambda_i)= \sum_{i=1}^M \sigma_i \big(\phi^N_v(\lambda_i)-\phi^N_v(\lambda_{i-1})\big).
        \end{equation}
        Similarly to the GFF, we define
        \begin{equation}
            \psi^N_v(\lambda) \circeq \mathbb{E}\Big[\psi^N_v \nvert \F_{\partial [v]_{\lambda} \cup [v]_{\lambda}^c}\Big] \quad \text{and} \quad
            \psi_v^N(\lambda,\lambda') \circeq \psi_v^N(\lambda') - \psi_v^N(\lambda).
        \end{equation}
        From hereon, we make the dependence on $N$ implicit everywhere for $\phi$ and $\psi$.
    \end{definition}

    The $\para$-GFF is the analogue of various types of inhomogeneous branching processes :
    \begin{itemize}
        \item The GREM, see e.g. \cite{MR2070334,MR2070335,MR883541,Derrida85,MR875300};
        \item The non-hierarchical GREM, see \cite{MR2209333,MR2531095};
        \item The perceptron GREM, see \cite{MR3372858};
        \item The multi-scale logarithmic potential (also called multi-scale log-REM), see e.g. \cite{MR2425780,SciPostPhys.1.2.011};
        \item The branching random walk in time-inhomogeneous environment, see e.g. \cite{MR2968674,MR3373310,MR3361256,arXiv:1509.08172};
        \item The variable speed branching Brownian motion, see e.g. \cite{MR3164771,MR3351476,MR2981635,MR3531703}.
    \end{itemize}

    \begin{remark}
        Our model is most closely related to the multi-scale log-REM of \cite{MR2425780}. In both cases :
        \begin{enumerate}
            \item There is a boundary effect, meaning that the variance of the variables of the field decays to $0$ as we approach the boundary;
            \item There is no exact hierarchical structure;
            \item The covariance between two random variables of the field is directly tied to the distance between their index in a finite-dimensional Euclidean space;
            \item The number of possible covariance values grows with the size of the system.
        \end{enumerate}
        Amongst the other models, not one has property 1 or 3, only the non-hierarchical GREM has property 2, and only the time-inhomogeneous branching random walk and the variable speed branching Brownian motion have property 4.

        The only significant difference between the single-scale log-REM and the critical GFF is the fact that the field is indexed in a $N'$-dimensional space (instead of 2-dimensional) and the covariance is defined to be logarithmic directly instead of it being a consequence of estimates on the Green function in two dimensions.
        \cite{MR2425780} calculate the limiting free energy and the limiting two-overlap distribution in the limit $N'\to \infty$ (for a finite system size) by using the {\it replica trick} and {\it Parisi's hierarchical ansatz}. In the thermodynamic limit, they recover the same structure as in the GREM case and argue that the results should also hold if $N'$ is fixed instead and the system size grows to infinity.
        In this article, we put their argument on rigorous ground via the $\para$-GFF. We go even further by showing that the limiting Gibbs measure has the same law as the one found in \cite{MR2070334} for the GREM. For an introduction to log-REM models and physical motivations, see e.g. \cite{PhysRevE.63.026110,MR2430565,MR2425780,MR2882779,SciPostPhys.1.2.011}.
    \end{remark}

\section{Motivation for the scale-inhomogeneous GFF}\label{sec:motivation}

    In contrast with branching random walks (BRWs) :
    \begin{itemize}
        \item The branching structure is {\it approximate} in the sense that $\phi_v(\lambda)$ and $\phi_{v'}(\lambda)$ are not perfectly correlated when $\lambda$ is smaller than the {\it branching scale}, namely the largest scale at which $[v]_{\lambda}$ and $[v']_{\lambda}$ intersect.
            The branching scale itself is arbitrarily defined since it is conceptually more of a transition interval : between the scale where $v'$ is ``well-inside'' $[v]_{\lambda}$ and the smallest scale for which $[v]_{\lambda} \cap [v']_{\lambda} = \emptyset$.
        \item At a given scale, the covariance of the increments of the field decays near the boundary of the domain. In the context of BRWs, this means that at a given time, the law of each point process would depend on the position of the associated ancestors in the tree.
    \end{itemize}
    Several covariance estimates can be found in Appendix \ref{sec:covariance.estimates}.

    We are interested in the $\para$-GFF to see how the results on the extremes (and the methods of proof) are robust to perturbations in the correlation structure.
    This interest is amplified by the fact that many models in recent applications have underlying approximate branching structures.
    For example :
    \begin{itemize}
        \item[$\bullet$] Cover times, see e.g. \cite{MR3263552,MR3129800,MR3602852,MR3126579,MR1998762,MR2123929,MR2206347,MR2946152,MR3178464,MR2912708,MR2921974};
        \item[$\bullet$] The randomized Riemann zeta function on the critical line, see e.g. \cite{MR3619786,arXiv:1706.08462,arXiv:1304.0677,arXiv:1604.08378};
        \item[$\bullet$] The Riemann zeta function on random intervals of the critical line, see e.g. \cite{arXiv:1612.08575,arXiv:1611.05562};
        \item[$\bullet$] The characteristic polynomials of random unitary matrices, see e.g. \cite{MR3594368,arXiv:1607.00243,arXiv:1602.08875};
    \end{itemize}
    In particular, note that all these models are heavily correlated in the {\it critical regime}, that is when the correlation starts to affect the extremal statistics.

    Generally, there are two ways to study the distribution of the extremes : via the {\it extremal process} and via the {\it Gibbs measure}.
    Since one of our goal here is to show the tree structure of the extremes in the limit  $N\to \infty$ (this interest comes partly from the physicists and the Parisi ultrametricity conjecture for mean field spin glass models (see e.g. \cite{MR2252929,MR1026102,MR3052333,MR2731561,MR3024566})), the mathematics in the latter approach is much simpler (at least in our case). A very important advancement was made recently in \cite{MR2999044} where it is shown that a random measure supported on the unit ball of a separable Hilbert space that satisfies the extended Ghirlanda-Guerra identities must have an ultrametric support with probability one. The summary in Section \ref{sec:structure} gives a detailed description of the steps we will make to prove the extended Ghirlanda-Guerra identities and the consequences we can deduce from the work of Panchenko.

    Remarkably, despite the imperfect branching structure of the $\para$-GFF and the growing number of scales as $N\rightarrow\infty$, the results of this paper show that the {\it limiting Gibbs measure} has the same tree structure as in the context of the GREM. More precisely, we show that the limiting Gibbs measure has the same law as a {\it Ruelle probability cascade} (see \cite{MR875300}) with {\it functional order parameter} $\zeta$ defined by the {\it limiting two-overlap distribution}. This is the content of Corollary \ref{cor:ruelle.probability.cascade}.
    In the limit, this means, in particular, that the extremes of the model are clustered in a hierarchical way and in fact satisfy the {\it ultrametric inequality}, see Corollary \ref{cor:ultrametricity}.

    Another reason why the study of the extremes via the Gibbs measure might be more desirable to prove ultrametricity results is its robustness, i.e. the applicability of the methods to other models.
    For instance, \cite{MR3628881} defines the notion of {\it approximate ultrametricity} for finite system sizes by imposing conditions on the sequence of Gibbs measures. It is proved that if the sequence of two-overlap distributions converge weakly and the {\it approximate extended Ghirlanda-Guerra identities} are satisfied (in his sense, see Definition 1.4 in \cite{MR3628881}), then the sequence of Gibbs measures (assuming they are supported on the unit ball of a separable Hilbert space) is {\it regularly approximately ultrametric}.
    His paper tie in very nicely with our approach since we prove the weak convergence of the two-overlap distribution in Theorem \ref{thm:two.overlap.distribution.limit}, we prove a slightly different version of the approximate extended Ghirlanda-Guerra identities in Theorem \ref{thm:IGFF.ghirlanda.guerra}, and then we show that the identities must hold exactly in the limit (Theorem \ref{thm:IGFF.ghirlanda.guerra.limit}).
    Of course, this doesn't prove that our model is {\it regularly approximately ultrametric}, but it seems at least plausible that the notion of approximate ultrametricity could hold for a large class of non-hierarchical models and could be (part of) the grand explanation behind the phenomenon of ultrametricity of the extremes in the system size limit.

\section{Structure of the paper}\label{sec:structure}

    In order to make the logical structure of this article as clear as possible, the new results and their proof are stated and written in a linear fashion.
    Some technical lemmas are relegated to appendices.
    However, we emphasize that these lemmas are not at all necessary to understand the main structure and are sparsely used.
    Below, we summarize the main results of the paper, give the main ingredients of the proofs and indicate exactly where the lemmas in the appendices are needed.

    \vspace{5mm}
    \noindent
    {{\bf Section \ref{sec:previous.results} :}} We recall the main results of \cite{MR3541850} :
            \begin{itemize}
                \item Theorem \ref{thm:IGFF.order1} : $\max_{v\in V_N} \psi_v / \log N^2 \stackrel{\PP}{\longrightarrow} \gamma^{\star}$.
                \item Theorem \ref{thm:IGFF.high.points} : $\log(|\{v\in V_N : \psi_v \geq \gamma \log N^2\}|) / \log N^2 \stackrel{\PP}{\longrightarrow} \mathcal{E}(\gamma).$
            \end{itemize}

    \noindent
    {{\bf Section \ref{sec:new.results} :}} The new main results are stated :
            \begin{itemize}
                \item Theorem \ref{thm:IGFF.free.energy} : Limit of the free energy on $V_N$ :
                    \begin{equation*}
                        \frac{1}{\log N^2} \log \sum_{v\in V_N} e^{\beta \psi_v} \xrightarrow{~\PP ~\text{and}~ L^p} \max_{\gamma\in [0,\gamma^{\star}]} \beta \gamma + \mathcal{E}(\gamma) = f^{\psi}(\beta).
                    \end{equation*}
                    The proof uses the results of Section \ref{sec:previous.results} for the $\PP$-convergence and we show that the powers of the free energy are uniformly integrable.
                    We find the explicit form of the maximum on the right-hand side using
                    \begin{itemize}
                        \item Lemma \ref{lem:tech.lemma.1} : Differentiability of $\mathcal{E}$,
                        \item Lemma \ref{lem:tech.lemma.2} : Solution of the maximization problem.
                    \end{itemize}
                \item Theorem \ref{thm:IGFF.free.energy.A.N.p} : Same as Theorem \ref{thm:IGFF.free.energy}, but on a set far enough from the boundary of $V_N$, denoted $A_{N,\rho}$.
                \item Theorem \ref{thm:two.overlap.distribution.limit} : If $\mathcal{G}_{\beta,N}$ denotes the Gibbs measure of $\psi$ and $q^N(v,v')$ denotes the normalized covariance (overlap) between $\psi_v$ and $\psi_{v'}$, then we compute the limit of the {\it two-overlap distribution}, namely $r \mapsto \mathbb{E}\mathcal{G}_{\beta,N}^{\times 2} \big[\bb{1}_{\{q^N(v,v') \leq r\}}\big]$. The main ingredients of the proof are :
                    \begin{itemize}
                        \item The Gibbs measure doesn't hold any weight outside $A_{N,\rho}$ in the limit,
                        \item The overlap estimates of Corollary \ref{cor:covariance.estimates.3},
                        \item Gaussian integration by parts,
                        \item The mean convergence of the derivative at $u=0$ of a perturbed version of the free energy to $f^{\psi^u}(\beta)$. This is proved using
                            \begin{itemize}
                                \item Theorem \ref{thm:IGFF.free.energy.A.N.p},
                                \item Lemma \ref{lem:IGFF.modified.free.energy.convexity} : Convexity of the free energy with respect to the perturbation variable $u$,
                                \item Lemma \ref{lem:IGFF.modified.free.energy.derivative} : Differentiability of $u\mapsto f^{\psi^u}(\beta)$ for all $|u| < \delta$.
                            \end{itemize}
                    \end{itemize}
                \item Theorem \ref{thm:IGFF.ghirlanda.guerra} : As $N\rightarrow \infty$, the extended Ghirlanda-Guerra identities hold approximately. The main ingredients are the same as in the proof of Theorem \ref{thm:two.overlap.distribution.limit}, but we also need Theorem \ref{thm:IGFF.free.energy.A.N.p} and the main ingredients to get a concentration result (Lemma \ref{lem:IGFF.ghirlanda.guerra.restricted.2}).
                \item Theorem \ref{thm:IGFF.ghirlanda.guerra.limit} : In the limit, the extended Ghirlanda-Guerra identities hold exactly. The main ingredients of the proof are :
                    \begin{itemize}
                        \item Theorem \ref{thm:two.overlap.distribution.limit} and Theorem \ref{thm:IGFF.ghirlanda.guerra},
                        \item The representation theorem of \cite{MR666087}.
                    \end{itemize}
            \end{itemize}

    \noindent
    {{\bf Section \ref{sec:consequences.GG.limit} :}} The consequences of Theorem \ref{thm:two.overlap.distribution.limit} and Theorem \ref{thm:IGFF.ghirlanda.guerra.limit} :
            \begin{itemize}
                \item Corollary \ref{cor:ultrametricity} : The limiting array of overlaps is almost-surely {\it ultrametric} under the mean of the limiting Gibbs measure.
                \item Corollary \ref{cor:ruelle.probability.cascade} : The limiting Gibbs measure has the same law as a Ruelle probability cascade \cite{MR875300} with functional order parameter $\zeta$ defined by the limiting two-overlap distribution, which means that it samples extreme values from an {\it exact} tree structure, where the number of branching scales is finite and controlled by the {\it inverse temperature} $\beta$.
            \end{itemize}

    \noindent
    {{\bf Section \ref{sec:proof.main.results} :}} Proof of the main results.

    \noindent
    {{\bf Appendix \ref{sec:covariance.estimates} :}} The sole purpose of this appendix is to prove Corollary \ref{cor:covariance.estimates.3}.

    \noindent
    {{\bf Appendix \ref{sec:technical.lemmas} :}} We indicated above where each of its four lemmas are needed.

\section{Some notations}\label{sec:some.notations}

    Now, we introduce some notations.
    The parameters $(\boldsymbol{\sigma}, \boldsymbol{\lambda})$ can be encoded simultaneously in the left-continuous step function
    \begin{equation}\label{eq:variance.function}
        \sigma(s) \circeq \sigma_1 \bb{1}_{\{0\}}(s) + \sum_{i=1}^M \sigma_i \bb{1}_{(\lambda_{i-1},\lambda_i]}(s), \ \ \ s\in [0,1].
    \end{equation}
    For any positive measurable function $f:[0,1]\to\R$, define the integral operators
    \begin{equation}
        \mathcal{J}_f(s) \circeq \int_0^s f(r) dr \quad \text{and} \quad \mathcal{J}_f(s_1,s_2) \circeq \int_{s_1}^{s_2} f(r) dr.
    \end{equation}
    We refer to $\mathcal{J}_{\sigma^2}(\cdot)$ as the {\it speed function}. The \textit{concavification} of $\mathcal{J}_{\sigma^2}$, denoted $\hat{\mathcal{J}}_{\sigma^2}$, is the function whose graph is the concave hull of $\mathcal{J}_{\sigma^2}$.
    From \cite{MR3541850}, we know that the asymptotics of the maximum and the log-number of high points of $\psi$ are controlled by this function.
    Its graph is an increasing and concave polygonal line, see Figure \ref{fig:concave} for an example.

    Clearly, there exists a unique non-increasing left-continuous step function $s \mapsto \bar{\sigma}(s)$ such that
    \vspace{-2mm}
    \begin{equation}
        \hat{\mathcal{J}}_{\sigma^2}(s) = \mathcal{J}_{\bar{\sigma}^2}(s) = \int_0^s \bar{\sigma}^2(r)\, dr \  \text{ for all $s\in (0,1]$\hspace{0.3mm}. }
    \end{equation}
    The scales in $[0,1]$ where $\bar{\sigma}$ jumps are denoted by
    \begin{equation}\label{eq:effective.scales}
        0 \circeq \lambda^0 < \lambda^1 < ... < \lambda^m \circeq 1,
    \end{equation}
    where $m \leq M$. To be consistent with previous notations, we set $\bar{\sigma}_l \circeq \bar{\sigma}(\lambda^l)$.
    In particular, note that
    \vspace{-1mm}
    \begin{equation}
        \bar{\sigma}_1 > \bar{\sigma}_2 > ... > \bar{\sigma}_m.
    \end{equation}
    We define $\bar{\sigma}_{m+1} \circeq 0$ and interpret $\beta_c(\bar{\sigma}_{m+1}) \circeq 2 / \bar{\sigma}_{m+1}$ as $+\infty$ whenever it is encountered.

    \begin{figure}[ht]
        \centering
        \includegraphics[scale=0.7]{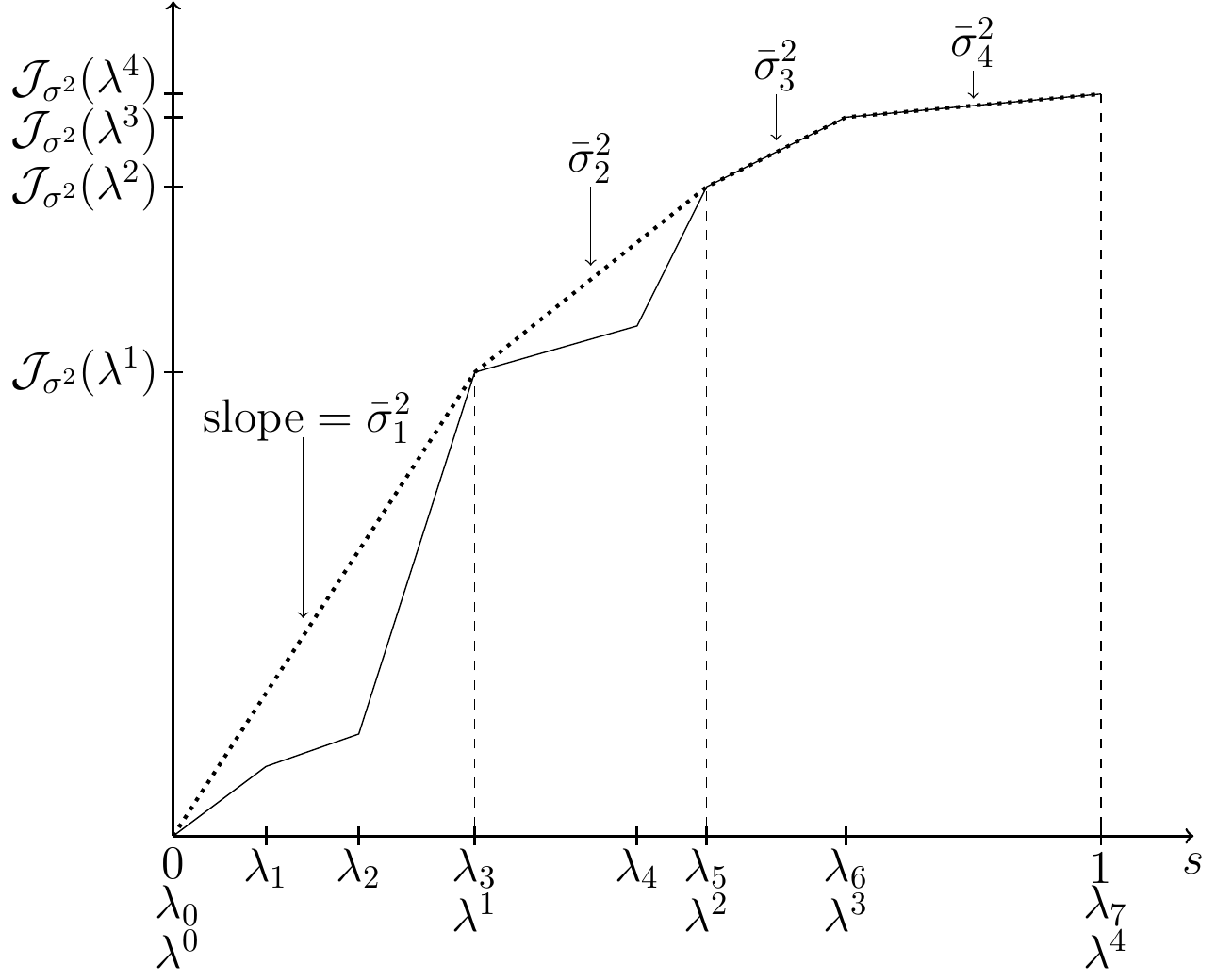}
        \captionsetup{width=0.8\textwidth}
        \caption{Example of $\mathcal{J}_{\sigma^2}$ \hspace{-1mm}(closed line) and $\hat{\mathcal{J}}_{\sigma^2}$ \hspace{-1mm}(dotted line) with $7$ values for $\sigma^2$.\vspace{-2mm}}
        \label{fig:concave}
    \end{figure}

\section{Previous results}\label{sec:previous.results}

    In this section, we recall the main results from \cite{MR3541850} on the first order asymptotics of the maximum and the log-number of $\gamma$-high points of the $\para$-GFF.
    These results are needed to compute the limiting free energy.

    \begin{theorem}\label{thm:IGFF.order1}
        Let $\{\psi_v\}_{v\in V_N}$ \hspace{-0.8mm}be the $(\boldsymbol{\sigma},\boldsymbol{\lambda})$-GFF on $V_N$ of Definition \ref{def:IGFF}, then
        \begin{equation}
            \lim_{N\rightarrow\infty} \frac{\max_{v\in V_N} \psi_v}{\log N^2} = \mathcal{J}_{\sigma^2\hspace{-0.3mm}/\bar{\sigma}}(1)\circeq \gamma^\star \quad \text{in probability}\hspace{0.3mm}.
        \end{equation}
        In fact, from Lemma 3.1 and Lemma 3.3 in \cite{MR3541850}, for any $\varepsilon \hspace{-0.2mm}> \hspace{-0.2mm}0$, there exists a constant $c = c(\varepsilon,\boldsymbol{\sigma},\boldsymbol{\lambda}) > 0$ such that for $N$ large enough,
        \begin{equation}\label{eq:IGFF.order1.upper.bound}
            \mathbb{P}\left(\max_{v\in V_N} \psi_v \geq (1 + \varepsilon) \gamma^{\star} \log N^2\right) \leq N^{-c}
        \end{equation}
        and
        \begin{equation}\label{eq:IGFF.order1.lower.bound}
            \mathbb{P}\left(\max_{v\in V_N} \psi_v \leq (1 - \varepsilon) \gamma^{\star} \log N^2\right) \leq N^{-c}.
        \end{equation}
    \end{theorem}

    The set of $\gamma$-high points of the field $\psi$ is defined as
    \begin{equation}
        \HH_N(\gamma) \circeq \{v\in V_N : \psi_v \geq \gamma \log N^2\},\quad \text{for all $\gamma\in [0,\gamma^{\star}]$.}
    \end{equation}
    The number of $\gamma$-high points depends on {\it critical levels} defined by
    \begin{equation}\label{eq:critic.levels}
        \gamma^{l}
        \circeq \int_0^1 \frac{\sigma^2(s)}{\bar{\sigma}(s \wedge \lambda^l)} ds, \quad 1\leq l\leq m,
        \qquad \gamma^0\circeq 0 \hspace{0.3mm}.
    \end{equation}
    For $\gamma\in (\gamma^{l-1},\gamma^l]$, define
    \begin{equation}\label{eq:entropy}
        \mathcal{E}(\gamma) \circeq (1 - \lambda^{l-1}) - \frac{(\gamma - \mathcal{J}_{\sigma^2\hspace{-0.3mm} / \bar{\sigma}}(\lambda^{l-1}))^2}{\mathcal{J}_{\sigma^2}(\lambda^{l-1},1)}
        \quad \text{and} \quad \mathcal{E}(0) \circeq 1.
    \end{equation}

    \begin{theorem}\label{thm:IGFF.high.points}
        Let $\{\psi_v\}_{v\in V_N}$ \hspace{-0.8mm}be the $(\boldsymbol{\sigma},\boldsymbol{\lambda})$-GFF on $V_N$ of Definition \ref{def:IGFF} and let $\gamma\in [0,\gamma^{\star})$, then
        \vspace{-1mm}
        \begin{equation}
            \lim_{N\rightarrow\infty} \frac{\log |\HH_N(\gamma)|}{\log N^2} = \mathcal{E}(\gamma)
            \quad
            \text{in probability\hspace{0.3mm}.}
        \end{equation}
        In fact, from Lemma 3.4 and Lemma 3.5 in \cite{MR3541850}, for any $\gamma\in [0,\gamma^{\star}]$ and for any $\varepsilon > 0$, there exists a constant $c = c(\gamma,\varepsilon,\boldsymbol{\sigma},\boldsymbol{\lambda}) > 0$ such that for $N$ large enough,
        \begin{equation}\label{eq:IGFF.high.points.upper.bound}
            \mathbb{P}\Big(|\HH_N(\gamma)| \geq N^{2\mathcal{E}(\gamma) + \varepsilon}\Big) \leq N^{-c},
        \end{equation}
        and, when $\gamma\in [0,\gamma^{\star})$,
        \begin{equation}\label{eq:IGFF.high.points.lower.bound}
            \mathbb{P}\Big(|\HH_N(\gamma)| < N^{2\mathcal{E}(\gamma) - \varepsilon}\Big) \leq N^{-c}.
        \end{equation}
    \end{theorem}

    \begin{remark}
        The case $\gamma = 0$ is not explicitly covered in \cite{MR3541850}.
        In that case, \eqref{eq:IGFF.high.points.upper.bound} is trivial (the probability is $0$) and \eqref{eq:IGFF.high.points.lower.bound} is a simple and direct application of the Paley-Zygmund inequality.
    \end{remark}

\section{New results}\label{sec:new.results}

    The first main result of this article concerns the {\it free energy} of the $\para$-GFF, which is defined by
    \vspace{-1mm}
    \begin{equation}\label{eq:IGFF.free.energy}
        f_N^{\psi}(\beta) \circeq \frac{1}{\log N^2} \log Z_N^{\psi}(\beta), \quad \beta > 0,
    \end{equation}
    where $Z_N^{\psi}(\beta) \circeq \sum_{v\in V_N} e^{\beta \psi_v}$.
    The $L^1$-limit of the free energy will be central to obtain the limiting two-overlap distribution of the $\para$-GFF and the extended Ghirlanda-Guerra identities.
    This limit is better expressed in terms of the limiting free energy of the $\text{REM}(\sigma)$ model consisting of $N^2$ i.i.d. Gaussian variables of variance $\sigma^2 \log N$.
    From Theorem 8.1 in \cite{MR1890289},
    \begin{align}\label{eq:REM.free.energy}
        f^{\text{REM}(\sigma)}(\beta)
        &\circeq \lim_{N\rightarrow \infty} \frac{\log Z_N^{\text{REM}(\sigma)}(\beta)}{\log N^2} \notag \\
        &\stackrel{\text{a.s.}}{=}
        \left\{\hspace{-1mm}
        \begin{array}{ll}
            2 (\beta / \beta_c(\sigma)), &\mbox{if } \beta >  \beta_c(\sigma), \\
            1 + (\beta / \beta_c(\sigma))^2, &\mbox{if } \beta \leq \beta_c(\sigma),
        \end{array}
        \right.
    \end{align}
    for all $\beta > 0$, where $\beta_c(\sigma) \circeq 2 / \sigma$.

    \begin{theorem}[Limit of the free energy on $V_N$]\label{thm:IGFF.free.energy}
        Let $\{\psi_v\}_{v\in V_N}$ be the $(\boldsymbol{\sigma},\boldsymbol{\lambda})$-GFF on $V_N$ of Definition \ref{def:IGFF}, then
        \begin{align}\label{eq:thm:IGFF.free.energy}
            \lim_{N\rightarrow\infty} f_N^{\psi}(\beta)
            &= \max_{\gamma\in [0,\gamma^{\star}]} (\beta \gamma + \mathcal{E}(\gamma)) \notag \\
            &= \sum_{j=1}^m f^{\text{REM}(\bar{\sigma}_j)}(\beta) \nabla \lambda^j \circeq f^{\psi}(\beta),
        \end{align}
        where the limit holds in probability and in $L^p, ~1 \leq p < \infty$.
    \end{theorem}

    \begin{remark}
        This result was first proved for the GREM by \cite{MR883541}, although with vastly different notations.
        The expression for the GREM can also be recovered from Theorem 1.6 in \cite{MR2070334}.
    \end{remark}

    In \cite{MR3354619}, the convergence in probability (and in $L^1$) of the free energy of the scale-inhomogeneous GFF with two variance parameters was only proved on a subset of $V_N$ that excludes the points that are too close to the boundary $\partial V_N$. This was done because the decay of variance near $\partial V_N$ makes the asymptotics of the log-number of points in sets of the form
    \begin{equation}\label{eq:gamma.high.set}
            \HH_{A_N}(\gamma) \circeq \{v\in A_N : \psi_v \geq \gamma \log N^2\}
    \end{equation}
    harder to determine when $A_N$ includes points close to $\partial V_N$. In contrast, let
    \begin{equation}\label{eq:A.N.rho.2}
        A_{N,\rho} \circeq \Big\{v\in V_N : \min_{z\in \Z^2 \backslash V_N} \|v - z\|_2 \geq N^{1 - \rho}\Big\}, \quad \rho\in (0,1].
    \end{equation}
    Then, it turns out that for $\rho > 0$ small enough, the variance of the increments of $\psi$ on $A_{N,\rho}$ is within a uniform bound from the analogue quantity in the context of the GREM (except when the scale $0$ is involved), see Lemma \ref{lem:IGFF.variance.estimates}.

    Theorem \ref{thm:IGFF.free.energy} not only generalizes Theorem 2.1 in \cite{MR3354619}, but is also stronger because it tells us that including points arbitrarily close to $\partial V_N$ in the free energy has no impact on its limit, as long as we include the center of $V_N$. We are able to prove Theorem \ref{thm:IGFF.free.energy} here because the asymptotics of $|\mathcal{H}_N(\gamma)|$ were proved on $V_N$ in \cite{MR3541850}.

    Even though Theorem \ref{thm:IGFF.free.energy} is interesting on its own, we will instead use the version on $A_{N,\rho}$ later in this article.

    \begin{theorem}[Limit of the free energy on $A_{N,\rho}$]\label{thm:IGFF.free.energy.A.N.p}
        Let $\{\psi_v\}_{v\in V_N}$ be the $(\boldsymbol{\sigma},\boldsymbol{\lambda})$-GFF on $V_N$ of Definition \ref{def:IGFF} and define
        \begin{equation}\label{eq:IGFF.free.energy.A.N.p}
            f_{N,\rho}^{\psi}(\beta) \circeq \frac{1}{\log N^2} \log Z_{N,\rho}^{\psi}(\beta), \quad \beta > 0,
        \end{equation}
        where $Z_{N,\rho}^{\psi}(\beta) \circeq \sum_{v\in A_{N,\rho}} e^{\beta \psi_v}$. Then, for all $\rho\in (0,1]$,
        \begin{equation}\label{eq:thm:IGFF.free.energy.A.N.p}
            \lim_{N\rightarrow\infty} f_{N,\rho}^{\psi}(\beta) = f^{\psi}(\beta),
        \end{equation}
        where the limit holds in probability and in $L^p, ~1 \leq p < \infty$.
    \end{theorem}

    For the second main result of this article, consider the normalized covariances or {\it overlaps} of the $\para$-GFF :
    \begin{equation}\label{eq:overlap}
        q^N(v,v') \circeq \frac{\esp{}{\psi_v \psi_{v'}}}{\mathcal{J}_{\sigma^2}(1) \log N + C_0}, \quad v,v'\in V_N,
    \end{equation}
    where $C_0$ is the constant introduced in Lemma \ref{lem:IGFF.variance.uniform.bound}.
    The overlap is the covariance divided by the uniform upper bound on the variance.
    From the Cauchy-Schwarz inequality, it is clear that $|q^N(v,v')| \leq 1$, for any $v,v'\in V_N$.

    We are concerned with the limiting distribution of the overlaps, when the variables are sampled from the \textit{Gibbs measure}
    \begin{equation}\label{eq:IGFF.gibbs.measure}
        \mathcal{G}_{\beta,N}(\{v\}) \circeq \frac{e^{\beta \psi_v}}{Z_N^{\psi}(\beta)}, \quad v\in V_N.
    \end{equation}
    Since the Gibbs measure samples extreme values of the field $\psi$, the overlaps under the Gibbs measure can be interpreted as measures of relative distance between the extremes.
    In spin-glass theory, the relevant object to classify the extreme value statistics of strongly correlated variables is the {\it two-overlap distribution}
    \begin{equation}\label{eq:two.overlap.distribution}
        \mathbb{E}\mathcal{G}_{\beta,N}^{\times 2} \big[\bb{1}_{\{q^N(v,v') \leq r\}}\big], \quad r\in [0,1].
    \end{equation}
    Since the overlaps are normalized, their asymptotics will also be normalized to lie in $[0,1]$.
    Define
    \vspace{-1mm}
    \begin{equation}\label{eq:normalized.speed.function}
        \bar{\mathcal{J}}_{\sigma^2}(\cdot) \circeq \frac{\mathcal{J}_{\sigma^2}(\cdot)}{\mathcal{J}_{\sigma^2}(1)}.
    \end{equation}
    This is motivated by the fact that if
    \begin{equation}
        b_N(v,v') \circeq \max\{\lambda\in [0,1] : [v]_{\lambda} \cap [v']_{\lambda} \neq \emptyset\}
    \end{equation}
    denotes the {\it branching scale between $v$ and $v'$} (the analogue of the normalized branching time for BRWs), then Corollary \ref{cor:covariance.estimates.3} in Appendix \ref{sec:covariance.estimates} shows (in particular) that for all $\rho\in (0,1]$ and $N$ large enough,
    \begin{equation}
        \max_{v,v'\in A_{N,\rho}} \hspace{-1mm}\big|q^N(v,v') - \bar{\mathcal{J}}_{\sigma^2}(b_N(v,v'))\big| \leq \frac{C_7}{\sqrt{\log N}} + C_8 \, \rho.
    \end{equation}
    For any {\it inverse temperature} $\beta > 0$, denote
    \begin{equation}\label{eq:l.beta}
        l_{\beta} \circeq
        \left\{\hspace{-1mm}
        \begin{array}{ll}
            \min\{l\in \{1,...,m\} : \beta \leq \beta_c(\bar{\sigma}_l) \circeq 2 / \bar{\sigma}_l\}, &\mbox{if } \beta \leq 2 / \bar{\sigma}_m, \\
            m+1, &\mbox{otherwise}.
        \end{array}
        \right.
    \end{equation}
    \noindent
    This is the smallest index $l$ for which a {\it critical inverse temperature} $\beta_c(\bar{\sigma}_l)$ is larger or equal to $\beta$.

    \begin{remark}
        In \cite{MR2070334}, $l(\beta)$ is defined such that $l(\beta) = l_{\beta} - 1$.
        Our choice is more natural with the notation we used in \eqref{eq:critic.levels} and \eqref{eq:entropy}, see the proof of Lemma \ref{lem:tech.lemma.2}.
    \end{remark}

    \begin{remark}
        In order to state the results on the limiting two-overlap distribution and on the Ghirlanda-Guerra identities, we need to restrict $\beta$ from taking any critical value $\beta_c(\bar{\sigma}_j) \circeq 2 / \bar{\sigma}_j$.
        The restriction will be used in exactly two places throughout the article :
        \begin{itemize}
            \item[$\bullet$] In the proof of Theorem \ref{thm:two.overlap.distribution.limit} (Limiting two-overlap distribution);
            \item[$\bullet$] In the proof of Lemma \ref{lem:IGFF.ghirlanda.guerra.restricted.2}. The result of Lemma \ref{lem:IGFF.ghirlanda.guerra.restricted.2} is then used to prove Theorem \ref{thm:IGFF.ghirlanda.guerra} (Approximate extended Ghirlanda-Guerra identities).
        \end{itemize}
        Let
        \begin{equation}
            \mathcal{B} \circeq (0,\infty) \backslash \cup_{1 \leq j \leq m} \{2 / \bar{\sigma}_j\}.
        \end{equation}
        Theorem \ref{thm:two.overlap.distribution.limit} and Theorem \ref{thm:IGFF.ghirlanda.guerra} (and their consequences) will hold as long as $\beta \in \mathcal{B}$.
        In the proof of both theorems, the restriction $\beta\in \mathcal{B}$ is needed to obtain the mean convergence of the derivative of a perturbed version of the free energy, see \eqref{eq:convergence.free.energy.derivative}. Ultimately, the restriction is related to Lemma \ref{lem:IGFF.modified.free.energy.derivative}, where it is shown that the limiting free energy is differentiable with respect to a perturbation inside $[\lambda^{j-1},\lambda^j]$ if and only if $\beta \neq 2 /\bar{\sigma}_j$.
    \end{remark}

    We are now ready to state the second main result of this article.

    \begin{theorem}[Limiting two-overlap distribution]\label{thm:two.overlap.distribution.limit}
        Let $\{\psi_v\}_{v\in V_N}$ be the $(\boldsymbol{\sigma},\boldsymbol{\lambda})$-GFF on $V_N$ of Definition \ref{def:IGFF}.
        Then, for $\beta\in \mathcal{B}$,
        \begin{equation}
            \lim_{N\rightarrow \infty} \mathbb{E}\mathcal{G}_{\beta,N}^{\times 2} \big[\bb{1}_{\{q^N(v,v') \leq r\}}\big] =
            \left\{\hspace{-1mm}
            \begin{array}{ll}
                0, &\mbox{if } r < 0, \\
                \beta_c(\bar{\sigma}_j) / \beta, &\mbox{if } r\in [x^{j-1},x^j), ~j \leq l_{\beta}-1, \\
                1, &\mbox{if } r \geq x^{l_{\beta}-1},
            \end{array}
            \right.
        \end{equation}
        where $\beta_c(\bar{\sigma}_j) \circeq 2 / \bar{\sigma}_j$ and $x^j \circeq \bar{\mathcal{J}}_{\sigma^2}(\lambda^j)$.
    \end{theorem}

    \begin{remark}
        This is the same expression as in the context of the GREM. Compare this to Proposition 1.11 in \cite{MR2070334}.
        In the homogeneous case, the theorem was proved by \cite{MR3211001} for a certain class of non-hierarchical log-correlated Gaussian fields with no boundary effect, by \cite{MR3354619} for the GFF (trivial adjustments of their proof show the same result for the BRW), and by \cite{MR3539644} for the BRW (using an alternative method).
    \end{remark}

    \begin{remark}
        In the context of the GFF, we have to show that the Gibbs measure doesn't carry any weight outside $A_{N,\rho}$ in the limit.
        In \cite{MR3354619}, a crucial step was to use self-averaging and Slepian's lemma in order to compare the free energy outside $A_{N,\rho}$ with that of a REM.
        Here, we find an upper bound on the free energy outside $A_{N,\rho}$ through the optimization problems for the maximum and $\gamma$-high points, see the proof of Lemma \ref{lem:restricted.free.energy.tech.lemma}.
        This approach is much more efficient when there are several {\it effective scales} $\lambda^j$.
    \end{remark}

    \begin{remark}
        Theorem \ref{thm:two.overlap.distribution.limit} tells us that even though the overlap between the extremes can be (almost) anything between $0$ and $1$ for finite system sizes, this variability disappears in the limit. The extremes can only branch asymptotically at the {\it effective distances} $N^{1-d}$, $d\in \{0,\lambda^1,\lambda^2,...,\lambda^{l_{\beta}-1}\}$, where $\lambda^j$ is defined in \eqref{eq:effective.scales}.
        We see that the number of branching scales $\lambda^j$ for the extremes is finite in the limit and increases as the {\it inverse temperature} $\beta$ becomes larger than some of the {\it critical thresholds} $0 < \beta_c(\bar{\sigma}_1) < \beta_c(\bar{\sigma}_2) < ... < \beta_c(\bar{\sigma}_{l_{\beta}-1}) < \infty$.
        In comparison, for homogeneous models (like the GFF and the BRW), there is only one critical inverse temperature
        \begin{itemize}
            \item above which the extremes only branch at scale $0$ or $1$ in the limit, and
            \item under which the extremes only branch at scale $0$ in the limit (meaning that the extremes are all asymptotically uncorrelated).
        \end{itemize}
    \end{remark}

    The third and final main result of this article concerns the {\it Ghirlanda-Guerra identities}.
    These identities were introduced in \cite{MR1662161} and an extended version of the identities was proved for a general class of models, called the mixed $p$-spin, in \cite{MR2600075}.
    Before taking the limit, we have the following approximate version.

    \begin{theorem}[Approximate extended Ghirlanda-Guerra identities]\label{thm:IGFF.ghirlanda.guerra}
        Let $\beta \in \mathcal{B}$, and let $\alpha < \alpha'$ be any pair of scales such that
        \begin{equation}\label{eq:scales.are.points.of.continuity.TOD}
           \lambda^{j-1} < \alpha < \alpha' < \lambda^{j + (m-j)\bb{1}_{\{j = l_{\hspace{-0.3mm}\beta}\hspace{-0.3mm}\}}}
        \end{equation}
        for some $j\in \{1,2,...,l_{\beta}\}$.
        Denote $\bb{v} \circeq (v^1,v^2,...,v^s)$ and $S_{\alpha,\alpha'} \circeq (\bar{\mathcal{J}}_{\sigma^2}(\alpha),\bar{\mathcal{J}}_{\sigma^2}(\alpha')]$.
        Then, for any $s\in \N$, any $k\in \{1,...,s\}$, and any functions $h : V_N^s \rightarrow \R$ such that $\sup_N \|h\|_{\infty} < \infty$,
        \begin{equation}\label{eq:thm:IGFF.ghirlanda.guerra.AGGI}
            \lim_{N \rightarrow \infty} \left|\hspace{-1mm}
            \begin{array}{l}
                \vspace{2mm}\mathbb{E}\mathcal{G}_{\beta,N}^{\times (s+1)}\big[\int_{S_{\alpha,\alpha'}} \hspace{-1mm} \bb{1}_{\{r < q^N(v^k,v^{s+1})\}} dr ~h(\bb{v})\big] \\[1mm]
                - \left\{\hspace{-1.5mm}
                \begin{array}{l}
                    \vspace{0.5mm}\frac{1}{s} \mathbb{E}\mathcal{G}_{\beta,N}^{\times 2}\big[\int_{S_{\alpha,\alpha'}} \hspace{-1mm}\bb{1}_{\{r < q^N(v^1,v^2)\}} dr\big] \mathbb{E}\mathcal{G}_{\beta,N}^{\times s}\big[h(\bb{v})\big] \\
                    + \frac{1}{s} \sum_{l \neq k}^s \mathbb{E}\mathcal{G}_{\beta,N}^{\times s}\big[\int_{S_{\alpha,\alpha'}} \hspace{-1mm}\bb{1}_{\{r < q^N(v^k,v^l)\}} dr ~h(\bb{v})\big]
                \end{array}
                \hspace{-1.5mm}\right\}
            \end{array}
            \hspace{-1mm}\right| = 0.\vspace{0.5mm}
        \end{equation}
    \end{theorem}

    \begin{remark}
        The word ``approximate'' here is {\bf NOT} to be understood in exactly the same sense as in Definition 1.4 of \cite{MR3628881}.
        It is approximate in the sense that the limit $N\to \infty$ is taken, but also because linear combinations of functions of the form $q\mapsto \int_{S_{\alpha,\alpha'}} \hspace{-1mm}\bb{1}_{\{r < q\}} dr$ do not describe all the bounded mesurable functions defined on $[0,1]$.
    \end{remark}

    \newpage
    We now show that the extended form of the Ghirlanda-Guerra identities hold exactly in the limit.
    Along with Theorem \ref{thm:two.overlap.distribution.limit}, these identities completely determine the law of the limiting array of overlaps when the variables are sampled by the Gibbs measure, see Section \ref{sec:consequences.GG.limit}.

    We follow closely the reasoning from page 101 in \cite{MR3052333} and page 1459 in \cite{MR3211001}.
    Let $(v_l)_{l\in \N}$ be an i.i.d. sequence sampled from the Gibbs measure $\mathcal{G}_{\beta,N}$ and let
    \begin{align}
        R^N
        &\circeq (R_{l,l'}^N)_{l,l'\in \N} \circeq (q^N(v^l,v^{l'}))_{l,l'\in \N}
    \end{align}
    be the array of overlaps of this sample.
    Note that the array $R^N$ is symmetric and non-negative definite because the entries are normalized covariances of the Gaussian field $\psi$. Since each point is sampled independently, it is also {\it weakly exchangeable}, namely, for any permutation $\pi$ of a finite number of indices,
    \begin{equation}
        \big(R_{\pi(l),\pi(l')}^N\big) \stackrel{\text{law}}{=} \big(R_{l,l'}^N\big).
    \end{equation}

    The push-forward of the probability measure $\mathbb{E}\mathcal{G}_{\beta,N}^{\times \infty}$ under the mapping
    \begin{equation}
        (v^l)_{l\in \N} \mapsto R^N
    \end{equation}
    defines a probability measure on the space $\mathcal{C}$ of $\N \times \N$ arrays with entries in $[-1,1]$, endowed with the product topology. Since $\mathcal{C}$ is a compact metric space (by Tychonoff's theorem), the space $\mathcal{M}_1(\mathcal{C})$ of probability measures on $\mathcal{C}$ is compact under the weak topology.
    Therefore, there exists a subsequence $\{N_m\}_{m\in \N}$ under which the above push-forward measures converge weakly to the distribution of some infinite array $R \circeq (R_{l,l'})_{l,l'\in \N}$ in the sense of convergence of all their finite dimensional marginals.
    In particular, the three properties of $R^N$ mentioned above are preserved by the limit, meaning that $R$ is also symmetric, non-negative definite and weakly exchangeable.

    By the representation theorem of \cite{MR666087} (see also the proof in \cite{MR2679002}) and the atoms in Theorem \ref{thm:two.overlap.distribution.limit}, we can assume that the limiting array $R$ is a random element of some probability space with measure $P$ (and corresponding expectation $E$), generated by
    \begin{equation}
        \big(R_{l,l'}\big)_{l,l'\in \N} = \big((\rho^l,\rho^{l'})_{\mathcal{H}} + (1 - x^{l_{\beta}-1}) \bb{1}_{\{l=l'\}}\big)_{l,l'\in \N},
    \end{equation}
    where $(\rho^l)_{l\in \N}$ is an i.i.d. sample of {\it replicas} from some random measure $\mu_{\beta}$ concentrated a.s. on the sphere of radius $\sqrt{x^{l_{\beta}-1}}$ of a separable Hilbert space $\mathcal{H}$.

    By construction, there exists a subsequence $\{N_m\}_{m\in \N}$ such that for any $s\in \N$ and any continuous function $h : [-1,1]^{s(s-1)/2} \rightarrow \R$,
    \begin{equation}\label{eq:subsequential.limit.mu}
        \lim_{m\rightarrow \infty} \mathbb{E}\mathcal{G}_{\beta,N_m}^{\times \infty}\big[h((R_{l,l'}^{N_m})_{1\leq l,l' \leq s})\big] = E \mu_{\beta}^{\times \infty}\big[h((R_{l,l'})_{1\leq l,l' \leq s})\big].
    \end{equation}
    In particular, from Theorem \ref{thm:two.overlap.distribution.limit}, we have
    \begin{equation}\label{eq:limiting.overlap.distribution.mu}
        E\mu_{\beta}^{\times 2}\big[\bb{1}_{\{R_{1,2} \leq r\}}\big] =
        \left\{\hspace{-1mm}
            \begin{array}{ll}
                0, &\mbox{if } r < 0, \\
                \beta_c(\bar{\sigma}_j) / \beta, &\mbox{if } r\in [x^{j-1},x^j), ~1 \leq j \leq l_{\beta}-1, \\
                1, &\mbox{if } r \geq x^{l_{\beta}-1},
            \end{array}
            \right.
    \end{equation}
    where $\beta_c(\bar{\sigma}_j) \circeq 2 / \bar{\sigma}_j$ and $x^j \circeq \bar{\mathcal{J}}_{\sigma^2}(\lambda^j)$.

    Next, we show the consequence of taking the limit \eqref{eq:subsequential.limit.mu} in the statement of Theorem \ref{thm:IGFF.ghirlanda.guerra}.
    Note that a bounded function $h : \{x^0,x^1,...,x^{l_{\beta}-1}\}^{s(s-1)/2} \rightarrow \R$ can always be embedded in a continuous function defined on $[-1,1]^{s(s-1)/2}$.

    \begin{theorem}[Extended Ghirlanda-Guerra identities in the limit]\label{thm:IGFF.ghirlanda.guerra.limit}
        Let $\beta\in \mathcal{B}$ and let $\mu_{\beta}$ be a subsequential limit of $\{\mathcal{G}_{\beta,N}\}_{N\in \N}$ in the sense of \eqref{eq:subsequential.limit.mu}.
        Then, for any $s\in \N$, any $k\in \{1,...,s\}$, and any functions $h : \{x^0,x^1,...,x^{l_{\beta}-1}\}^{s(s-1)/2} \rightarrow \R$ and $g : \{x^0,x^1,...,x^{l_{\beta}-1}\} \rightarrow \R$, we have
        \begin{equation}\label{eq:thm:IGFF.ghirlanda.guerra.limit}
        \begin{aligned}
            &E \mu_{\beta}^{\times (s+1)}\big[g(R_{k,s+1})\, h((R_{i,i'})_{1 \leq i,i' \leq s})\big] \\
            &\hspace{5mm}= \frac{1}{s} E \mu_{\beta}^{\times 2}\big[g(R_{1,2})\big] E \mu_{\beta}^{\times s}\big[h((R_{i,i'})_{1 \leq i,i' \leq s})\big] \\
            &\hspace{10mm}\quad+ \frac{1}{s} \sum_{l \neq k}^s E \mu_{\beta}^{\times s}\big[g(R_{k,l})\, h((R_{i,i'})_{1 \leq i,i' \leq s})\big].
        \end{aligned}
        \end{equation}
    \end{theorem}

    \begin{proof}[Proof of Theorem \ref{thm:IGFF.ghirlanda.guerra.limit}]
        Let $\alpha < \alpha'$ be a pair of scales such that
        \begin{equation}\label{eq:condition.between.discontinuity.points}
            \lambda^{j-1} < \alpha < \alpha' < \lambda^{j + (m-j)\bb{1}_{\{j = l_{\hspace{-0.3mm}\beta}\hspace{-0.3mm}\}}}
        \end{equation}
        for some $j\in \{1,2,...,l_{\beta}\}$.
        From \eqref{eq:subsequential.limit.mu} and from Theorem \ref{thm:IGFF.ghirlanda.guerra} (in the particular case where $h$ is a function of the overlaps), we deduce
        \begin{equation}\label{eq:thm:IGFF.ghirlanda.guerra.limit.first}
            \begin{array}{l}
                \vspace{2mm}E\mu_{\beta}^{\times (s+1)}\big[\int_{(\bar{\mathcal{J}}_{\sigma^2}(\alpha), \bar{\mathcal{J}}_{\sigma^2}(\alpha')]} \hspace{-1mm} \bb{1}_{\{r < R_{k,s+1}\}} dr ~h((R_{i,i'})_{1 \leq i,i' \leq s})\big] \\
                \vspace{2mm}\quad=\frac{1}{s} E\mu_{\beta}^{\times 2}\big[\int_{(\bar{\mathcal{J}}_{\sigma^2}(\alpha), \bar{\mathcal{J}}_{\sigma^2}(\alpha')]} \hspace{-1mm}\bb{1}_{\{r < R_{1,2}\}} dr\big] E\mu_{\beta}^{\times s}\big[h((R_{i,i'})_{1 \leq i,i' \leq s})\big] \\
                \quad\quad+ \frac{1}{s} \sum_{l \neq k}^s E\mu_{\beta}^{\times s}\big[\int_{(\bar{\mathcal{J}}_{\sigma^2}(\alpha), \bar{\mathcal{J}}_{\sigma^2}(\alpha')]} \hspace{-1mm}\bb{1}_{\{r < R_{k,l}\}} dr ~h((R_{i,i'})_{1 \leq i,i' \leq s})\big].
            \end{array}
        \end{equation}
        From \eqref{eq:limiting.overlap.distribution.mu}, we know that $\bb{1}_{\{r <  R_{i,i'}\}}$ is $E\mu_{\beta}^{\times 2}$-a.s. constant in $r$ on the interval $[x^{j-1},x^{j + (m-j)\bb{1}_{\{j = l_{\hspace{-0.3mm}\beta}\hspace{-0.3mm}\}}})$. Therefore, from \eqref{eq:condition.between.discontinuity.points} and \eqref{eq:thm:IGFF.ghirlanda.guerra.limit.first}, we get
        \begin{equation}\label{eq:thm:IGFF.ghirlanda.guerra.limit.finish}
            \begin{array}{l}
                \vspace{2mm}E\mu_{\beta}^{\times (s+1)}\big[\bb{1}_{\{x^{j-1} < R_{k,s+1}\}} ~h((R_{i,i'})_{1 \leq i,i' \leq s})\big] \\
                \vspace{2mm}\quad=\frac{1}{s} E\mu_{\beta}^{\times 2}\big[\bb{1}_{\{x^{j-1} < R_{1,2}\}}\big] E\mu_{\beta}^{\times s}\big[h((R_{i,i'})_{1 \leq i,i' \leq s})\big] \\
                \quad\quad+ \frac{1}{s} \sum_{l \neq k}^s E\mu_{\beta}^{\times s}\big[\bb{1}_{\{x^{j-1} < R_{k,l}\}} ~h((R_{i,i'})_{1 \leq i,i' \leq s})\big].
            \end{array}
        \end{equation}
        Since $R_{i,i'} \geq 0$ $E\mu_{\beta}^{\times 2}$-a.s. by \eqref{eq:limiting.overlap.distribution.mu}, the last equation is also trivially satisfied with say $x^{-1} \circeq -1$.
        But, any function $g : \{x^0,x^1,...,x^{l_{\beta}-1}\} \rightarrow \R$ can be written as a linear combination of the indicator functions $\bb{1}_{\{x^{j-1} < \cdot\}}, ~j\in \{0\} \cup \{1,2,...,l_{\beta}\}$, so we get the conclusion (by the linearity of \eqref{eq:thm:IGFF.ghirlanda.guerra.limit.finish}).
    \end{proof}

    \begin{remark}
        The extended Ghirlanda-Guerra identities are still the subject of ongoing research in spin glass theory and the study of log-correlated random fields, so it is still not clear why these identities seem to be a property shared in the limit by such a vast collection of models.
        Perhaps even more universal could be the {\it stochastic stability} property of {\it random overlap structures} (ROSt's), which are defined and treated (for example) in \cite{MR1657840,MR2219862,MR2537550,MR2678900,MR3055263}. For instance, it is conjectured in \cite{MR3055263} that the laws of the ROSt's satisfying the Ghirlanda-Guerra identities correspond to the extremes of the convex set of laws of the stochastically stable ROSt's.
        It is also conjectured that the stochastic stability of a ROSt for a subsequence of {\it$p$-th power cavity fields} implies ultrametricity.
        In this sense, it is expected that the stochastic stability property is more universal then the Ghirlanda-Guerra identities but still implies ultrametricity under technical conditions.
        In \cite{MR2945621}, it is shown how the Aizenman-Contucci stochastic stability property can be combined with a specific form of the Ghirlanda-Guerra identities into a unified stability property analogous to the Bolthausen-Sznitman invariance property in the setting of Ruelle probability cascades (see \cite{MR1652734}).
    \end{remark}

    \begin{remark}
        It is expected that the results of \cite{MR3541850} on the first order asymptotics of the maximum and $\gamma$-high points can be extended to the more general case where the variance function $\sigma$ in \eqref{eq:variance.function} is piecewise $C^1$. Therefore, it is also expected that the results in the present article could be generalised just like \cite{MR2070335} did when they generalized the results of the GREM to the CREM (the GREM with a continuum of hierarchies).

        We could take this further by imposing $\sigma$ to be piecewise $C^1$ and by working directly with the continuous version of the two-dimensional GFF instead of the discrete version. A formal definition of such a field is given in Section 1.3 of \cite{MR3541850} as well as a conjecture on the Hausdorff dimension of the $\gamma$-thick points (the analogue of the $\gamma$-high points). The field is a random distribution (i.e. generalized function), so it cannot be defined pointwise, but we can make sense of the collection of circle averages around a point $v\in [0,1]^2$ as a stochastic process. In fact, we would expect such a process, after a time-change, to be equal in law to $\int_0^{\cdot} \sigma(s) dB_v(s)$, where $B_v$ is a Brownian motion adapted to a certain filtration $\mathbb{F}_v$. We could then ask if it is possible to characterize the limiting (with respect to the approximation procedure) law of the Liouville measure (the analogue of the Gibbs measure) of this new field. For an introduction to these concepts, see e.g. \cite{Berestycki2015ln,MR3274356,MR2322706}.

        Finally, another natural question is to ask if there is a way to introduce randomness in the function $s\mapsto \sigma(s)$ and still make sense of the questions above, although this is not clear since the process $(\sigma(s))_{s\in [0,1]}$ cannot be adapted (let alone predictable) simultaneously to all the filtrations $\mathbb{F}_v, ~v\in [0,1]^2$. Maybe there is a way around this problem if the filtrations share ``information'' in a very structured way.
    \end{remark}

\section{Consequences of Theorem \ref{thm:two.overlap.distribution.limit} and Theorem \ref{thm:IGFF.ghirlanda.guerra.limit}}\label{sec:consequences.GG.limit}

    The first consequence concerns the geometry of the overlaps in the limit.
    It was shown in \cite{MR2599202} (see also \cite{MR2825947} for a simplified proof) that any limiting array of overlaps that takes finitely many values and satisfy \eqref{eq:thm:IGFF.ghirlanda.guerra.limit} must be ultrametric under $E \mu_{\beta}^{\times \infty}$.

    \begin{corollary}[Ultrametricity in the limit]\label{cor:ultrametricity}
        Let $\beta\in \mathcal{B}$ and let $\mu_{\beta}$ be a subsequential limit of $\{\mathcal{G}_{\beta,N}\}_{N\in \N}$ in the sense of \eqref{eq:subsequential.limit.mu}.
        We must have
        \begin{equation}\label{eq:ultrametricity}
            E\mu_{\beta}^{\times 3}\big(R_{1,2} \geq R_{1,3} \wedge R_{2,3}\big) = 1.
        \end{equation}
        Since the {\it replicas} $\rho^l$ all have norm $\sqrt{x^{l_{\beta}-1}}$ almost-surely in $\mathcal{H}$, then \eqref{eq:ultrametricity} is equivalent to the ultrametric inequality
        \begin{equation}
            E\mu_{\beta}^{\times 3}\big(\|\rho^1 - \rho^2\| \leq \|\rho^1 - \rho^3\| \vee \|\rho^2 - \rho^3\|\big) = 1.
        \end{equation}
    \end{corollary}

    \begin{remark}
        The random measure $\mu_{\beta}$ gives more weight to extreme values, so we can interpret this corollary as saying that the extremes are clustered in a hierarchical way.
        From \eqref{eq:limiting.overlap.distribution.mu}, the number of hierarchies increases as the {\it inverse temperature} $\beta$ becomes larger than some of the {\it critical thresholds} $\beta_c(\bar{\sigma}_j)$.
    \end{remark}

    The second consequence makes the description of the structure of $\mu_{\beta}$ even more precise.
    Probability cascades were introduced in \cite{MR875300} to describe the limiting Gibbs measure of the GREM.
    Since the $\para$-GFF satisfies the extended Ghirlanda-Guerra identities and the limiting two-overlap distribution takes finitely many values, we can show that the limiting Gibbs measure $\mu_{\beta}$ is a Ruelle probability cascade.
    First, we define probability cascades by following \cite{MR3052333}.

    For a given $r \geq 1$, let
    \vspace{-1mm}
    \begin{equation}
        \mathcal{T} \circeq \{\emptyset\} \cup \N \cup \N^2 \cup ... \cup \N^r
    \end{equation}
    be the vertex set of a tree rooted at $\emptyset$. Each vertex $v = (n_1,...,n_p)\in \N^p$, for $p \leq r - 1$, has children
    \vspace{-1mm}
    \begin{equation}
        v n \circeq (n_1,...,n_p,n)\in \N^{p+1}, \quad n\in \N.
    \end{equation}
    Each vertex $v\in \N^p$ is connected to the root by a path. Denote by $p(v)$ the set of vertices (excluding the root) on the shortest path from $v$ to the root.
    Additionally, fix two sequences of parameters :
    \vspace{1mm}
    \begin{equation}\label{eq:RPC.parameters}
    \begin{aligned}
        &0 \circeq \zeta_{-1} < \zeta_0 < \zeta_1 < ... < \zeta_{r-1} < \zeta_r \circeq 1, \\
        &0 \circeq q_0 < q_1 < ... < q_{r-1} < q_r \leq 1.
    \end{aligned}
    \end{equation}
    For all $v\in \mathcal{T}$, denote by $|v|$ its distance in the tree, namely $|v| \circeq \#p(v)$.
    Then, for all $v\in \mathcal{T} \backslash \N^r$, generate independent Poisson point processes, denoted by $\Pi_v$, with mean measure $\zeta_{|v|} x^{-1 - \zeta_{|v|}} dx$ on $(0,\infty)$.
    We arrange the points in $\Pi_v$ in decreasing order :
    \begin{equation}\label{eq:PPP.decreasing.order}
        z_{v 1} > z_{v 2} > ... > z_{v n} > ...
    \end{equation}
    For each vertex $v\in \mathcal{T}\backslash \N^r$, the relative weight of each point in $\Pi_v$ is defined by
    \begin{equation}\label{eq:PPP.relative.weights}
        w_{v n} \circeq \frac{z_{v n}}{\sum_{i\in \N} z_{v i}}, \quad n\in \N.
    \end{equation}
    Say we are on a separable Hilbert space $\mathscr{H}$ with orthonormal basis $\{e_v\}_{v\in \mathcal{T} \backslash \{\emptyset\}}$.
    Consider the vectors in $\mathscr{H}$
    \vspace{1mm}
    \begin{equation}\label{eq:elements.to.sample.from.in.H}
        h_v \circeq \hspace{-1.5mm}\sum_{u\in p(v)} \hspace{-1mm}e_u\, (q_{|u|} - q_{|u|-1})^{1/2}, \quad v\in \mathcal{T} \backslash \{\emptyset\},
    \end{equation}
    and define a random measure on them by
    \begin{equation}\label{eq:RPC.definition}
        G(h_v) \circeq \prod_{u\in p(v)} w_u, \quad v\in \N^r.
    \end{equation}
    The random measure $G$ is called a Ruelle probability cascade (RPC) associated with the parameters in \eqref{eq:RPC.parameters}.
    It is defined up to an orthonormal change of basis.

    We can think of $\N^r$ as the leaves in the tree structure.
    From \eqref{eq:elements.to.sample.from.in.H}, each element in $\{h_v\}_{v\in \N^r}$ has norm $\sqrt{q_r}$ and the scalar product between two such elements can only take values in the finite set
    \begin{equation}
        \{0,q_1,q_2,...,q_r\}.
    \end{equation}
    The weights \eqref{eq:PPP.relative.weights}, associated with each branch in the tree, are random.
    Hence, \eqref{eq:RPC.definition} defines a random probability measure and each instance of $G$ samples elements in $\{h_v\}_{v\in \N^r}$ by choosing a branch independently at each step, between the root and a leaf, with probability given by the associated weight in \eqref{eq:PPP.relative.weights}. The tree structure is illustrated in Figure \ref{fig:RPC.tree.structure}. Since the weights are ordered in decreasing order at each scale, the branches on the left are more likely to be selected at each step.

    \newpage
    \begin{figure}[ht]
        \centering
        \includegraphics[height=11cm,width=12.5cm]{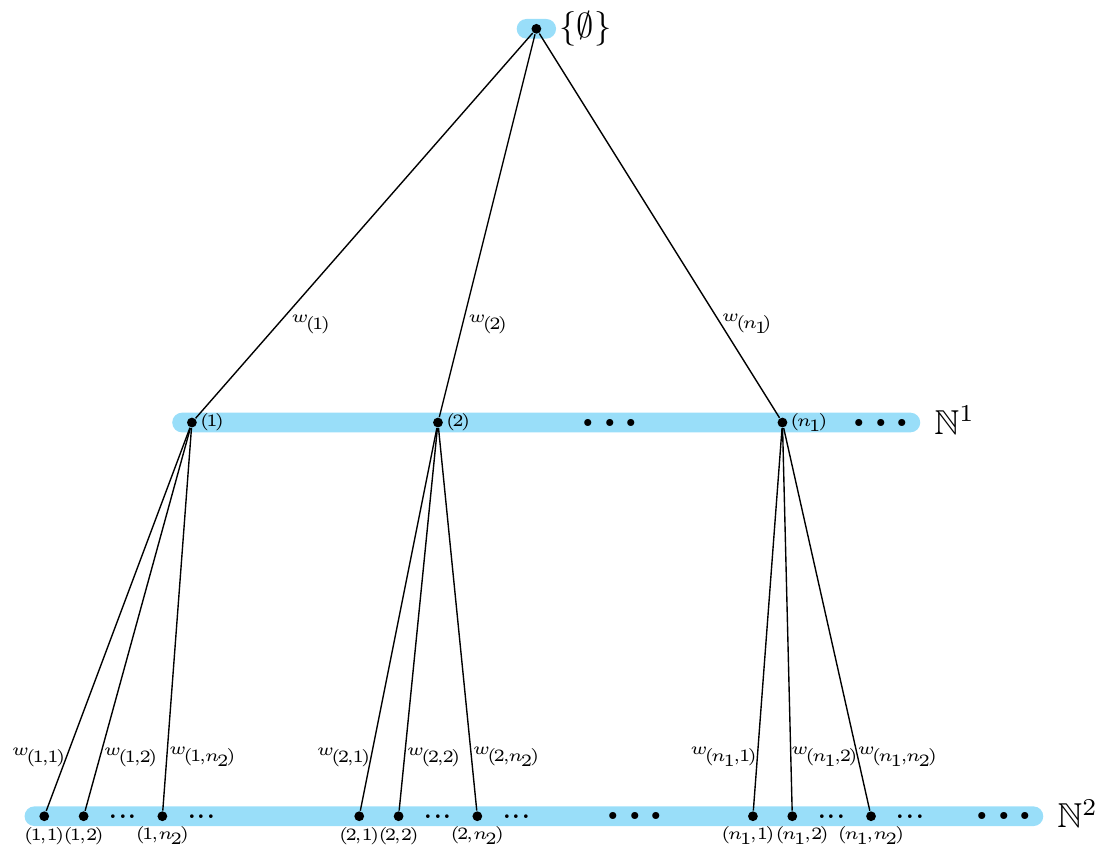}
        \captionsetup{width=0.8\textwidth}
        \caption{Exact tree structure of a Ruelle probability cascade with $r = 2$ levels. Given an instance $\omega\in \Omega$ of the random weights, the measure $G$ samples elements in $\{h_v\}_{v\in \N^r}$ with probability equal to the product of the probabilities associated with the branches on the shortest path from the root to the leaf $v$. For example, $(G(\omega))(h_{(1,2)}) = w_{(1)}(\omega) w_{(1,2)}(\omega)$. For the scalar products, we have, for example, $(h_{(1,1)},h_{(2,2)})_{\mathscr{H}} = 0$ and $(h_{(2,1)},h_{(2,2)})_{\mathscr{H}} = q_1$.}
        \label{fig:RPC.tree.structure}
    \end{figure}

    The corollary below shows that the limiting Gibbs measure of the $\para$-GFF is a RPC, despite its underlying tree structure being only approximate for finite $N$.

    \begin{corollary}\label{cor:ruelle.probability.cascade}
        Let $\beta\in \mathcal{B}$ and let $\mu_{\beta}$ be a subsequential limit of $\{\mathcal{G}_{\beta,N}\}_{N\in \N}$ in the sense of \eqref{eq:subsequential.limit.mu}. Then, $\mu_{\beta}$ has the same law as a RPC with parameters
        \begin{itemize}
            \item[$\bullet$] \vspace{0mm}$r = l_{\beta} - 1$,
            \item[$\bullet$] \vspace{1mm}$\zeta_j = E\mu_{\beta}^{\times 2}\big[\bb{1}_{\{R_{1,2} \leq x^j\}}\big] = (2 / \bar{\sigma}_{j+1}) / \beta$, \quad for all $j\in \{0,1,...,r-1\}$,
            \item[$\bullet$] \vspace{1mm}$q_j = x^j \circeq \bar{\mathcal{J}}_{\sigma^2}(\lambda^j)$, \quad for all $j\in \{0,1,...,r\}$.
        \end{itemize}
    \end{corollary}

    \begin{proof}
        The proof follows directly from Theorem 2.13 in \cite{MR3052333} or from the proof of Theorem 1.13 in \cite{MR2070334} (once we have the ultrametricity).
        We simply need to match the parameters so that $\{q_j\}_{j=0}^r$ are the atoms of $E\mu_{\beta}^{\times 2}(R_{1,2}\in \cdot)$ and $\{\nabla \zeta_j\}_{j=0}^r$ are the corresponding probabilities.
    \end{proof}

\newpage
\section{Proofs of the main results}\label{sec:proof.main.results}

    Throughout the proofs, $c, C, \widetilde{C}, etc.,$ will denote positive constants whose value can change from line to line and can only depend on the parameters $(\boldsymbol{\sigma},\boldsymbol{\lambda})$, unless additional variables are specified. Equations are always implicitly stated to hold for $N$ large enough when needed.

    \subsection{Computation of the limiting free energy}\label{sec:free.energy}

    Theorem \ref{thm:IGFF.free.energy} is a direct consequence of Lemma \ref{lem:IGFF.free.energy.prob}, which shows $f_N^{\psi}(\beta) \rightarrow f^{\psi}(\beta)$ in probability, and Lemma \ref{lem:IGFF.free.energy.uniform.integrability}, which shows the uniform integrability of the sequence $\{|f_N^{\psi}(\beta)|^p\}_{N\in \N}$ for all $p\in [1,\infty)$.

    \begin{lemma}[Convergence in probability of the free energy]\label{lem:IGFF.free.energy.prob}
        Let $\eta > 0$ and $\beta > 0$. There exists $c = c(\eta,\beta,\boldsymbol{\sigma},\boldsymbol{\lambda}) > 0$ such that for $N$ large enough,
        \begin{equation}
            \prob{}{|f_N^{\psi}(\beta) - f^{\psi}(\beta)| > \eta} \leq N^{-c}.
        \end{equation}
    \end{lemma}

    \begin{proof}
        Fix $\eta > 0$ and $\beta > 0$. For all $i\in \{0,1,...,K+1\}$, define $\gamma_i \circeq i \gamma^{\star} / K$.
        We will choose $K\in \N$ large enough later.
        We prove the upper bound first.
        Define
        \begin{equation}
            \HH_N^{\text{abs}}(\gamma) \circeq \{v\in V_N : |\psi_v| \geq \gamma \log N^2\}.
        \end{equation}
        From \eqref{eq:IGFF.order1.upper.bound}, \eqref{eq:IGFF.high.points.upper.bound}, and the symmetry of Gaussian densities, the event
        \begin{equation}
            U_{N,K,\eta} \circeq \bigcap_{i=0}^K \left\{|\HH_N^{\text{abs}}(\gamma_i)| < N^{2 \mathcal{E}(\gamma_i) + \eta}\right\} \bigcap \left\{\max_{v\in V_N} |\psi_v| < \frac{K+1}{K} \gamma^{\star} \log N^2\right\}
        \end{equation}
        satisfies $\prob{}{U_{N,K,\eta}^c} \leq N^{-c(K,\eta,\bb{\sigma},\bb{\lambda})}$ for any given $K$.
        On the event $U_{N,K,\eta}$,
        \vspace{-2mm}
        \begin{align}\label{eq:free.energy.eq.Z.upper}
            Z_N^{\psi}(\beta)
            &\circeq \sum_{v\in V_N} e^{\beta \psi_v}
            \leq \sum_{i=1}^{K+1} (|\HH_N^{\text{abs}}(\gamma_{i-1})| - |\HH_N^{\text{abs}}(\gamma_i)|) N^{2 \beta \gamma_i} \notag \\
            &= \sum_{i=1}^K (N^{2 \beta \gamma_{i+1}} - N^{2 \beta \gamma_i}) |\HH_N^{\text{abs}}(\gamma_i)| + N^{2 \beta \gamma_1} |\HH_N^{\text{abs}}(\gamma_0)| \notag \\
            &\leq N^{2 \beta \gamma^{\star} \hspace{-0.7mm}/\hspace{-0.2mm} K} \sum_{i=0}^K N^{2 \beta \gamma_i} |\HH_N^{\text{abs}}(\gamma_i)|.
        \end{align}
        We used the fact that $|\HH_N^{\text{abs}}(\gamma_{K+1})| = 0$ to obtain the second equality.
        Therefore, on $U_{N,K,\eta}$,
        \begin{align}
            f_N^{\psi}(\beta)
            &\stackrel{\eqref{eq:free.energy.eq.Z.upper}}{\leq} \frac{\beta \gamma^{\star}}{K} + \frac{\log(K+1)}{\log N^2} + \max_{0 \leq i \leq K} (\beta \gamma_i + \mathcal{E}(\gamma_i)) + \frac{\eta}{2} \notag \\
            &\stackrel{\phantom{\eqref{eq:free.energy.eq.Z.upper}}}{\leq} \max_{0 \leq i \leq K} (\beta \gamma_i + \mathcal{E}(\gamma_i)) + \eta \hspace{7mm}
                \begin{array}{l}
                    \text{for } K ~\text{large enough with respect} \\
                    \text{to $\eta$ and $\beta$, and $N$ large enough} \\
                    \text{with respect to $K$ and $\eta$},
                \end{array} \notag \\
            &\stackrel{\phantom{\eqref{eq:free.energy.eq.Z.upper}}}{\leq} \max_{\gamma\in [0,\gamma^{\star}]} (\beta \gamma + \mathcal{E}(\gamma)) + \eta \notag \\
            &\stackrel{\phantom{\eqref{eq:free.energy.eq.Z.upper}}}{=} f^{\psi}(\beta) + \eta \hspace{27.2mm}
                \begin{array}{l}
                    \text{by Lemma } \ref{lem:tech.lemma.2}.
                \end{array}
        \end{align}
        Thus, for $K$ large enough (fixed, depending on $\eta$ and $\beta$), we have
        \begin{equation}\label{lem:IGFF.free.energy.prob.upper.bound}
            \prob{}{f_N^{\psi}(\beta) > f^{\psi}(\beta) + \eta} \leq \prob{}{U_{N,K,\eta}^c} \leq N^{-c(K,\eta,\bb{\sigma},\bb{\lambda})}.
        \end{equation}

        We now prove the lower bound.
        Recall that
        \begin{equation}
            \HH_N(\gamma) \circeq \{v\in V_N : \psi_v \geq \gamma \log N^2\}.
        \end{equation}
        From \eqref{eq:IGFF.order1.upper.bound} and \eqref{eq:IGFF.high.points.lower.bound}, the event
        \begin{equation}
            B_{N,K,\eta} \circeq \bigcap_{i=1}^{K-1} \left\{|\HH_N(\gamma_i)| \geq N^{2 \mathcal{E}(\gamma_i) - \eta}\right\} \bigcap \left\{\max_{v\in V_N} \psi_v < \frac{K+1}{K} \gamma^{\star} \log N^2\right\}
        \end{equation}
        satisfies $\prob{}{B_{N,K,\eta}^c} \leq N^{-c(K,\eta,\bb{\sigma},\bb{\lambda})}$ for any given $K$.
        On the event $B_{N,K,\eta}$,
        \vspace{-2mm}
        \begin{align}\label{eq:free.energy.eq.Z.lower}
            Z_N^{\psi}(\beta)
            &\circeq \sum_{v\in V_N} e^{\beta \psi_v}
            \geq \sum_{i=1}^{K+1} (|\HH_N(\gamma_{i-1})| - |\HH_N(\gamma_i)|) N^{2 \beta \gamma_{i-1}} \notag \\
            &= \sum_{i=1}^K (N^{2 \beta \gamma_i} - N^{2 \beta \gamma_{i-1}}) |\HH_N(\gamma_i)| + N^{2 \beta \gamma_0} |\HH_N(\gamma_0)| \notag \\
            &\geq \frac{1}{2} \sum_{i=1}^{K-1} N^{2 \beta \gamma_i} |\HH_N(\gamma_i)|.
        \end{align}
        We used the fact that $|\mathcal{H}_N(\gamma_{K+1})| = 0$ to obtain the second equality.
        We dropped the $0$-th and $K$-th summands to obtain the last inequality and took $N$ large enough that $1 - N^{-2\beta \gamma^{\star}\hspace{-0.7mm} / \hspace{-0.2mm}K} \geq 1/2$.
        Therefore, on $B_{N,K,\eta}$,
        \begin{align}
            f_N^{\psi}(\beta)
            &\stackrel{\eqref{eq:free.energy.eq.Z.lower}}{\geq} \max_{1 \leq i \leq K-1} (\beta \gamma_i + \mathcal{E}(\gamma_i)) - \frac{\eta}{2} - \frac{\log 2}{\log N^2} \notag \\
            &\stackrel{\phantom{\eqref{eq:free.energy.eq.Z.lower}}}{\geq} \max_{1 \leq i \leq K-1} (\beta \gamma_i + \mathcal{E}(\gamma_i)) - \frac{3\eta}{4} \hspace{7mm}
                \begin{array}{l}
                    \text{for } N ~\text{large enough} \\
                    \text{with respect to $\eta$},
                \end{array} \notag \\
            &\stackrel{\phantom{\eqref{eq:free.energy.eq.Z.lower}}}{\geq} \max_{\gamma\in [0,\gamma^{\star}]} (\beta \gamma + \mathcal{E}(\gamma)) - \eta \hspace{13.9mm}
                \begin{array}{l}
                    \text{for } K ~\text{large enough with respect} \\
                    \text{to $\eta$ and $\beta$ since $\gamma\mapsto (\beta\gamma + \mathcal{E}(\gamma))$} \\
                    \text{is continuous by Lemma \ref{lem:tech.lemma.1}},
                \end{array} \notag \\
            &\stackrel{\phantom{\eqref{eq:free.energy.eq.Z.lower}}}{=} f^{\psi}(\beta) - \eta \hspace{33.3mm}
                \begin{array}{l}
                    \text{by Lemma } \ref{lem:tech.lemma.2}.
                \end{array}
        \end{align}
        Thus, for $K$ large enough (fixed, depending on $\eta$ and $\beta$), we have
        \begin{equation}\label{lem:IGFF.free.energy.prob.lower.bound}
            \prob{}{f_N^{\psi}(\beta) < f^{\psi}(\beta) - \eta} \leq \prob{}{B_{N,K,\eta}^c} \leq N^{-c(K,\eta,\bb{\sigma},\bb{\lambda})}.
        \end{equation}
        Equations \eqref{lem:IGFF.free.energy.prob.upper.bound} and \eqref{lem:IGFF.free.energy.prob.lower.bound} together prove the lemma.
    \end{proof}

    For the uniform integrability, we follow the proof of \cite{MR883541}, originally given in the context of the GREM.

    \begin{lemma}[Uniform integrability of $\{|f_N^{\psi}(\beta)|^p\}_{N\in \N}$]\label{lem:IGFF.free.energy.uniform.integrability}
        Let $\beta > 0$ and $1 \leq p < \infty$. Then,
        \begin{equation}
            \lim_{\alpha \rightarrow \infty} \sup_{N\in \N} \esp{}{|f_N^{\psi}(\beta)|^p ~\bb{1}_{\{|f_N^{\psi}(\beta)|^p > \alpha\}}} = 0.
        \end{equation}
    \end{lemma}

    \begin{proof}
        By definition, $f_N^{\psi}(\beta) \circeq \frac{1}{\log N^2} \log \sum_{v\in V_N} e^{\beta \psi_v}$.
        Bound from above every summand by the maximum summand and bound from below by keeping only the maximum summand. If $\xi_N \circeq \max_{v\in V_N} \psi_v / (\log N^2)$, it is easily seen that for $N \geq 2$,
        \begin{equation}\label{eq:cappocaccia.boun}
            \beta \xi_N \leq f_N^{\psi}(\beta) \leq \beta \xi_N + \frac{\log (N+1)^2}{\log N^2} \leq \beta \xi_N + 2.
        \end{equation}
        Assume that $\alpha^{1/p} - 2 > 0$. By splitting the event $\{|f_N^{\psi}(\beta)| > \alpha^{1/p}\}$ in two parts : $\{f_N^{\psi}(\beta) > \alpha^{1/p}\}$ and $\{-f_N^{\psi}(\beta) > \alpha^{1/p}\}$, and then using \eqref{eq:cappocaccia.boun}, we deduce
        \begin{align}\label{eq:lem:IGFF.free.energy.uniform.integrability.sum}
            &\esp{}{|f_N^{\psi}(\beta)|^p ~\bb{1}_{\{|f_N^{\psi}(\beta)| > \alpha^{1/p}\}}} \notag \\
            &\quad\leq \mathbb{E}\Big[(\beta\xi_N + 2)^p ~\bb{1}_{\{\beta\xi_N + 2 > \alpha^{1/p}\}}\Big] + \mathbb{E}\Big[(-\beta\xi_N)^p ~\bb{1}_{\{-\beta\xi_N > \alpha^{1/p}\}}\Big] \notag \\
            &\quad= \sum_{l=1}^{\infty} \mathbb{E}\Big[(\beta\xi_N + 2)^p ~\bb{1}_{\{(l+1) \alpha^{1/p} \geq \beta\xi_N + 2 > l \alpha^{1/p}\}}\Big] \notag \\
            &\quad\quad+ \sum_{l=1}^{\infty} \mathbb{E}\Big[(-\beta\xi_N)^p ~\bb{1}_{\{(l+1) \alpha^{1/p} \geq -\beta\xi_N > l \alpha^{1/p}\}}\Big] \notag \\
            &\quad\leq 2 \sum_{l=1}^{\infty} (l+1)^p\, \alpha ~\mathbb{P}\Big(|\xi_N| > \frac{1}{\beta}(l \alpha^{1/p} - 2)\Big).
        \end{align}
        Note that $|\xi_N| \leq \max_{v\in V_N} |\psi_v| / (\log N^2)$, and $\max_{v\in V_N} \var{}{\psi_v} \leq \mathcal{J}_{\sigma^2}(1) \log N + C_0$ by Lemma \ref{lem:IGFF.variance.uniform.bound}. Therefore, for all $l\in \N$, a union bound and a standard Gaussian tail estimate yield (when $N$ is large enough, say $N \geq N_0 \geq 2$)
        \begin{align}\label{eq:lem:IGFF.free.energy.uniform.integrability.gaussian.estimate}
            \mathbb{P}\Big(|\xi_N| > \frac{1}{\beta} (l \alpha^{1/p} - 2)\Big)
            &\leq (N+1)^2 \max_{v\in V_N} 2\, \mathbb{P}\Big(\psi_v > \frac{1}{\beta} (l \alpha^{1/p} - 2) \log N^2\Big) \notag \\
            &\leq (N+1)^2 N^{- 2 \frac{(l \alpha^{1/p} - 2)^2}{\beta^2 \mathcal{J}_{\sigma^2}(1)}} \notag \\
            &\leq (N+1)^2 N^{-2 \frac{(\alpha^{1/p} - 2)^2}{\beta^2 \mathcal{J}_{\sigma^2}(1)}} N^{- 2 \frac{(l - 1)^2 \alpha^{2/p}}{\beta^2 \mathcal{J}_{\sigma^2}(1)}}.
        \end{align}
        To obtain the last inequality, we wrote $(l\alpha^{1/p} - 2)^2 = (\alpha^{1/p} - 2 + (l-1)\alpha^{1/p})^2$ and used $(a+b)^2 \geq a^2 + b^2, ~a,b \geq 0$.
        If we further assume that $(\alpha^{1/p} - 2)^2 > \beta^2 \mathcal{J}_{\sigma^2}(1)$, the sum in \eqref{eq:lem:IGFF.free.energy.uniform.integrability.sum} tends to $0$ as $\alpha \rightarrow \infty$, uniformly for $N \geq N_0$.
    \end{proof}

    \begin{proof}[Proof of Theorem \ref{thm:IGFF.free.energy.A.N.p}]
        Since $x \mapsto \log x$ is an increasing function and $A_{N,\rho} \subseteq V_N$, we have the upper bound on the limit in probability :
        \begin{equation}
            f_{N,\rho}^{\psi}(\beta) \leq \frac{1}{\log N^2} \log \sum_{v\in V_N} e^{\beta \psi_v} \circeq f_N^{\psi}(\beta) \stackrel{N\rightarrow \infty}{\longrightarrow} f^{\psi}(\beta).
        \end{equation}
        On the other hand, from Lemma A.2 in \cite{MR3541850} and the independence of the increments, we know that for any $\delta\in (0,1/2]$ and $j\in \{1,2,...,m\}$, then for $N$ large enough and all $v\in V_N^{\delta} \circeq \big\{v\in V_N : \min_{z\in \partial V_N} \|v - z\|_2 \geq \delta N\big\}$, we have
        \vspace{-1mm}
        \begin{equation}
            -C_1(\delta,\bb{\sigma}) \leq \var{}{\nabla \psi_v(\lambda^j)} - \bar{\sigma}_j^2 \nabla \lambda^j \log N \leq C_2(\bb{\sigma}).\vspace{1mm}
        \end{equation}
        Hence, from the remark at the end of Lemma 3.1 in \cite{MR3541850}, we know that Theorem \ref{thm:IGFF.order1} and Theorem \ref{thm:IGFF.high.points} (in this paper) hold on $V_N^{\delta}$; the proof is in fact easier. Since $A_{N,\rho} \supseteq V_N^{\delta}$ for $N$ large enough, we have
        \begin{equation}
            f_{N,\rho}^{\psi}(\beta) \circeq \frac{1}{\log N^2} \log \sum_{v\in A_{N,\rho}} \hspace{-1mm}e^{\beta \psi_v} \geq \frac{1}{\log N^2} \log \sum_{v\in V_N^{\delta}} e^{\beta \psi_v}.
        \end{equation}
        A rerun of the proof of the lower bound in Lemma \ref{lem:IGFF.free.energy.prob}, with $\mathcal{H}_N(\gamma)$ restricted to $V_N^{\delta}$, yields the conclusion.
    \end{proof}

    \subsection{The Gibbs measure near the boundary}\label{sec:Gibbs.measure.outside.A.N.p}

    The first step in the proof of Theorem \ref{thm:two.overlap.distribution.limit} is to show that the Gibbs measure does not carry any weight near the boundary of $V_N$ in the limit $N\rightarrow \infty$.
    For this purpose, recall
    \begin{equation}\label{eq:A.N.rho}
        A_{N,\rho} \circeq \Big\{v\in V_N : \min_{z\in \Z^2 \backslash V_N} \|v - z\|_2 \geq N^{1 - \rho}\Big\}, \quad \rho\in (0,1].
    \end{equation}
    This box contains the points in $V_N$ that are at least at a distance of $N^{1-\rho}$ from the exterior.
    The Gibbs measure of the $\para$-GFF restricted to $A_{N,\rho}$ is
    \begin{equation}\label{eq:IGFF.gibbs.measure.restricted}
        \mathcal{G}_{\beta,N,\rho}(\{v\}) \circeq \frac{e^{\beta \psi_v}}{Z_{N,\rho}^{\psi}(\beta)}, \quad v\in A_{N,\rho},
    \end{equation}
    where $Z_{N,\rho}^{\psi}(\beta) \circeq \sum_{v\in A_{N,\rho}} e^{\beta \psi_v}$.
    We start by proving an upper bound on the following quantity :
    \begin{equation}\label{eq:restricted.free.energy.definition}
        \widetilde{f}_{N,\rho}^{\psi}(\beta) \circeq \frac{1}{\log N^2} \log \hspace{-1mm}\sum_{v\in A_{N,\rho}^c} \hspace{-1mm}e^{\beta \psi_v}.
    \end{equation}

    \begin{lemma}\label{lem:restricted.free.energy.tech.lemma}
        Let $\eta > 0$, $\beta > 0$ and $\rho\in (0,\lambda_1)$.
        There exists $c = c(\eta,\beta,\rho,\boldsymbol{\sigma},\boldsymbol{\lambda}) > 0$ such that
        \begin{equation}
            \prob{}{\widetilde{f}_{N,\rho}^{\psi}(\beta) > (f^{\psi}(\beta) - \rho/2) + \eta} \leq N^{-c}
        \end{equation}
        for $N$ large enough.
    \end{lemma}

    \begin{proof}
        In \cite{MR3541850}, Theorem 1.2 and Theorem 1.3 prove that
        \begin{equation}\label{eq:lem:restricted.free.energy.tech.lemma.max}
            \lim_{N\rightarrow \infty} \frac{\max_{v\in V_N} \psi_v}{\log N^2} = \gamma^{\star}\hspace{-1mm}, \qquad \text{in probability},
        \end{equation}
        where
        \begin{equation*}
            \begin{array}{l}
                \gamma^{\star} = \max_{\gamma_1,\gamma_2,...,\gamma_M} \sum_{i=1}^M \nabla \gamma_i, \\[2mm]
                \text{under the constraints }
                \sum_{i=1}^k \left(\nabla \lambda_i - \frac{(\nabla \gamma_i)^2}{\sigma_i^2 \nabla \lambda_i}\right) \geq 0, \quad 1 \leq k \leq M,
            \end{array}
        \end{equation*}
        and, for all $0 \leq \gamma < \gamma^{\star}$,
        \begin{equation}\label{eq:lem:restricted.free.energy.tech.lemma.high.points}
            \lim_{N\rightarrow \infty} \frac{\log(|\HH_N(\gamma)|)}{\log N^2} = \mathcal{E}(\gamma), \qquad \text{in probability},
        \end{equation}
        where
        \begin{equation*}
            \begin{array}{l}
                \mathcal{E}(\gamma) = \max_{\gamma_1,\gamma_2,...,\gamma_{M-1}} \sum_{i=1}^{M-1} \left(\nabla \lambda_i - \frac{(\nabla \gamma_i)^2}{\sigma_i^2 \nabla \lambda_i}\right) + \left(\nabla \lambda_M - \frac{(\gamma - \gamma_{M-1})^2}{\sigma_M^2 \nabla \lambda_M}\right), \\[3mm]
                \text{under the constraints }
                \sum_{i=1}^k \left(\nabla \lambda_i - \frac{(\nabla \gamma_i)^2}{\sigma_i^2 \nabla \lambda_i}\right) \geq 0, \quad 1 \leq k \leq M-1.
            \end{array}
        \end{equation*}
        The unique solution of each optimization problem is rigorously found in Appendix A of \cite{Ouimet2014master} by using the Karush-Kuhn-Tucker theorem, and the solutions are shown to coincide with \eqref{eq:lem:restricted.free.energy.tech.lemma.max} and \eqref{eq:lem:restricted.free.energy.tech.lemma.high.points} in \cite{MR3541850}.

        Now, to bound $\widetilde{f}_{N,\rho}^{\psi}(\beta)$ from above, we need to find the analogues of \eqref{eq:lem:restricted.free.energy.tech.lemma.max} and \eqref{eq:lem:restricted.free.energy.tech.lemma.high.points} on $A_{N,\rho}^c$ instead of $V_N$.
        To this end, we recall the {\it set of representatives at scale $\lambda$} from \cite{MR3541850}, denoted by $R_{\lambda}$. Loosely speaking, at a given scale $\lambda$, the points in $R_{\lambda}\subseteq V_N$ represent the $O(N^{2\lambda})$ nods of the underlying branching quaternary tree structure of the GFF, see Figure \ref{fig:branching.structure.GFF}. This branching structure is motivated by the fact that if $v_{\lambda}$ denotes the representative at scale $\lambda > 0$ that is closest to $v$, then, from Lemma A.6 in \cite{MR3541850}, we know that $\max_{v\in V_N} \var{}{\psi_v(\lambda) - \psi_{v_{\lambda}}(\lambda)} \leq C$, for $N$ large enough.

        \begin{figure}[H]
            \includegraphics[width=70mm]{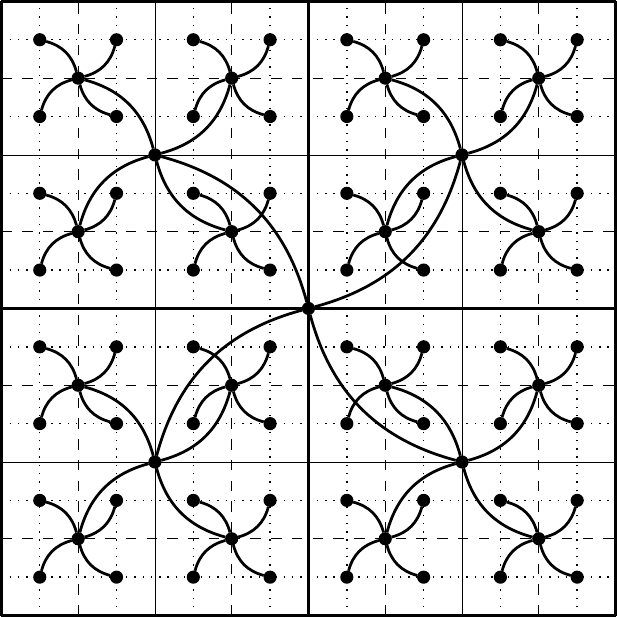}
            \captionsetup{width=0.8\textwidth}
            \caption{The representatives at scale $0, 1/4, 1/2$ and $3/4$.}
            \label{fig:branching.structure.GFF}
        \end{figure}

        More precisely, let $R_1 \circeq V_N$, and for $\lambda\in [0,1)$, the set $R_{\lambda}$ contains $\lfloor N^{\lambda} \rfloor^2$ $v$'s with neighborhoods $[v]_{\lambda}$ that can only touch at their boundary (if they do touch) and are not cut off by $\partial V_N$.
        To remove any ambiguity, define $R_{\lambda}$ in such a way that $\max_{v\in V_N} \min_{z\in R_{\lambda}} \|v - z\|_2$ is minimized.
        For instance, if $N = 2^n$, $\lambda\in [0,1)$ and $\lambda n\in \N_0$, then divide $V_N$ into a grid with $N^{2\lambda}$ squares of side length $N^{1-\lambda}$; the center point of each square is a representative at scale $\lambda$.

        Since we assumed $\rho\in (0,\lambda_1)$, the only difference is that there are $O(N^{2(\lambda_i - \rho/2)})$ representatives at each scale $\lambda_i$ on $A_{N,\rho}^c$ ($\psi$ is still defined on $V_N$) instead of $O(N^{2\lambda_i})$. Therefore, a rerun of the proof of Lemma 3.1 and 3.4 in \cite{MR3541850} (note that only the upper bounds work) shows that for all $\varepsilon > 0$,
        \vspace{-1mm}
        \begin{equation}\label{eq:lem:restricted.free.energy.tech.lemma.max.restricted}
            \mathbb{P}\left(\max_{v\in A_{N,\rho}^c} \psi_v \geq (1 + \varepsilon) \gamma_{\rho}^{\star} \log N^2\right) \leq N^{-c(\varepsilon,\bb{\sigma},\bb{\lambda})},
        \end{equation}
        for $N$ large enough, where
        \begin{equation*}
            \begin{array}{l}
                \gamma_{\rho}^{\star} \circeq \max_{\gamma_1,\gamma_2,...,\gamma_M} \sum_{i=1}^M \nabla \gamma_i, \\[2mm]
                \text{under the constraints }
                \sum_{i=1}^k \left(\nabla \lambda_i - \frac{(\nabla \gamma_i)^2}{\sigma_i^2 \nabla \lambda_i}\right) \geq \rho/2, \quad 1 \leq k \leq M,
            \end{array}
        \end{equation*}
        and, for all $0 \leq \gamma \leq \gamma_{\rho}^{\star}$,
        \begin{equation}\label{eq:lem:restricted.free.energy.tech.lemma.high.points.restricted}
            \mathbb{P}\Big(|\{v\in A_{N,\rho}^c : \psi_v \geq \gamma \log N^2\}| \geq N^{2\mathcal{E}_{\rho}(\gamma) + \varepsilon}\Big) \leq N^{-c(\gamma,\varepsilon,\bb{\sigma},\bb{\lambda})},
        \end{equation}
        for $N$ large enough, where
        \begin{equation*}
            \begin{array}{l}
                \mathcal{E}_{\rho}(\gamma) \circeq \max_{\gamma_1,\gamma_2,...,\gamma_{M-1}} \sum_{i=1}^{M-1} \left(\nabla \lambda_i - \frac{(\nabla \gamma_i)^2}{\sigma_i^2 \nabla \lambda_i}\right) + \left(\nabla \lambda_M - \frac{(\gamma - \gamma_{M-1})^2}{\sigma_M^2 \nabla \lambda_M}\right) - \rho/2, \\[3mm]
                \text{under the constraints }
                \sum_{i=1}^k \left(\nabla \lambda_i - \frac{(\nabla \gamma_i)^2}{\sigma_i^2 \nabla \lambda_i}\right) \geq \rho/2, \quad 1 \leq k \leq M-1.
            \end{array}
        \end{equation*}
        A rerun of the upper bound in the proof of Lemma \ref{lem:IGFF.free.energy.prob} shows that for all $\eta > 0$, there exists a constant $c = c(\eta,\beta,\rho,\boldsymbol{\sigma},\boldsymbol{\lambda}) > 0$ such that for $N$ large enough,
        \begin{equation}\label{eq:lem:restricted.free.energy.tech.lemma.last.1}
            \mathbb{P}\left(\widetilde{f}_{N,\rho}^{\psi}(\beta) > \max_{\gamma\in [0,\gamma_{\rho}^{\star}]} (\beta \gamma + \mathcal{E}_{\rho}(\gamma)) + \eta\right) \leq N^{-c}.
        \end{equation}
        The constraints associated with $\gamma_{\rho}^{\star}$ and $\mathcal{E}_{\rho}(\gamma)$ are respectively more restrictive than the constraints associated with $\gamma^{\star}$ and $\mathcal{E}(\gamma)$, so we obviously have
        \begin{equation}
            \gamma_{\rho}^{\star} \leq \gamma^{\star} \quad \text{and} \quad \mathcal{E}_{\rho}(\gamma) \leq \mathcal{E}(\gamma) - \rho/2, \quad 0 \leq \gamma \leq \gamma_{\rho}^{\star}.
        \end{equation}
        Therefore,
        \begin{equation}\label{eq:lem:restricted.free.energy.tech.lemma.last.2}
            \max_{\gamma\in [0,\gamma_{\rho}^{\star}]} (\beta \gamma + \mathcal{E}_{\rho}(\gamma)) \leq \max_{\gamma\in [0,\gamma^{\star}]} (\beta \gamma + \mathcal{E}(\gamma) - \rho/2) \stackrel{\eqref{eq:thm:IGFF.free.energy}}{=} f^{\psi}(\beta) - \rho/2.
        \end{equation}
        The conclusion of the lemma follows directly from \eqref{eq:lem:restricted.free.energy.tech.lemma.last.1} and \eqref{eq:lem:restricted.free.energy.tech.lemma.last.2}.
    \end{proof}

    \begin{lemma}\label{lem:gibbs.measure.near.boundary}
        Let $\beta > 0$ and $\rho\in (0,\lambda_1)$. Then,
        \begin{equation}
            \lim_{N\rightarrow \infty} \mathcal{G}_{\beta,N}(A_{N,\rho}^c) = 0,
        \end{equation}
        where the limit holds in $\PP$-probability and in $L^p, ~1 \leq p < \infty$.
    \end{lemma}

    \begin{remark}
        The result in Lemma \ref{lem:gibbs.measure.near.boundary} would not hold if we considered instead the complement of $V_N^{\delta}$, which is much larger than the complement of $A_{N,\rho}$.
    \end{remark}

    \begin{proof}[Proof of Lemma \ref{lem:gibbs.measure.near.boundary}]
        Fix $\rho\in (0,\lambda_1)$ and $\varepsilon\in (0,1)$, and let $\widetilde{\eta} > 0$ depend on $\rho$. We have
        \begin{align}
            \prob{}{\mathcal{G}_{\beta,N}(A_{N,\rho}^c) > \varepsilon}
            &\leq \prob{}{\mathcal{G}_{\beta,N}(A_{N,\rho}^c) > \varepsilon, \frac{1}{\log N^2} \log Z_N^{\psi}(\beta) \geq f^{\psi}(\beta) - \widetilde{\eta}} \notag \\
            &\quad+ \prob{}{\frac{1}{\log N^2} \log Z_N^{\psi}(\beta) < f^{\psi}(\beta) - \widetilde{\eta}} \notag \\
            &\circeq (1) + (2).
        \end{align}
        For any $\widetilde{\eta} > 0$, we have $(2)\to 0$ by \eqref{lem:IGFF.free.energy.prob.lower.bound}.
        Furthermore, since
        \begin{equation}
            \left\{\mathcal{G}_{\beta,N}(A_{N,\rho}^c) > \varepsilon\right\} \subseteq \Big\{\log \hspace{-1mm}\sum_{v\in A_{N,\rho}^c} \hspace{-1mm}e^{\beta \psi_v} > \log Z_N^{\psi}(\beta) + \log \varepsilon\Big\},
        \end{equation}
        then
        \begin{equation}\label{eq:gibbs.measure.near.boundary.eq.1}
            (1) \leq \mathbb{P}\Bigg(\frac{1}{\log N^2} \log \hspace{-1mm}\sum_{v\in A_{N,\rho}^c} \hspace{-1mm}e^{\beta \psi_v} > f^{\psi}(\beta) - \rho/2 + \eta\Bigg),
        \end{equation}
        where
        \begin{equation}
            \eta \circeq \rho/2 - \widetilde{\eta} + \frac{\log \varepsilon}{\log N^2}.
        \end{equation}
        Choose $\widetilde{\eta} > 0$ small enough, with respect to $\rho$, and $N$ large enough, with respect to $\rho$ and $\varepsilon$, that $\eta > 0$. The right-hand side of \eqref{eq:gibbs.measure.near.boundary.eq.1} converges to $0$ by Lemma \ref{lem:restricted.free.energy.tech.lemma}.
        This proves $\lim_{N\rightarrow \infty} \mathcal{G}_{\beta,N}(A_{N,\rho}^c) = 0$ in $\PP$-probability.
        Since
        \begin{equation}
            \sup_{N\in \N} |\mathcal{G}_{\beta,N}(A_{N,\rho}^c)|^p \leq 1,
        \end{equation}
        the $L^p$ convergence follows trivially.
    \end{proof}

    The fact that the Gibbs measure does not carry any weight on $A_{N,\rho}^c$ in the limit generalizes to expectations of bounded functions of $s$ vertices in $V_N$ sampled from the product of Gibbs measures. In Section \ref{sec:proof.ghirlanda.guerra.identities}, this will be used to obtain the approximate extended Ghirlanda-Guerra identities on $V_N$ from the ones on $A_{N,\rho}$.

    \begin{proposition}\label{prop:product.gibbs.near.boundary}
        Let $\beta > 0$ and $\rho\in (0,\lambda_1)$.
        Denote $\boldsymbol{v} \circeq (v^1,v^2,...,v^s)$.
        Then, for any $s\in \N$ and any functions $h \hspace{-0.5mm}:\hspace{-0.5mm} V_N^s \rightarrow \R$ such that $\sup_N \|h\|_{\infty} < \infty$,
        \begin{equation}
            \lim_{N\rightarrow \infty} \Big|\mathbb{E}\mathcal{G}_{\beta,N}^{\times s} \big[h(\bb{v})\big] - \mathbb{E}\mathcal{G}_{\beta,N,\rho}^{\times s} \big[h(\bb{v})\big]\Big| = 0.
        \end{equation}
    \end{proposition}

    \begin{proof}
        Introducing an auxiliary term,
        \begin{align}
            \Big|\mathbb{E}\mathcal{G}_{\beta,N}^{\times s} \big[h(\bb{v})\big] - \mathbb{E}\mathcal{G}_{\beta,N,\rho}^{\times s} \big[h(\bb{v})\big]\Big|
            &\leq \Big|\mathbb{E}\mathcal{G}_{\beta,N}^{\times s}\big[h(\bb{v})\big] - \mathbb{E}\mathcal{G}_{\beta,N}^{\times s}\big[h(\bb{v}) \, \bb{1}_{\{\boldsymbol{v}\in A_{N,\rho}^{\times s}\}}\big]\Big| \notag \\
            &\, + \Big|\mathbb{E}\mathcal{G}_{\beta,N}^{\times s}\big[h(\bb{v}) \, \bb{1}_{\{\bb{v}\in A_{N,\rho}^{\times s}\}}\big] - \mathbb{E}\mathcal{G}_{\beta,N,\rho}^{\times s}\big[h(\bb{v})\big]\Big| \notag \\
            &\circeq (1) + (2).
        \end{align}
        Now, by monotonicity and sub-additivity,
        \begin{equation}
            (1) = \mathbb{E}\mathcal{G}_{\beta,N}^{\times s}\big[h(\bb{v}) \bb{1}_{\{\exists i\in \{1,...,s\} ~\text{s.t. } v^i\in A_{N,\rho}^c\}}\big] \leq s ~\mathbb{E}\mathcal{G}_{\beta,N}(A_{N,\rho}^c) \cdot \sup_N \|h\|_{\infty}.
        \end{equation}
        Similarly, for the second term,
        \begin{align}
            (2)
            &= \mathbb{E}\mathcal{G}_{\beta,N,\rho}^{\times s}\big[h(\bb{v})\big] - \mathbb{E}\mathcal{G}_{\beta,N}^{\times s}\big[h(\bb{v}) \, \bb{1}_{\{\bb{v}\in A_{N,\rho}^{\times s}\}}\big] \notag \\
            &= \esp{}{\frac{\mathcal{G}_{\beta,N}^{\times s}\big[h(\bb{v}) \, \bb{1}_{\{\bb{v}\in A_{N,\rho}^{\times s}\}}\big]}{\mathcal{G}_{\beta,N}^{\times s}(\bb{v}\in A_{N,\rho}^{\times s})} \left(1 - \mathcal{G}_{\beta,N}^{\times s}(\bb{v}\in A_{N,\rho}^{\times s})\right)} \notag \\
            &\leq \esp{}{1 - \mathcal{G}_{\beta,N}^{\times s}(\bb{v}\in A_{N,\rho}^{\times s})} \cdot \sup_N \|h\|_{\infty} \notag \\
            &\leq s ~\mathbb{E}\mathcal{G}_{\beta,N}(A_{N,\rho}^c) \cdot \sup_N \|h\|_{\infty}.
        \end{align}
        By Lemma \ref{lem:gibbs.measure.near.boundary}, $(1) + (2) \rightarrow 0$ as $N\rightarrow \infty$.
        This ends the proof.
    \end{proof}

    When $s = 2$ and $h(v,v') \circeq \bb{1}_{\{q^N(v,v') \leq r\}}$, Proposition \ref{prop:product.gibbs.near.boundary} tells us that we can compute the limiting two-overlap distribution of Theorem \ref{thm:two.overlap.distribution.limit} by only considering a restricted version, where the points are sampled from $A_{N,\rho}^2$ instead of $V_N^2$.
    This property will be crucial to control the covariance of the increments in the next section, where we adapt the Bovier-Kurkova technique.
    The proof of Theorem \ref{thm:two.overlap.distribution.limit} will be given right after, in Section \ref{sec:proof.two.overlap.disribution.limit}.

    \begin{corollary}\label{cor:two.overlap.distribution.near.boundary}
        Let $\beta > 0$ and $\rho\in (0,\lambda_1)$. Then, for any $r\in \R$,
        \begin{equation}\label{eq:cor:two.overlap.distribution.near.boundary.limit}
            \lim_{N\rightarrow \infty} \Big|\mathbb{E}\mathcal{G}_{\beta,N}^{\times 2} \big[\bb{1}_{\{q^N(v,v') \leq r\}}\big] - \mathbb{E}\mathcal{G}_{\beta,N,\rho}^{\times 2} \big[\bb{1}_{\{q^N(v,v') \leq r\}}\big]\Big| = 0.
        \end{equation}
        Note that \eqref{eq:cor:two.overlap.distribution.near.boundary.limit} is valid even if $r$ depends on $\rho$.
    \end{corollary}

    \subsection{An adaptation of the Bovier-Kurkova technique}\label{sec:bovier.kurkova.technique.general}

    The Bovier-Kurkova technique is a way to compute the two-overlap distribution of a model in terms of the free energy of a perturbed version of that model.
    In the context of this paper, this connection is established by Proposition \ref{prop:BV.technique} below in the case $(s = 1, k = 1, h \equiv 1)$. One difficulty in the present case is the fact that the covariance between the increments of the field depends on their position relative to the boundary. The restriction to the set $A_{N,\rho}$ is a way to control this, cf. Lemma \ref{lem:covariance.estimates.4}.

    To simplify the notation, recall
    \begin{equation}\label{eq:normalized.speed.function.2}
        \bar{\mathcal{J}}_{\sigma^2}(\cdot) \circeq \frac{\mathcal{J}_{\sigma^2}(\cdot)}{\mathcal{J}_{\sigma^2}(1)},
    \end{equation}
    and denote the increments of overlaps by
    \begin{equation}\label{eq:overlap.increments}
        q_{\alpha,\alpha'}^N(v,v') \circeq \frac{\mathbb{E}\big[\psi_v(\alpha,\alpha')\psi_{v'}\big]}{\mathcal{J}_{\sigma^2}(1) \log N + C_0}, \quad v,v'\in V_N,
    \end{equation}
    where $C_0$ is the constant introduced in Lemma \ref{lem:IGFF.variance.uniform.bound}.
    Estimates on \eqref{eq:overlap.increments} are given in terms of \eqref{eq:normalized.speed.function.2} in Corollary \ref{cor:covariance.estimates.3} of Appendix \ref{sec:covariance.estimates}.
    The following lemma uses these estimates in order to compare $q^N(v,v')$ and $q_{\alpha,\alpha'}^N(v,v')$.

    \begin{lemma}\label{lem:covariance.estimates.4}
        Let $0 \hspace{-0.5mm}\leq \hspace{-0.5mm}\alpha \hspace{-0.5mm}< \hspace{-0.5mm}\alpha' \hspace{-0.5mm}\leq \hspace{-0.5mm}1$ and $\rho\in (0,1]$.
        Then, for all $v,v'\in A_{N,\rho}$, for all
        \begin{equation}
            \varepsilon \geq \frac{C_7}{\sqrt{\log N}} + C_8 \, \rho, \qquad (C_7,C_8 ~\text{are from \eqref{eq:covariance.estimate.3.2}})
        \end{equation}
        and for $N$ large enough (dependent on $\alpha$ and $\alpha'$, but independent from $v,v'$, and independent from $\rho$ (except when $\alpha = 0$)) :
        \begin{itemize}
            \item[(1)] If $q^N(v,v') \leq \bar{\mathcal{J}}_{\sigma^2}(\alpha) - \varepsilon$, then
                \begin{equation*}
                    q_{\alpha,\alpha'}^N(v,v') = O\left((\log N)^{-1/2}\right) + O(\rho).
                \end{equation*}
            \item[(2)] If $\bar{\mathcal{J}}_{\sigma^2}(\alpha) + \varepsilon \leq q^N(v,v') \leq \bar{\mathcal{J}}_{\sigma^2}(\alpha') - \varepsilon$, then
                \begin{equation*}
                    q_{\alpha,\alpha'}^N(v,v') = q^N(v,v') - \bar{\mathcal{J}}_{\sigma^2}(\alpha) + O\left((\log N)^{-1/2}\right) + O(\rho).
                \end{equation*}
            \item[(3)] If $\bar{\mathcal{J}}_{\sigma^2}(\alpha') + \varepsilon \leq q^N(v,v')$, then
                \begin{equation*}
                    q_{\alpha,\alpha'}^N(v,v') = \bar{\mathcal{J}}_{\sigma^2}(\alpha,\alpha') + O\left((\log N)^{-1/2}\right) + O(\rho).
                \end{equation*}
        \end{itemize}
        In all three cases, $O(\rho)$ is uniform in $N$.
    \end{lemma}

    \begin{proof}
        From \eqref{eq:covariance.estimate.3.2}, we know that $|q^N(v,v') - \bar{\mathcal{J}}_{\sigma^2}(b_N(v,v'))| \leq \varepsilon$. Thus, in each case respectively, we deduce $(1)$ : $b_N \leq \alpha$, $(2)$ : $\alpha \leq b_N \leq \alpha'$, and $(3)$ : $\alpha' \leq b_N$. Use \eqref{eq:covariance.estimate.3.2} again to get the appropriate bounds on $q_{\alpha,\alpha'}^N(v,v')$.
    \end{proof}

    Here is the main result of this section.

    \begin{proposition}\label{prop:BV.technique}
        Let $0 \hspace{-0.5mm}\leq \hspace{-0.5mm}\alpha \hspace{-0.5mm}< \hspace{-0.5mm}\alpha' \hspace{-0.5mm}\leq \hspace{-0.5mm}1$, $\rho\in (0,1]$ and $S_{\alpha,\alpha'} \circeq (\bar{\mathcal{J}}_{\sigma^2}(\alpha),\bar{\mathcal{J}}_{\sigma^2}(\alpha')]$.
        Let $\beta > 0$, $s\in \N$, $k\in \{1,...,s\}$, and let $h : V_N^s \rightarrow \R$ be such that $\sup_N \|h\|_{\infty} < \infty$. Then, for all
        \vspace{-1mm}
        \begin{equation}
            \varepsilon \geq \frac{C_7}{\sqrt{\log N}} + C_8 \, \rho, \qquad (C_7,C_8 ~\text{are from \eqref{eq:covariance.estimate.3.2}})
        \end{equation}
        and for $N$ large enough (dependent on $\alpha$ and $\alpha'$, but independent from $v,v'$, and independent from $\rho$ (except when $\alpha = 0$)), we have
        \begin{align}\label{eq:BV.technique}
            &\left|\hspace{-1mm}
                \begin{array}{l}
                    \resizebox{0.33\hsize}{!}{%
                        $\frac{\begin{array}{l}\mathbb{E}\mathcal{G}_{\beta,N,\rho}^{\times s} [\psi_{v^k}(\alpha,\alpha') h(\bb{v})]\end{array}}{\begin{array}{l}\beta\, (\mathcal{J}_{\sigma^2}(1) \log N + C_0)\end{array}}$
                    }
                    \hspace{-0.5mm}
                    -   \left\{\hspace{-1mm}
                        \begin{array}{l}
                            \sum_{l=1}^s \mathbb{E}\mathcal{G}_{\beta,N,\rho}^{\times s} \big[\int_{S_{\alpha,\alpha'}} \hspace{-1mm}\bb{1}_{\{r < q^N(v^k,v^l)\}} dr \, h(\bb{v})\big] \\[1.5mm]
                            - s \, \mathbb{E}\mathcal{G}_{\beta,N,\rho}^{\times (s+1)} \big[\int_{S_{\alpha,\alpha'}} \hspace{-1mm}\bb{1}_{\{r < q^N(v^k,v^{s+1})\}} dr \, h(\bb{v})\big]
                        \end{array}
                        \hspace{-1.5mm}\right\}
                \end{array}
            \hspace{-1mm}\right| \notag \\
            &\quad\leq C \cdot s \cdot \sup_N \|h\|_{\infty} \cdot \left\{\hspace{-1mm}
                \begin{array}{l}
                    \mathbb{E}\mathcal{G}_{\beta,N,\rho}^{\times 2} \big[\bb{1}_{\{\bar{\mathcal{J}}_{\sigma^2}(\alpha) - \varepsilon \leq q^N(v,v') \leq \bar{\mathcal{J}}_{\sigma^2}(\alpha) + \varepsilon\}}\big] \\[1.5mm]
                    \mathbb{E}\mathcal{G}_{\beta,N,\rho}^{\times 2} \big[\bb{1}_{\{\bar{\mathcal{J}}_{\sigma^2}(\alpha') - \varepsilon \leq q^N(v,v') \leq \bar{\mathcal{J}}_{\sigma^2}(\alpha') + \varepsilon\}}\big] \\[1.5mm]
                    + O\left((\log N)^{-1/2}\right) + O(\rho)
                \end{array}
                \hspace{-1.5mm}\right\},
        \end{align}
        where $O(\rho)$ is uniform in $N$ and $C > 0$ is a universal constant.
    \end{proposition}

    \begin{proof}
        For any $l\in \{1,...,s+1\}$,
        \begin{align}\label{eq:prop:BV.technique.beginning}
            &\mathbb{E}\mathcal{G}_{\beta,N,\rho}^{\times (s+1)} \Big[\int_{S_{\alpha,\alpha'}} \hspace{-3mm}\bb{1}_{\{r < q^N(v^k,v^l)\}} dr ~h(\bb{v})\Big] \\
            &\quad= \mathbb{E}\mathcal{G}_{\beta,N,\rho}^{\times (s+1)} \Big[(q^N(v^k,v^l) - \bar{\mathcal{J}}_{\sigma^2}(\alpha)) \bb{1}_{\{\bar{\mathcal{J}}_{\sigma^2}(\alpha) < q^N(v^k,v^l) \leq \bar{\mathcal{J}}_{\sigma^2}(\alpha')\}}\, h(\bb{v})\Big] \notag \\
            &\quad\quad+ \mathbb{E}\mathcal{G}_{\beta,N,\rho}^{\times (s+1)} \Big[\bar{\mathcal{J}}_{\sigma^2}(\alpha,\alpha') \bb{1}_{\{\bar{\mathcal{J}}_{\sigma^2}(\alpha') < q^N(v^k,v^l)\}}\, h(\bb{v})\Big]. \notag
        \end{align}
        On the other hand,
        \begin{equation}\label{eq:generalization.gaussian.integration.by.parts.apply.to}
            \mathbb{E}\mathcal{G}_{\beta,N,\rho}^{\times s}\big[\psi_{v^k}(\alpha,\alpha') h(\bb{v})\big] = \hspace{-1mm}\sum_{\bb{v}\in A_{N,\rho}^{\times s}} \hspace{-1mm}\esp{}{\frac{\psi_{v^k}(\alpha,\alpha') h(\bb{v}) \prod_{l=1}^s \hspace{-0.5mm}\exp(\beta \psi_{v^l})}{\prod_{l'=1}^s \sum_{v^{l'}\hspace{-0.7mm}\in A_{N,\rho}} \hspace{-0.5mm}\exp(\beta \psi_{v^{l'}})}}.
        \end{equation}
        For a centered Gaussian vector $\boldsymbol{X} \circeq (X_1,...,X_n)$ and a twice-continuously differentiable function $F$ on $\R^n$, of moderate growth at infinity, we have the formula $\esp{}{X_i F(\boldsymbol{X})} = \sum_{j=1}^n \esp{}{X_i X_j} \esp{}{\partial_{X_j} F(\boldsymbol{X})}$. Here, for any $\bb{v}\in A_{N,\rho}^{\times s}$, the relevant Gaussian vector is
        \begin{equation}
            \big(\psi_{v^k}(\alpha,\alpha') ~;~ \psi_{v^l}, l\in \{1,...,s\} ~;~ \psi_{v^{l'}}, v^{l'}\hspace{-1mm}\in A_{N,\rho}, l'\in \{1,...,s\}\big),
        \end{equation}
        where $X_i \circeq \psi_{v^k}(\alpha,\alpha')$ and $F \circeq h(\bb{v}) \prod_{l=1}^s \hspace{-0.5mm}\exp(\beta \psi_{v^l}) / \prod_{l'=1}^s \sum_{v^{l'}\hspace{-0.7mm}\in A_{N,\rho}} \hspace{-0.5mm}\exp(\beta \psi_{v^{l'}})$.
        Applying the formula to the right-hand side of \eqref{eq:generalization.gaussian.integration.by.parts.apply.to} yields
        \begin{equation}\label{eq:generalization.gaussian.integration.by.parts.apply.to.follow.up}
            \begin{aligned}
                \eqref{eq:generalization.gaussian.integration.by.parts.apply.to}
                &=\sum_{l=1}^s \beta\, \mathbb{E}\mathcal{G}_{\beta,N,\rho}^{\times s}\Big[\mathbb{E}\big[\psi_{v^k}(\alpha,\alpha')\psi_{v^l}\big] h(\bb{v})\Big] \\
                &\quad\quad- s\, \beta\, \mathbb{E}\mathcal{G}_{\beta,N,\rho}^{\times (s+1)}\Big[\mathbb{E}\big[\psi_{v^k}(\alpha,\alpha')\psi_{v^{s+1}}\big] h(\bb{v})\Big].
            \end{aligned}
        \end{equation}
        If we divide \eqref{eq:generalization.gaussian.integration.by.parts.apply.to.follow.up} on both sides by $\beta\, (\mathcal{J}_{\sigma^2}(1) \log N + C_0)$, we deduce
        \begin{equation}\label{eq:prop:BV.technique.end}
            \frac{\mathbb{E}\mathcal{G}_{\beta,N,\rho}^{\times s}\big[\psi_{v^k}(\alpha,\alpha') h(\bb{v})\big]}{\beta\, (\mathcal{J}_{\sigma^2}(1) \log N + C_0)} =
            \left\{\hspace{-1mm}
                \begin{array}{l}
                    \vspace{1mm}\sum_{l=1}^s \mathbb{E}\mathcal{G}_{\beta,N,\rho}^{\times s} \big[q_{\alpha,\alpha'}^N(v^k,v^l) h(\bb{v})\big] \\
                    - s\, \mathbb{E}\mathcal{G}_{\beta,N,\rho}^{\times (s+1)} \big[q_{\alpha,\alpha'}^N(v^k,v^{s+1}) h(\bb{v})\big]
                \end{array}
            \hspace{-1mm}\right\}.
        \end{equation}
        Now, one by one, take the difference in absolute value between each of the $s+1$ expectations inside the braces in \eqref{eq:prop:BV.technique.end} and the corresponding expectation on the left-hand side of \eqref{eq:prop:BV.technique.beginning}. We obtain the bound \eqref{eq:BV.technique} by using Lemma \ref{lem:covariance.estimates.4}.
    \end{proof}

    \subsection{Computation of the limiting two-overlap distribution}\label{sec:proof.two.overlap.disribution.limit}

    Let $\alpha,\alpha'\in [0,1]$ be such that
    \vspace{-2mm}
    \begin{equation}\label{eq:a.a.condition}
        \lambda^{j^{\star}-1} \leq \lambda_{i^{\star}-1} \leq \alpha < \alpha' \leq \lambda_{i^{\star}} \leq \lambda^{j^{\star}}
    \end{equation}
    for some $i^{\star}$ and $j^{\star}$.
    Define $\psi^{u}$, the {\it perturbed scale-inhomogeneous GFF}, mentioned in the previous section, by
    \vspace{-1mm}
    \begin{equation}\label{eq:perturbed.psi}
        \psi_v^u \circeq u\, \phi_v(\alpha,\alpha') + \psi_v, \quad \text{where } u > -\sigma_{i^{\star}}.
    \end{equation}
    The dependence on $\alpha$ and $\alpha'$ is made implicit to lighten the notation. In the proof of Theorem \ref{thm:two.overlap.distribution.limit}, Proposition \ref{prop:BV.technique} will be used to link the limiting two-overlap distribution of $\psi$ to the derivative of the limiting free energy of $\psi^u$ with respect to the perturbation parameter $u$.

    \begin{proof}[Proof of Theorem \ref{thm:two.overlap.distribution.limit}]
        By Corollary \ref{cor:two.overlap.distribution.near.boundary}, it suffices to prove that
        \begin{align}\label{eq:two.overlap.distribution.limit.to.show}
            &\lim_{\rho\rightarrow 0} \lim_{N\rightarrow \infty} \mathbb{E}\mathcal{G}_{\beta,N,\rho}^{\times 2} \big[\bb{1}_{\{q^N(v,v') \leq r\}}\big] \notag \\
            &\hspace{15mm}=
            \left\{\hspace{-1mm}
            \begin{array}{ll}
                0, &\mbox{if } r < 0, \\
                (2 / \bar{\sigma}_j) / \beta, &\mbox{if } r\in [\bar{\mathcal{J}}_{\sigma^2}(\lambda^{j-1}),\bar{\mathcal{J}}_{\sigma^2}(\lambda^j)), ~j \leq l_{\beta}-1, \\
                1, &\mbox{if } r \geq \bar{\mathcal{J}}_{\sigma^2}(\lambda^{l_{\beta}-1}).
            \end{array}
            \right.
        \end{align}
        Since $[0,1] \subseteq \R$ is compact, the space $\mathcal{M}_1([0,1])$ of probability measures on $[0,1]$ is compact under the weak topology.
        Thus, any subsequence of the cumulative distribution functions on the left-hand side of \eqref{eq:two.overlap.distribution.limit.to.show} has a subsequence converging to a cumulative distribution function.
        Pick any converging sub-subsequence and denote its limit by $r \mapsto Q_{\beta}(r)$.
        Since $\mathcal{M}_1([0,1])$ is a metric space, the proof is reduced to showing that $Q_{\beta}$ is given by the right-hand side of \eqref{eq:two.overlap.distribution.limit.to.show}.

        We already know that $Q_{\beta}(r) = 0$ for all $r < 0$ since Corollary \ref{cor:covariance.estimates.3} implies
        \begin{equation}
            \liminf_{\rho\rightarrow 0} \, \liminf_{N\rightarrow \infty} \min_{v,v'\in A_{N,\rho}} \hspace{-1mm}q^N(v,v') \geq 0.
        \end{equation}
        We also have $Q_{\beta}(r) = 1$ for all $r \geq 1$ since $\max_{v,v'\in V_N} q^N(v,v') \leq 1$ by Lemma \ref{lem:IGFF.variance.uniform.bound} and the Cauchy-Schwarz inequality.

        To determine $Q_{\beta}$ on $[0,1)$, let $\alpha, \alpha'\in [0,1]$ be such that $\bar{\mathcal{J}}_{\sigma^2}(\alpha),\bar{\mathcal{J}}_{\sigma^2}(\alpha')$ are continuity points of $Q_{\beta}$ and \eqref{eq:a.a.condition} is satisfied.
        Direct differentiation gives
        \begin{equation}\label{eq:direct.differentiation.free.energy}
            \frac{2\sigma_{i^{\star}}}{\beta^2 \mathcal{J}_{\sigma^2}(1)} \left.\frac{\partial}{\partial u} \esp{}{f_{N,\rho}^{\psi^u}(\beta)} \right|_{u=0} \hspace{-1mm}= \frac{\mathbb{E}\mathcal{G}_{\beta,N,\rho}\big[\psi_v(\alpha,\alpha')\big]}{\beta \mathcal{J}_{\sigma^2}(1) \log N}.
        \end{equation}
        Combine this result with Proposition \ref{prop:BV.technique} in the special case ($s = 1$, $k = 1$, $h \equiv 1$). After taking the limits $N\rightarrow \infty$ (use Corollary \ref{cor:two.overlap.distribution.near.boundary} on the right-hand side of \eqref{eq:BV.technique}), $\rho \rightarrow 0$ and then $\varepsilon\rightarrow 0$, we find
        \begin{equation}\label{eq:proof.two.overlap.distribution.limit.eq}
            \int_{(\bar{\mathcal{J}}_{\sigma^2}(\alpha),\bar{\mathcal{J}}_{\sigma^2}(\alpha')]} \hspace{-4mm}Q_{\beta}(r) dr = \lim_{\rho\rightarrow 0} \lim_{N\rightarrow \infty} \frac{2 \sigma_{i^{\star}}}{\beta^2 \mathcal{J}_{\sigma^2}(1)} \left.\frac{\partial}{\partial u} \esp{}{f_{N,\rho}^{\psi^u}(\beta)}\right|_{u=0}.
        \end{equation}
        For all $\rho\in (0,1]$, the function $u\mapsto \mathbb{E}\big[f_{N,\rho}^{\psi^u}(\beta)\big]$ is convex by Lemma \ref{lem:IGFF.modified.free.energy.convexity}, and by Theorem \ref{thm:IGFF.free.energy.A.N.p},
        \vspace{-1mm}
        \begin{equation}\label{eq:convergence.free.energy}
            \lim_{N\rightarrow \infty} \esp{}{f_{N,\rho}^{\psi^u}(\beta)} = f^{\psi^u}(\beta).
        \end{equation}
        Pointwise limits preserve convexity, so $u \mapsto f^{\psi^u}(\beta)$ is convex. From Lemma \ref{lem:IGFF.modified.free.energy.derivative}, assuming $\beta\in \mathcal{B}$, we also know that $u \mapsto f^{\psi^u}(\beta)$ is differentiable on an open interval $(-\delta,\delta)$, for $\delta = \delta(\beta,\alpha,\alpha',\boldsymbol{\sigma},\boldsymbol{\lambda})$ small enough. In particular, by another standard result of convexity (see e.g. Theorem 25.7 in \cite{MR0274683}),
        \begin{equation}\label{eq:convergence.free.energy.derivative}
            \lim_{N\rightarrow \infty} \frac{\partial}{\partial u} \esp{}{f_{N,\rho}^{\psi^u}(\beta)} = \frac{\partial}{\partial u} f^{\psi^u}(\beta),
        \end{equation}
        for all $u\in (-\delta,\delta)$ (and all $\rho\in (0,1]$). The derivative of $u \mapsto f^{\psi^u}(\beta)$ at $u=0$ is given by \eqref{eq:lem:IGFF.modified.free.energy.derivative}. Thus, from \eqref{eq:proof.two.overlap.distribution.limit.eq}, we get
        \begin{equation}\label{eq:proof.two.overlap.distribution.limit.end}
            \int_{(\bar{\mathcal{J}}_{\sigma^2}(\alpha),\bar{\mathcal{J}}_{\sigma^2}(\alpha')]} \hspace{-4mm}Q_{\beta}(r) dr =
            \left\{\hspace{-1.5mm}
            \begin{array}{ll}
                \vspace{0.7mm}\bar{\mathcal{J}}_{\sigma^2}(\alpha,\alpha') \frac{(2 / \bar{\sigma}_{j^{\star}})}{\beta}, &\mbox{if } j^{\star} \leq l_{\beta} - 1, \\
                \bar{\mathcal{J}}_{\sigma^2}(\alpha,\alpha'), &\mbox{if } j^{\star} \geq l_{\beta}.
            \end{array}
            \right.
        \end{equation}
        But $Q_{\beta}$ is right-continuous (it's a cumulative distribution function) and \eqref{eq:proof.two.overlap.distribution.limit.end} holds for all pairs $\bar{\mathcal{J}}_{\sigma^2}(\alpha),\bar{\mathcal{J}}_{\sigma^2}(\alpha')$ of continuity points satisfying \eqref{eq:a.a.condition}, so $Q_{\beta}$ must be equal to the right-hand side of \eqref{eq:two.overlap.distribution.limit.to.show}. This ends the proof.
    \end{proof}

    \subsection{Proof of the approximate extended Ghirlanda-Guerra identities}\label{sec:proof.ghirlanda.guerra.identities}

    We start by proving a concentration result.
    Denote $\bb{v} \circeq (v^1,...,v^s)$ everywhere in this section.

    \begin{lemma}\label{lem:IGFF.ghirlanda.guerra.restricted.2}
        Let $\lambda_{i^{\star}-1} \hspace{-0.3mm}\leq \hspace{-0.1mm}\alpha \hspace{-0.3mm}<\hspace{-0.1mm} \alpha' \hspace{-0.3mm}\leq\hspace{-0.1mm} \lambda_{i^{\star}}$ for some $i^{\star}$, and let $\beta\in \mathcal{B}$ and $\rho\in (0,1]$.
        Then, for any $s\in \N$, any $k\in \{1,...,s\}$ and any functions $h : V_N^s \rightarrow \R$ such that $\sup_N \|h\|_{\infty} < \infty$,
        \begin{align}\label{eq:lem:IGFF.ghirlanda.guerra.restricted.2}
            &\lim_{N \rightarrow \infty} \frac{\Big|\mathbb{E}\mathcal{G}_{\beta,N,\rho}^{\times s}\big[\psi_{v^k}(\alpha,\alpha') h(\bb{v})\big] - \mathbb{E}\mathcal{G}_{\beta,N,\rho}\big[\psi_{v^k}(\alpha,\alpha')\big] \mathbb{E}\mathcal{G}_{\beta,N,\rho}^{\times s}\big[h(\bb{v})\big]\Big|}{\beta\, (\mathcal{J}_{\sigma^2}(1) \log N + C_0)} = 0.
        \end{align}
    \end{lemma}

    \begin{proof}
        If we apply Jensen's inequality to the expectation $\mathbb{E}\mathcal{G}_{\beta,N,\rho}^{\times s}[\cdot]$, followed by the triangle inequality, we have
        \begin{align}\label{eq:lem:IGFF.ghirlanda.guerra.restricted.2.beginning}
            &\Big|\mathbb{E}\mathcal{G}_{\beta,N,\rho}^{\times s}\big[\psi_{v^k}(\alpha,\alpha') h(\bb{v})\big] - \mathbb{E}\mathcal{G}_{\beta,N,\rho}\big[\psi_{v^k}(\alpha,\alpha')\big] \mathbb{E}\mathcal{G}_{\beta,N,\rho}^{\times s}\big[h(\bb{v})\big]\Big| \notag \\
            &\quad\leq \mathbb{E}\mathcal{G}_{\beta,N,\rho} \Big|\psi_{v^k}(\alpha,\alpha') - \mathbb{E}\mathcal{G}_{\beta,N,\rho} \big[\psi_{v^k}(\alpha,\alpha')\big]\Big| \cdot \sup_N \|h\|_{\infty} \notag \\
            &\quad\leq ((a) + (b)) \cdot \sup_N \|h\|_{\infty},
        \end{align}
        where
        \begin{equation}
            \begin{aligned}
                (a) + (b)
                &\circeq \mathbb{E}\mathcal{G}_{\beta,N,\rho}\Big|\psi_{v^k}(\alpha,\alpha') - \mathcal{G}_{\beta,N,\rho}\big[\psi_{v^k}(\alpha,\alpha')\big]\Big| \\
                &\quad+ \mathbb{E}\Big|\mathcal{G}_{\beta,N,\rho}\big[\psi_{v^k}(\alpha,\alpha')\big] - \mathbb{E}\mathcal{G}_{\beta,N,\rho}\big[\psi_{v^k}(\alpha,\alpha')\big]\Big|.
            \end{aligned}
        \end{equation}
        In the remainder, we follow the strategy developed in the proof of Theorem 3.8 in \cite{MR3052333}, where the same concentration result was proved for the mixed $p$-spin model.
        We show that, for all $\rho\in (0,1]$,
        \begin{equation}
            \lim_{N\rightarrow \infty} \frac{(a)}{\log N} = 0 \quad \text{and} \quad \lim_{N\rightarrow \infty} \frac{(b)}{\log N} = 0.
        \end{equation}

        \noindent
        \begin{center}
            \fbox{\textbf{Step 1 :} For all $\rho\in (0,1]$, $\lim_{N\rightarrow \infty} \frac{(a)}{\log N} = 0$.}
        \end{center}
        Note that
        \begin{align}
            (a)
            &= \mathbb{E}\mathcal{G}_{\beta,N,\rho}\Bigg|\sum_{v^2\in A_{N,\rho}} (\psi_{v^1}(\alpha,\alpha') - \psi_{v^2}(\alpha,\alpha'))\frac{\exp(\beta \psi_{v^2})}{\sum_{z^2\in A_{N,\rho}} \hspace{-0.5mm}\exp(\beta \psi_{z^2})}\Bigg| \notag \\
            &\leq \mathbb{E}\mathcal{G}_{\beta,N,\rho}^{\times 2}\big|\psi_{v^1}(\alpha,\alpha') - \psi_{v^2}(\alpha,\alpha')\big|.
        \end{align}
        For $u \geq 0$, we define a perturbed version of the last quantity (where the Gibbs measure $\mathcal{G}_{\beta,N,\rho,u}$ is now defined with respect to $\psi^u$) :
        \begin{align}
            D(u)
            &\circeq \mathbb{E}\mathcal{G}_{\beta,N,\rho,u}^{\times 2} \big|\psi_{v^1}(\alpha,\alpha') - \psi_{v^2}(\alpha,\alpha')\big| \notag \\
            &= \mathbb{E}\Bigg[\sum_{v^1,v^2 \in A_{N,\rho}} \hspace{-3mm} \big|\psi_{v^1}(\alpha,\alpha') - \psi_{v^2}(\alpha,\alpha')\big|\Bigg. \\[-5pt]
            &\hspace{8mm} \Bigg.\cdot \frac{e^{\beta (u\phi_{v^1}(\alpha,\alpha') + \psi_{v^1})}}{\sum_{z^1\in A_{N,\rho}} e^{\beta (u \phi_{z^1}(\alpha,\alpha') + \psi_{z^1})}} \cdot \frac{e^{\beta (u \phi_{v^2}(\alpha,\alpha') + \psi_{v^2})}}{\sum_{z^2\in A_{N,\rho}} e^{\beta (u \phi_{z^2}(\alpha,\alpha') + \psi_{z^2})}}\Bigg] \notag.
        \end{align}
        We can easily verify that
        \begin{equation}\label{eq:difference.psi.u.integrals}
            u D(0) = \int_0^u D(y) dy - \int_0^u \int_0^x \frac{\partial}{\partial y} D(y) dy dx,
        \end{equation}
        and also that
        \begin{equation}
            \frac{\partial}{\partial y} D(y) = \frac{\beta}{\sigma_{i^{\star}}} \mathbb{E}\mathcal{G}_{\beta,N,\rho,y}^{\times 3}
            \Bigg[\hspace{-1mm}
            \begin{array}{l}
                \vspace{1mm} \big|\psi_{v^1}(\alpha,\alpha') - \psi_{v^2}(\alpha,\alpha')\big| \\
                \cdot\, \big(\psi_{v^1}(\alpha,\alpha') + \psi_{v^2}(\alpha,\alpha') - 2\psi_{v^3}(\alpha,\alpha')\big)
            \end{array}
            \hspace{-1mm}\Bigg].
        \end{equation}
        If we separate the last expectation in two parts and apply the Cauchy-Schwarz inequality to each one of them, we find (for $y \geq 0$) :
        \begin{align}\label{eq:lem:IGFF.ghirlanda.guerra.restricted.2.a.bound.derivative}
            \left|\frac{\partial}{\partial y} D(y)\right|
            &\leq \frac{\beta}{\sigma_{i^{\star}}}
                \left\{\hspace{-1mm}
                \begin{array}{l}
                    \vspace{2mm}\mathbb{E}\mathcal{G}_{\beta,N,\rho,y}^{\times 3} \big|\psi_{v^1}(\alpha,\alpha') - \psi_{v^2}(\alpha,\alpha')\big| \big|\psi_{v^1}(\alpha,\alpha') - \psi_{v^3}(\alpha,\alpha')\big| \\
                    +\, \mathbb{E}\mathcal{G}_{\beta,N,\rho,y}^{\times 3} \big|\psi_{v^1}(\alpha,\alpha') - \psi_{v^2}(\alpha,\alpha')\big| \big|\psi_{v^2}(\alpha,\alpha') - \psi_{v^3}(\alpha,\alpha')\big|
                \end{array}
                \hspace{-1mm}\right\} \notag \\
            &\leq \frac{\beta}{\sigma_{i^{\star}}} \cdot 2\, \mathbb{E}\mathcal{G}_{\beta,N,\rho,y}^{\times 2} \Big[\big(\psi_{v^1}(\alpha,\alpha') - \psi_{v^2}(\alpha,\alpha')\big)^2\Big].
        \end{align}
        From the elementary inequality $(c + d)^2 \leq 2c^2 + 2d^2$, we also have
        \begin{align}\label{eq:lem:IGFF.ghirlanda.guerra.restricted.2.a.bound.derivative.2}
            &2\, \mathbb{E}\mathcal{G}_{\beta,N,\rho,y}^{\times 2} \Big[\big(\psi_{v^1}(\alpha,\alpha') - \psi_{v^2}(\alpha,\alpha')\big)^2\Big] \notag \\
            &\quad \leq 8\, \mathbb{E}\mathcal{G}_{\beta,N,\rho,y} \Big[\Big(\psi_v(\alpha,\alpha') - \mathcal{G}_{\beta,N,\rho,y}\big[\psi_v(\alpha,\alpha')\big]\Big)^2\Big].
        \end{align}
        By putting \eqref{eq:lem:IGFF.ghirlanda.guerra.restricted.2.a.bound.derivative} and \eqref{eq:lem:IGFF.ghirlanda.guerra.restricted.2.a.bound.derivative.2} together in \eqref{eq:difference.psi.u.integrals}, we obtain (for $u > 0$) :
        \begin{align}
            D(0)
            &\leq \frac{1}{u} \int_0^u D(y) dy + \int_0^u \left|\frac{\partial}{\partial y} D(y)\right| dy \notag \\
            &\leq 2 \Bigg(\frac{1}{u} \int_0^u \mathbb{E}\mathcal{G}_{\beta,N,\rho,y}\Big[\Big(\psi_v(\alpha,\alpha') - \mathcal{G}_{\beta,N,\rho,y}\big[\psi_v(\alpha,\alpha')\big]\Big)^2\Big] dy\Bigg)^{1/2} \notag \\
            &\quad\quad+ \frac{8 \beta}{\sigma_{i^{\star}}} \int_0^u \mathbb{E}\mathcal{G}_{\beta,N,\rho,y} \Big[\Big(\psi_v(\alpha,\alpha') - \mathcal{G}_{\beta,N,\rho,y}\big[\psi_v(\alpha,\alpha')\big]\Big)^2\Big] dy.
        \end{align}
        In order to bound $\frac{1}{u}\int_0^u D(y) dy$, we separated $D(y)$ in two parts (with the triangle inequality) and we applied the Cauchy-Schwarz inequality to the two resulting expectations $\frac{1}{u} \int_0^u \mathbb{E}\mathcal{G}_{\beta,N,\rho,y}[\, \cdot\, ]\, dy$. Denote
        \begin{equation}\label{eq:lem:IGFF.ghirlanda.guerra.restricted.2.a.varepsilon}
            \varepsilon_{N,\rho}(u) \circeq \frac{1}{\log N} \int_0^u \mathbb{E}\mathcal{G}_{\beta,N,\rho,y} \Big[\Big(\psi_v(\alpha,\alpha') - \mathcal{G}_{\beta,N,\rho,y}\big[\psi_v(\alpha,\alpha')\big]\Big)^2\Big] dy.
        \end{equation}
        So far, we have shown that
        \begin{equation}\label{eq:lem:IGFF.ghirlanda.guerra.restricted.2.a.so.far}
            \frac{(a)}{\log N} \leq \frac{D(0)}{\log N} \leq 2 \sqrt{\frac{\varepsilon_{N,\rho}(u)}{u \log N}} + \frac{8 \beta}{\sigma_{i^{\star}}}\, \varepsilon_{N,\rho}(u).
        \end{equation}
        Let
        \begin{equation}\label{eq.F.u}
            F(u) \circeq f_{N,\rho}^{\psi^u}(\beta) = \frac{1}{\log N^2} \log\hspace{-1mm} \sum_{v\in A_{N,\rho}} \hspace{-1mm}e^{\beta (u \phi_v(\alpha,\alpha') + \psi_v)},
        \end{equation}
        and note that
        \begin{align}\label{eq:lem:IGFF.ghirlanda.guerra.restricted.2.a.so.far.2}
            \esp{}{F''(y)}
            &= \frac{\beta^2}{\sigma_{i^{\star}}^2 \log N^2}\, \mathbb{E}\Big[\mathcal{G}_{\beta,N,\rho,y} \big[\big(\psi_v(\alpha,\alpha')\big)^2\big] - \Big(\mathcal{G}_{\beta,N,\rho,y}\big[\psi_v(\alpha,\alpha')\big]\Big)^2\Big] \notag \\
            &= \frac{\beta^2}{2 \sigma_{i^{\star}}^2}\, \cdot \frac{1}{\log N}\, \mathbb{E}\mathcal{G}_{\beta,N,\rho,y} \Big[\Big(\psi_v(\alpha,\alpha') - \mathcal{G}_{\beta,N,\rho,y}\big[\psi_v(\alpha,\alpha')\big]\Big)^2\Big].
        \end{align}
        From \eqref{eq:lem:IGFF.ghirlanda.guerra.restricted.2.a.varepsilon} and the convexity of $F$ (see Lemma \ref{lem:IGFF.modified.free.energy.convexity}), we have, for all $y\in (0,\sigma_{i^{\star}})$,
        \begin{align}\label{eq:lem:IGFF.ghirlanda.guerra.restricted.2.varepsilon.bound}
            \varepsilon_{N,\rho}(u)
            &= \frac{2 \sigma_{i^{\star}}^2}{\beta^2} \int_0^u \esp{}{F''(y)} dy = \frac{2 \sigma_{i^{\star}}^2}{\beta^2}\, \esp{}{F'(u) - F'(0)} \notag \\
            &\leq \frac{2 \sigma_{i^{\star}}^2}{\beta^2}\, \esp{}{\frac{F(u + y) - F(u)}{y} - \frac{F(0) - F(-y)}{y}}.
        \end{align}
        By putting \eqref{eq:lem:IGFF.ghirlanda.guerra.restricted.2.varepsilon.bound} in \eqref{eq:lem:IGFF.ghirlanda.guerra.restricted.2.a.so.far} and by using the mean convergence in Theorem \ref{thm:IGFF.free.energy.A.N.p}, we get, for all $\rho\in (0,1]$ and all $u > 0$ and $y\in (0,\sigma_{i^{\star}})$,
        \begin{equation}
            \limsup_{N\rightarrow \infty} \frac{(a)}{\log N} \leq \frac{8 \beta}{\sigma_{i^{\star}}} \cdot \frac{2 \sigma_{i^{\star}}^2}{\beta^2} \left(\frac{f(u + y) - f(u)}{y} - \frac{f(0) - f(-y)}{y}\right),
        \end{equation}
        where $f(u) \circeq f^{\psi^u}(\beta)$.
        From Lemma \ref{lem:IGFF.modified.free.energy.derivative}, there exists $\delta = \delta(\beta,\alpha,\alpha',\boldsymbol{\sigma},\boldsymbol{\lambda})$ such that $f$ is differentiable on $(-\delta,\delta)$. Therefore, take $u \rightarrow 0^+$ and then $y \rightarrow 0^+$ in the above equation, the right-hand side goes to $0$. The left-hand side does not depend on $u$ or $y$, so we conclude that for all $\rho\in (0,1]$, $\lim_{N\rightarrow \infty} (a) / \log N = 0$.

        \vspace{2mm}
        \begin{center}
            \fbox{\textbf{Step 2 :} For all $\rho\in (0,1]$, $\lim_{N\rightarrow \infty} \frac{(b)}{\log N} = 0$.}
        \end{center}
        \vspace{0mm}
        Let $F(u) \circeq f_{N,\rho}^{\psi^u}(\beta)$ as in \eqref{eq.F.u} and, for $u\in (0,\sigma_{i^{\star}})$, let
        \begin{equation}
            \eta(u) \circeq \big|F(-u) - \mathbb{E}\big[F(-u)\big]\big| + \big|F(0) - \mathbb{E}\big[F(0)\big]\big| + \big|F(u) - \mathbb{E}\big[F(u)\big]\big|.
        \end{equation}
        Differentiation of the free energy gives
        \begin{equation}
            (b) = \frac{\sigma_{i^{\star}} \log N^2}{\beta} \mathbb{E}\Big|F'(0) - \mathbb{E}\big[F'(0)\big]\Big|.
        \end{equation}
        From the convexity of $F$ (see Lemma \ref{lem:IGFF.modified.free.energy.convexity}),
        \begin{align}
            F'(0) - \mathbb{E}\big[F'(0)\big]
            &\leq \frac{F(u) - F(0)}{u} - \mathbb{E}\big[F'(0)\big] \notag \\
            &\leq \left|\frac{\mathbb{E}\big[F(u)\big] - \mathbb{E}\big[F(0)\big]}{u} - \mathbb{E}\big[F'(0)\big]\right| + \frac{\eta(u)}{u}, \\
            F'(0) - \mathbb{E}\big[F'(0)\big]
            &\geq \frac{F(0) - F(-u)}{u} - \mathbb{E}\big[F'(0)\big] \notag \\
            &\geq - \left|\frac{\mathbb{E}\big[F(0)\big] - \mathbb{E}\big[F(-u)\big]}{u} - \mathbb{E}\big[F'(0)\big]\right| - \frac{\eta(u)}{u}.
        \end{align}
        By taking the absolute value and the expectation, we get
        \begin{align}\label{eq:lem:IGFF.ghirlanda.guerra.restricted.2.b.end}
            \frac{\beta}{2 \sigma_{i^{\star}}} \cdot \frac{(b)}{\log N}
            &\leq \left|\frac{\mathbb{E}\big[F(u)\big] - \mathbb{E}\big[F(0)\big]}{u} - \mathbb{E}\big[F'(0)\big]\right| \notag \\
            &\quad+ \left|\frac{\mathbb{E}\big[F(0)\big] - \mathbb{E}\big[F(-u)\big]}{u} - \mathbb{E}\big[F'(0)\big]\right| + \frac{\mathbb{E}\big[\eta(u)\big]}{u}.
        \end{align}
        Recall that $F$ and $\eta$ are functions of $N$ and $\rho$ by definition.
        From Theorem \ref{thm:IGFF.free.energy.A.N.p}, we know that for all $\rho\in (0,1]$ and all $u\in (0,\sigma_{i^{\star}})$,
        \begin{equation}
            \lim_{N\rightarrow \infty} \mathbb{E}\big[\eta(u)\big] = 0.
        \end{equation}
        Using \eqref{eq:convergence.free.energy} and \eqref{eq:convergence.free.energy.derivative} in \eqref{eq:lem:IGFF.ghirlanda.guerra.restricted.2.b.end}, we get, for all $\rho\in (0,1]$ and all $u\in (0,\sigma_{i^{\star}})$,
        \begin{align}
            \limsup_{N\rightarrow \infty} \left\{\frac{\beta}{2 \sigma_{i^{\star}}} \cdot \frac{(b)}{\log N}\right\}
            &\leq \left|\frac{f(u) \hspace{-0.3mm}-\hspace{-0.3mm} f(0)}{u} \hspace{-0.3mm}-\hspace{-0.3mm} f'(0)\right| \notag \\
            &\quad+ \left|\frac{f(0) \hspace{-0.3mm}-\hspace{-0.3mm} f(-u)}{u} \hspace{-0.3mm}-\hspace{-0.3mm} f'(0)\right|,
        \end{align}
        where $f(u) \circeq f^{\psi^u}(\beta)$. Finally, take $u\rightarrow 0^+$ in the last equation, the differentiability of $f$ at $0$ (from Lemma \ref{lem:IGFF.modified.free.energy.derivative}) implies that for all $\rho\in (0,1]$, $\lim_{N\rightarrow \infty} (b) / \log N = 0$.
        This ends the proof of Lemma \ref{lem:IGFF.ghirlanda.guerra.restricted.2}.
    \end{proof}

    Finally, we can prove the approximate extended Ghirlanda-Guerra identities.

    \begin{proof}[Proof of Theorem \ref{thm:IGFF.ghirlanda.guerra}]
        In addition to \eqref{eq:scales.are.points.of.continuity.TOD}, assume that $\lambda_{i^{\star}-1} \hspace{-0.3mm} \leq \hspace{-0.1mm}\alpha \hspace{-0.3mm}<\hspace{-0.1mm} \alpha' \hspace{-0.3mm} \leq \hspace{-0.1mm} \lambda_{i^{\star}}$ for some $i^{\star}$. Also, let $\rho\in (0,\lambda_1)$.
        If we combine Lemma \ref{lem:IGFF.ghirlanda.guerra.restricted.2} and Proposition \ref{prop:BV.technique} with the triangle inequality, we get
        \begin{equation}\label{eq:prop:IGFF.ghirlanda.guerra.restricted.bound.1}
            \hspace{2mm}\left|\hspace{-1mm}
            \begin{array}{l}
            \vspace{2mm}\resizebox{0.29\hsize}{!}{%
                $\frac{\mathbb{E}\mathcal{G}_{\beta\hspace{-0.3mm},N\hspace{-0.6mm},\rho}[\psi_{v^k}(\alpha,\alpha')]}{\beta\, (\mathcal{J}_{\sigma^2}(1) \log N + C_0)}$
            }
            \hspace{-1mm}\mathbb{E}\mathcal{G}_{\beta,N,\rho}^{\times s}\big[h(\bb{v})\big] \\
            \hspace{0mm} - \left\{\hspace{-1.5mm}
                \begin{array}{l}
                    \vspace{0.5mm}\sum_{l = 1}^s \mathbb{E}\mathcal{G}_{\beta,N,\rho}^{\times s}\big[\int_{S_{\alpha,\alpha'}} \hspace{-1mm}\bb{1}_{\{r < q^N(v^k,v^l)\}} dr \, h(\bb{v})\big] \\
                    - s \, \mathbb{E}\mathcal{G}_{\beta,N,\rho}^{\times (s+1)}\big[\int_{S_{\alpha,\alpha'}} \hspace{-1mm}\bb{1}_{\{r < q^N(v^k,v^{s+1})\}} dr \, h(\bb{v})\big]
                \end{array}
                \hspace{-1.5mm}\right\}
            \end{array}\hspace{-1mm}
            \right| \leq \text{RHS}\eqref{eq:BV.technique} + o_N(1),
        \end{equation}
        where $\text{RHS}$ means ``right-hand side of''.
        Furthermore, from Proposition \ref{prop:BV.technique} in the special case ($s=1$, $k = 1$, $h \equiv 1$),
        \begin{equation}\label{eq:prop:IGFF.ghirlanda.guerra.restricted.bound.2}
            \left|
            \begin{array}{l}
            \vspace{2mm}\resizebox{0.29\hsize}{!}{%
                $\frac{\mathbb{E}\mathcal{G}_{\beta\hspace{-0.3mm},N\hspace{-0.6mm},\rho}[\psi_{v^k}(\alpha,\alpha')]}{\beta\, (\mathcal{J}_{\sigma^2}(1) \log N + C_0)}$
            } \\
            \hspace{0mm}- \left\{\hspace{-1mm}
                \begin{array}{l}
                    \vspace{0.5mm}\mathbb{E}\mathcal{G}_{\beta,N,\rho}\big[\int_{S_{\alpha,\alpha'}} \hspace{-1mm}\bb{1}_{\{r < q^N(v^k,v^k)\}} dr\big] \\
                    - \mathbb{E}\mathcal{G}_{\beta,N,\rho}^{\times 2}\big[\int_{S_{\alpha,\alpha'}} \hspace{-1mm}\bb{1}_{\{r < q^N(v^1,v^2)\}} dr\big]
                \end{array}
                \hspace{-1.5mm}\right\}
            \end{array}
            \right| \leq \text{RHS}_{(s=1,h\equiv 1)}\eqref{eq:BV.technique}.
        \end{equation}
        By combining the last two bounds with the triangle inequality, we find
        \begin{align}\label{eq:prop:IGFF.ghirlanda.guerra.restricted.to.prove}
            &\limsup_{N\to\infty} \left|\hspace{-1mm}
            \begin{array}{l}
                \vspace{2mm}\mathbb{E}\mathcal{G}_{\beta,N,\rho}^{\times (s+1)}\big[\int_{S_{\alpha,\alpha'}} \hspace{-1mm} \bb{1}_{\{r < q^N(v^k,v^{s+1})\}} dr ~h(\bb{v})\big] \\
                \hspace{0mm}- \left\{\hspace{-1.5mm}
                \begin{array}{l}
                    \frac{1}{s} \mathbb{E}\mathcal{G}_{\beta,N,\rho}^{\times 2}\big[\int_{S_{\alpha,\alpha'}} \hspace{-1mm}\bb{1}_{\{r < q^N(v^1,v^2)\}} dr\big] \mathbb{E}\mathcal{G}_{\beta,N,\rho}^{\times s}\big[h(\bb{v})\big] \\[1.5mm]
                    + \frac{1}{s} \sum_{l \neq k}^s \mathbb{E}\mathcal{G}_{\beta,N,\rho}^{\times s}\big[\int_{S_{\alpha,\alpha'}} \hspace{-1mm}\bb{1}_{\{r < q^N(v^k,v^l)\}} dr ~h(\bb{v})\big]
                \end{array}
                \hspace{-1.5mm}\right\}
            \end{array}\hspace{-1mm}
            \right| \\[1mm]
            &\leq \widetilde{C} \cdot \sup_N \|h\|_{\infty} \cdot \left\{\hspace{-1mm}
                \begin{array}{l}
                    \limsup_{N\to\infty} \mathbb{E}\mathcal{G}_{\beta,N,\rho}^{\times 2} \big[\bb{1}_{\{\bar{\mathcal{J}}_{\sigma^2}(\alpha) - \varepsilon \leq q^N(v,v') \leq \bar{\mathcal{J}}_{\sigma^2}(\alpha) + \varepsilon\}}\big] \\[1.5mm]
                    \limsup_{N\to\infty}\mathbb{E}\mathcal{G}_{\beta,N,\rho}^{\times 2} \big[\bb{1}_{\{\bar{\mathcal{J}}_{\sigma^2}(\alpha') - \varepsilon \leq q^N(v,v') \leq \bar{\mathcal{J}}_{\sigma^2}(\alpha') + \varepsilon\}}\big] \\[1.5mm]
                    + O(\rho)
                \end{array}
                \hspace{-1mm}\right\}. \notag
        \end{align}
        Using the triangle inequality, Proposition \ref{prop:product.gibbs.near.boundary} and Corollary \ref{cor:two.overlap.distribution.near.boundary} in \eqref{eq:prop:IGFF.ghirlanda.guerra.restricted.to.prove}, it is easy to show that inequality \eqref{eq:prop:IGFF.ghirlanda.guerra.restricted.to.prove} is also true if $\mathcal{G}_{\beta,N,\rho}$ is replaced everywhere by $\mathcal{G}_{\beta,N}$.
        From Theorem \ref{thm:two.overlap.distribution.limit}, condition \eqref{eq:scales.are.points.of.continuity.TOD} guarantees that $\bar{\mathcal{J}}_{\sigma^2}(\alpha)$ and $\bar{\mathcal{J}}_{\sigma^2}(\alpha')$ are continuity points of the limiting two-overlap distribution.
        Thus, after the replacement of the Gibbs measures in \eqref{eq:prop:IGFF.ghirlanda.guerra.restricted.to.prove}, take $\rho\to 0$ and then $\varepsilon\to 0$ to deduce \eqref{eq:thm:IGFF.ghirlanda.guerra.AGGI}.

        If we only assume \eqref{eq:scales.are.points.of.continuity.TOD}, note that $\lambda_{i-1} \hspace{-0.5mm}\leq \hspace{-0.5mm}\alpha \hspace{-0.5mm}< \hspace{-0.5mm}\lambda_{i}$ and $\lambda_{i'-1} \hspace{-0.5mm}< \hspace{-0.5mm}\alpha' \hspace{-0.5mm}\leq \hspace{-0.5mm}\lambda_{i'}$ for some $i,i'$. By the above argument, we have \eqref{eq:thm:IGFF.ghirlanda.guerra.AGGI} for each pair of scales
        \begin{equation}
            \alpha < \lambda_i ~~;~~ \lambda_i < \lambda_{i + 1} ~~; \ldots ;~~ \lambda_{i'- 2} < \lambda_{i' - 1} ~~;~~ \lambda_{i' - 1} < \alpha'.
        \end{equation}
        Add all the limits together and use the triangle inequality to conclude.
    \end{proof}

\newpage
\appendix

\section{Covariance estimates}\label{sec:covariance.estimates}

    The Markov property of the GFF, which is a consequence of the strong Markov property of the simple random walk (in the covariance function in \eqref{eq:GFF.green.function}),
    implies that the value of the field inside a neighborhood is independent of the field outside given the boundary, see e.g.~\citet{MR585179}.
    \hspace{-0.5mm}In particular, for the \hspace{-0.5mm}neighborhood \hspace{-0.5mm}$[v]_\lambda$, this implies
    \begin{equation}\label{eq:GFF.Markov}
        \phi_v(\lambda) \circeq \mathbb{E}\big[\phi_v \nvert \F_{\partial [v]_{\lambda} \cup [v]_{\lambda}^c}\big] = \mathbb{E}\big[\phi_v \nvert \F_{\partial [v]_{\lambda}}\big].
    \end{equation}
    Define the {\it branching scale between $v$ and $v'$} in $V_N$ :
    \begin{equation}\label{eq:rho}
        b_N(v,v') \circeq \max\big\{\lambda\in [0,1] : [v]_\lambda \cap [v']_\lambda \neq \emptyset\big\}.
    \end{equation}
    This is the largest $\lambda$ for which the two neighborhoods $[v]_\lambda$ and $[v']_\lambda$ intersect.
    We always have by definition that $\|v-v'\|_2$ is of order $N^{1-b_N(v,v')}$.
    The branching scale plays the same role as the branching time (normalized to lie in $[0,1]$) in branching random walk.
    Define
    \begin{equation*}
        \varepsilon_N \circeq \frac{\log 4}{\log N}.
    \end{equation*}
    For all $v,v'\in V_N$ ($v \neq v'$), this definition guarantees that for all $N\in \N$,
    \begin{equation}\label{eq:property.figure.overlap.over.under.branching}
        \begin{array}{lll}
            \hspace{4mm}&\hspace{7.5mm}&[v]_{1 \wedge (b_N + \varepsilon_N)} \cap [v']_{1 \wedge (b_N + \varepsilon_N)} = \emptyset \\
            \hspace{4mm}&\text{and}& \\
            \hspace{4mm}&\hspace{7.5mm}&[v]_{b_N} \cup [v']_{b_N} \subseteq [v]_{0 \vee (b_N - \varepsilon_N)} \cap [v']_{0 \vee (b_N - \varepsilon_N)}.
        \end{array}
    \end{equation}
    To convince the reader, see Figure \ref{fig:overlap.over.under.branching} below and note that $N^{\varepsilon_N} = 4$.

    If $\lambda < \lambda'$ and $\mu < \mu'$, a direct consequence of \eqref{eq:property.figure.overlap.over.under.branching} and the Markov property of the GFF is the fact that when
    \begin{equation*}
        \begin{array}{l}
            \hspace{6mm} v \neq v' \quad \text{and} \quad
            \left\{\hspace{-1mm}
            \begin{array}{l}
                \hspace{4.3mm} (1) : \lambda, \mu \geq b_N(v,v') + \varepsilon_N, \\
                \text{or } (2) : \lambda \geq b_N(v,v') + \varepsilon_N > b_N(v,v') - \varepsilon_N \geq \mu', \\
                \text{or } (3) : b_N(v,v') - \varepsilon_N \geq \lambda \geq \mu' + \varepsilon_N,
            \end{array}
            \right. \\
            \text{or } \\
            \hspace{6mm} v = v' \quad \text{and} \quad \lambda \geq \mu'
        \end{array}
    \end{equation*}
    then
    \begin{equation}\label{eq:GFF.Markov.independence}
        \phi_v(\lambda,\lambda') \quad \text{is independent of} \quad \phi_{v'}(\mu,\mu').
    \end{equation}
    This is because the shell $[v]_{\lambda} \cap [v]_{\lambda'}^c$ does not intersect the shell $[v']_{\mu} \cap [v']_{\mu'}^c$ in all cases, see Figure 2.2 in \cite{MR3541850}.
    The ``spacing'' $\varepsilon_N$ is not optimal but sufficient for our purpose.
    We stress that, in general, the field $\psi$ does not have the Markov property.
    However, by working with increments of the field $\psi$, the property analogous to \eqref{eq:GFF.Markov.independence} can be proved.
    The following lemma is a refinement of Lemma A.1 in \cite{MR3541850}, where the error term $\varepsilon_N$ is introduced to make the statement hold for all $N$, not only $N$ large enough.

    \begin{lemma}\label{lem:IGFF.psi.Markov}
        Let $v,v'\in V_N$, $\lambda < \lambda'$, $\mu < \mu'$ and $\varepsilon_N \circeq (\log 4) / (\log N)$. If
        \begin{equation*}
        \begin{array}{l}
            \hspace{6mm} v \neq v' \quad \text{and} \quad
            \left\{\hspace{-1mm}
            \begin{array}{l}
                \hspace{4.4mm} (1) : \lambda, \mu \geq b_N(v,v') + \varepsilon_N, \\
                \text{or } (2) : \lambda \geq b_N(v,v') + \varepsilon_N > b_N(v,v') - \varepsilon_N \geq \mu', \\
                \text{or } (3) : b_N(v,v') - \varepsilon_N \geq \lambda \geq \mu' + \varepsilon_N,
            \end{array}
            \right. \\
            \text{or } \\
            \hspace{6mm} v = v' \quad \text{and} \quad \lambda \geq \mu'
        \end{array}
    \end{equation*}
    then
    \begin{equation}\label{eq:IGFF.Markov.independence}
        \psi_v(\lambda,\lambda') \quad \text{is independent of} \quad \psi_{v'}(\mu,\mu').
    \end{equation}
    \end{lemma}

    \begin{proof}
        Using the tower property of conditional expectations, we have the following decomposition (see (A.4) in \cite{MR3541850}) :
        \begin{equation}\label{eq:martingale.transform.decomposition}
            \psi_v(\lambda,\lambda') = \hspace{-11mm} \sum_{\substack{1 \leq i \leq M : \\ \lambda \leq \lambda_{i-1} < \lambda' \text{ or } \lambda < \lambda_i \leq \lambda' \\ \text{or } \lambda_{i-1} \leq \lambda < \lambda' \leq \lambda_i}} \hspace{-11.5mm} \sigma_i \hspace{0.5mm}\phi_v(\lambda \vee \lambda_{i-1},\lambda' \wedge \lambda_i).
        \end{equation}
        The conclusion follows directly from \eqref{eq:GFF.Markov.independence} above.
    \end{proof}

    \begin{figure}[ht]
        \centering
        \includegraphics[scale=0.65]{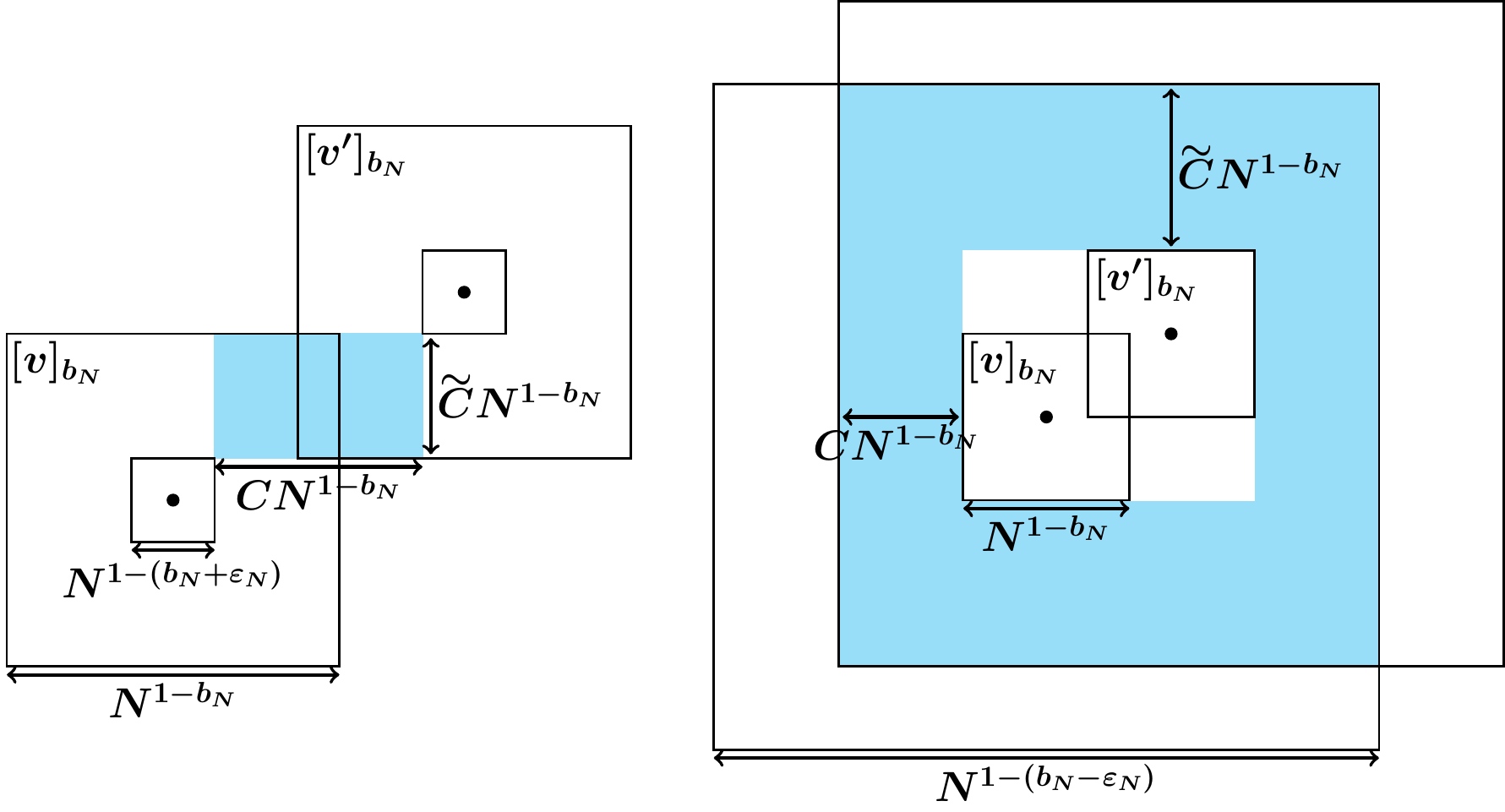}
        \captionsetup{width=0.7\textwidth}
        \caption{Illustration of Equation \eqref{eq:property.figure.overlap.over.under.branching}.}
        \label{fig:overlap.over.under.branching}
    \end{figure}

    The next lemma gives upper and lower bounds on the variance of the increments of the field $\psi$ in $A_{N,\rho}$. Recall from \eqref{eq:A.N.rho} that
    \begin{equation}
        A_{N,\rho} \circeq \Big\{v\in V_N : \min_{z\in \Z^2 \backslash V_N} \|v - z\|_2 \geq N^{1 - \rho}\Big\}, \quad \rho\in (0,1].
    \end{equation}

    \begin{lemma}\label{lem:IGFF.variance.estimates}
        Let $\lambda_{i-1} \hspace{-0.5mm}\leq \hspace{-0.5mm}\alpha \hspace{-0.5mm}< \hspace{-0.5mm}\alpha' \hspace{-0.5mm}\leq \hspace{-0.5mm}\lambda_i$ for some $i\in \{1,...,M\}$, $\alpha \neq 0$ and $\rho\in (0,\alpha]$. Then, for $N$ large enough (dependent on $\alpha$, but independent from $\rho$),
        \begin{equation}\label{eq:variance.estimate}
            \max_{v\in A_{N,\rho}} \big|\mathbb{E}\big[\psi_v(\alpha,\alpha')^2\big] - (\alpha' - \alpha) \sigma_i^2 \log N\big| \leq C \sigma_i^2.
        \end{equation}
    \end{lemma}

    \begin{proof}
        This is Lemma A.2 in \cite{MR3541850} with $v\in A_{N,\rho}$ instead of $v\in V_N^{\delta}$.
        The proof is exactly the same and the constant $C$ is independent of $\alpha$ because $\rho\in (0,\alpha]$ implies that the boxes $[v]_{\alpha}$ are not cut off by $\partial V_N$.
    \end{proof}

    The next lemma shows that the upper bound on the variance of the increments in \eqref{eq:variance.estimate} is in fact uniform on $V_N$.
    We extend the statement to include all combinations of scales $\alpha < \alpha'$.

    \begin{lemma}\label{lem:IGFF.variance.uniform.bound}
        There exists a constant $C_0 = C_0(\boldsymbol{\sigma}) > 0$ such that for all scales $0 \leq \alpha < \alpha' \leq 1$ and $N$ large enough (independent from $\alpha$ and $\alpha'$),
        \begin{equation}
            \max_{v\in V_N} \mathbb{E}\big[\psi_v(\alpha,\alpha')^2\big] \leq \mathcal{J}_{\sigma^2}(\alpha,\alpha') \log N + C_0.
        \end{equation}
    \end{lemma}

    \begin{proof}
        This follows immediately from Lemma A.3 in \cite{MR3541850} and the independence of the increments.
    \end{proof}

    In Section \ref{sec:bovier.kurkova.technique.general}, estimates on the covariance of the increments are needed to bound certain overlaps and adapt the Bovier-Kurkova technique.
    The next two lemmas take care of this problem.
    To simplify the notation, define
    \begin{align}
        &\phi_v(A) \circeq \esp{}{\phi_v \nvert \F_{\partial (A \cap V_N)}}, \\
        &\phi_v(A_1,A_2) \circeq \phi_v(A_2) - \phi_v(A_1),
    \end{align}
    for any sets $A,A_1,A_2\subseteq \Z^2$. With this notation, we can also mix sets and scales with the obvious meaning. For example,
    \begin{equation}
        \phi_v(A,\lambda) \circeq \phi_v(\lambda) - \phi_v(A).
    \end{equation}
    For simplicity, we write $b_N$ instead of $b_N(v,v')$ in the remaining of this section.

    \begin{lemma}\label{lem:covariance.estimates}
        Let $\lambda_{i-1} \hspace{-0.5mm}\leq \hspace{-0.5mm}\alpha \hspace{-0.5mm}< \hspace{-0.5mm}\alpha' \hspace{-0.5mm}\leq \hspace{-0.5mm}\lambda_i$ for some $i\in \{1,...,M\}$, $\alpha \neq 0$, $\rho\in (0,\alpha/2],$ and $\varepsilon_N \circeq (\log 4) / (\log N)$. All four equations below hold for $N$ large enough (dependent on $\alpha$ and $\alpha'$\hspace{-0.5mm}, but independent from $\rho$ and $v,v'$).
        All the constants $C_i, ~1 \leq i \leq 4$, depend only on $\para$.

        \vspace{3mm}
        \noindent
        For all $v,v'\in A_{N,\rho}$ such that $1 \wedge (\alpha' + 2\varepsilon_N) \leq b_N \leq 1$,
        \begin{equation}\label{eq:covariance.estimate.0}
            \big|\mathbb{E}\big[\psi_v(\alpha,\alpha') \psi_{v'}\big] - (\alpha' - \alpha) \sigma_i^2 \log N\big| \leq C_1 \sqrt{\log N}.
        \end{equation}
        For all $v,v'\in A_{N,\rho}$ such that $\alpha' - 2\varepsilon_N \leq b_N \leq 1 \wedge (\alpha' + 2\varepsilon_N)$,
        \begin{equation}\label{eq:covariance.estimate.1}
            \big|\mathbb{E}\big[\psi_v(\alpha,\alpha') \psi_{v'}\big] - (\alpha' - \alpha) \sigma_i^2 \log N\big| \leq C_2 \sqrt{\log N}.
        \end{equation}
        For all $v,v'\in A_{N,\rho}$ such that $\alpha + 2\varepsilon_N \leq b_N \leq \alpha' - 2\varepsilon_N$,
        \begin{equation}\label{eq:covariance.estimate.2}
            \big|\mathbb{E}\big[\psi_v(\alpha,\alpha') \psi_{v'}\big] - \left(b_N - \alpha\right) \sigma_i^2 \log N\big| \leq C_3 \sqrt{\log N}.
        \end{equation}
        For all $v,v'\in V_N$ such that $b_N \leq \alpha + 2\varepsilon_N$,
        \begin{equation}\label{eq:covariance.estimate.3}
            \big|\mathbb{E}\big[\psi_v(\alpha,\alpha') \psi_{v'}\big]\big| \leq C_4 \sqrt{\log N}.
        \end{equation}
    \end{lemma}

    \begin{proof}[Proof of Equation \eqref{eq:covariance.estimate.0}]
        Let $v,v'\in A_{N,\rho}$ be such that $1 \wedge (\alpha' + 2\varepsilon_N) \leq b_N \leq 1$.
        The case $b_N = 1$ (i.e. $v = v'$) is covered by Lemma \ref{lem:IGFF.variance.estimates}.
        Therefore, assume
        \begin{equation*}
            \alpha' + 2\varepsilon_N \leq b_N < 1.
        \end{equation*}
        From $(1)-(3)$ in Lemma \ref{lem:IGFF.psi.Markov} :
        \begin{equation}
            \begin{aligned}
                &(2) : ~~\mathbb{E}\big[\psi_v(\alpha,\alpha') \psi_{v'}(1 \wedge (b_N + \varepsilon_N),1)\big] = 0, \\
                &(3) : ~~\mathbb{E}\big[\psi_v(\alpha,\alpha') \psi_{v'}(\alpha' + \varepsilon_N,b_N - \varepsilon_N)\big] = 0, \\
                &(3) : ~~\mathbb{E}\big[\psi_v(\alpha,\alpha') \psi_{v'}(\alpha - \varepsilon_N)\big] = 0.
            \end{aligned}
        \end{equation}
        Moreover, by the Cauchy-Schwarz inequality and Lemma \ref{lem:IGFF.variance.uniform.bound},
        \begin{align}
            \left.
            \begin{array}{ll}
                \vspace{2mm}&\big|\mathbb{E}\big[\psi_v(\alpha,\alpha') \psi_{v'}(b_N - \varepsilon_N,1 \wedge (b_N + \varepsilon_N))\big]\big| \\
                \vspace{2mm}&\big|\mathbb{E}\big[\psi_v(\alpha,\alpha') \psi_{v'}(\alpha',\alpha' + \varepsilon_N)\big]\big| \\
                &\big|\mathbb{E}\big[\psi_v(\alpha,\alpha') \psi_{v'}(\alpha - \varepsilon_N,\alpha)\big]\big|
            \end{array}
            \hspace{-1mm}\right\} \leq C \sqrt{\varepsilon_N} \log N.
        \end{align}
        From the last six equations, it thus suffices to prove
        \begin{equation}\label{eq:covariance.estimate.1.to.prove}
            \big|\mathbb{E}\big[\psi_v(\alpha,\alpha') \psi_{v'}(\alpha,\alpha')\big] - (\alpha' - \alpha) \sigma_i^2 \log N\big| \leq C \sqrt{\log N}.
        \end{equation}
        But, from Definition \ref{def:IGFF} and the tower property of conditional expectations, it is easily shown (see \eqref{eq:martingale.transform.decomposition}) that when $\lambda_{i-1} \hspace{-0.5mm}\leq \alpha \hspace{-0.5mm}< \alpha' \hspace{-0.5mm}\leq \lambda_i$,
        \begin{equation}\label{eq:covariance.estimate.1.martingale.transform}
            \psi_u(\alpha,\alpha') = \sigma_i \phi_u(\alpha,\alpha'), \quad u\in V_N.
        \end{equation}
        Therefore, to show \eqref{eq:covariance.estimate.1.to.prove}, it suffices to prove
        \begin{equation}\label{eq:covariance.estimate.1.to.prove.2}
            \big|\mathbb{E}\big[\phi_v(\alpha,\alpha') \phi_{v'}(\alpha,\alpha')\big] - (\alpha' - \alpha) \log N\big| \leq C \sqrt{\log N}.
        \end{equation}
        Since $b_N \geq \alpha' + 2\varepsilon_N$ by hypothesis, we have
        \begin{equation}\label{eq:covariance.estimate.1.boxes.distance}
            [v]_{\alpha} \cup [v']_{\alpha}\subseteq [v]_{\alpha - \varepsilon_N} \quad \text{and} \quad [v]_{\alpha'} \cup [v']_{\alpha'}\subseteq [v]_{\alpha' - \varepsilon_N}.
        \end{equation}
        From \eqref{eq:covariance.estimate.1.boxes.distance} and Lemma A.5 in \cite{MR3541850}, we deduce
        \begin{equation}
            \mathbb{E}\big[\phi_u(\lambda,[v]_{\lambda - \varepsilon_N})^2\big] \leq C, \quad \text{for all } u\in \{v,v'\}, ~\lambda\in \{\alpha,\alpha'\}.
        \end{equation}
        By combining these four inequalities in \eqref{eq:covariance.estimate.1.to.prove.2} with the Cauchy-Schwarz inequality and Lemma \ref{lem:IGFF.variance.uniform.bound}, it suffices to prove
        \begin{align}\label{eq:covariance.estimate.1.to.prove.3}
            &\big|\mathbb{E}\big[\phi_v([v]_{\alpha - \varepsilon_N},[v]_{\alpha' - \varepsilon_N}) \phi_{v'}([v]_{\alpha - \varepsilon_N},[v]_{\alpha' - \varepsilon_N})\big] - (\alpha' - \alpha) \log N\big| \leq C.
        \end{align}
        For $u\in \{v,v'\}$, the Markov property \eqref{eq:GFF.Markov} yields
        \begin{equation}\label{eq:covariance.estimate.1.markov.property.yield}
            \mathbb{E}\big[\phi_u([v]_{\alpha - \varepsilon_N},1) \nvert \F_{\partial [v]_{\alpha' - \varepsilon_N}}\big] = \phi_u([v]_{\alpha - \varepsilon_N},[v]_{\alpha' - \varepsilon_N}).
        \end{equation}
        Using $(\clubsuit) : \esp{}{\esp{}{X\nvert\F}\esp{}{Y\nvert\F}} = \esp{}{XY} - \esp{}{(X - \esp{}{X\nvert \F})(Y - \esp{}{Y\nvert \F})}$ together with \eqref{eq:covariance.estimate.1.markov.property.yield}, we can compute the covariance in \eqref{eq:covariance.estimate.1.to.prove.3} :
        \begin{align}\label{eq:covariance.estimate.1.before}
            &\mathbb{E}\big[\phi_v([v]_{\alpha - \varepsilon_N},[v]_{\alpha' - \varepsilon_N}) \phi_{v'}([v]_{\alpha - \varepsilon_N},[v]_{\alpha' - \varepsilon_N})\big] \notag \\[2.5pt]
            &\stackrel{\eqref{eq:covariance.estimate.1.markov.property.yield}}{=} \mathbb{E}\Big[\mathbb{E}\big[\phi_v([v]_{\alpha - \varepsilon_N},1) \nvert \F_{\partial [v]_{\alpha' - \varepsilon_N}}\big] \mathbb{E}\big[\phi_{v'}([v]_{\alpha - \varepsilon_N},1) \nvert \F_{\partial [v]_{\alpha' - \varepsilon_N}}\big]\Big] \notag \\
            &\hspace{1.6mm}\stackrel{(\clubsuit)}{=} \hspace{1.8mm} \mathbb{E}\big[\phi_v([v]_{\alpha - \varepsilon_N},1) \phi_{v'}([v]_{\alpha - \varepsilon_N},1)\big] \notag \\[1pt]
            &\hspace{12mm}- \mathbb{E}\big[\phi_v([v]_{\alpha' - \varepsilon_N},1) \phi_{v'}([v]_{\alpha' - \varepsilon_N},1)\big].
        \end{align}
        But, it is well known that $\{\phi_u(B,1)\}_{u\in B}$ is a GFF on $B$ when $B\subseteq \Z^2$ is a finite box, see e.g. \citet{Zeitouni2014ln}.
        Simply choose $B = [v]_{\lambda - \varepsilon_N}$, $\lambda = \alpha,\alpha'$, in \eqref{eq:covariance.estimate.1.before}, then by the covariance definition in \eqref{eq:GFF.green.function},
        \begin{equation}\label{eq:tech.lemma.1.green.function.difference}
            \eqref{eq:covariance.estimate.1.before} = G_{[v]_{\alpha - \varepsilon_N}}(v,v') - G_{[v]_{\alpha' - \varepsilon_N}}(v,v')\ .
        \end{equation}
        Using standard estimates for the discrete Green function, we can now evaluate the last expression.
        For every finite box $B\subseteq \Z^2$, Proposition 1.6.3 of \citet{MR1117680} shows that (keeping in mind our normalization by $\pi/2$ in \eqref{eq:GFF.green.function}) :
        \begin{equation}\label{eq:lawler.green.function}
            G_B(x,y) = \left[\sum_{z\in \partial B} \hspace{-0.5mm}\probw{x}{W_{\tau_{\partial B}} = z} a(z - y)\right] - a(y - x), \quad x,y\in B,
        \end{equation}
        where
        \begin{equation}\label{eq:lawler.estime.noyau.potentiel.1}
            a(w) =
            \left\{\hspace{-1mm}
            \begin{array}{ll}
                \log(\|w\|_2) + \text{const.} + O(\|w\|_2^{-2}), ~&\mbox{if } w\in \Z^2\backslash\{\boldsymbol{0}\}, \\
                0, ~&\mbox{if } w = \boldsymbol{0},
            \end{array}
            \right.
        \end{equation}
        and $\mathscr{P}_x$ is the law of the simple random walk starting at $x\in \Z^2$. Using \eqref{eq:lawler.green.function}, we can rewrite the difference of Green functions in \eqref{eq:tech.lemma.1.green.function.difference} as
        \begin{equation}\label{eq:lawler.green.function.difference}
            \hspace{-3mm}\sum_{\hspace{3mm}z\in \partial [v]_{\alpha - \varepsilon_N}} \hspace{-6.5mm}\mathscr{P}_v\Big(W_{\tau_{\partial [v]_{\alpha - \varepsilon_N}}} \hspace{-2mm}= z\Big) a(z - v') \hspace{0.5mm}- \hspace{-7.5mm}\sum_{\hspace{4mm}z\in \partial [v]_{\alpha' - \varepsilon_N}} \hspace{-7mm}\mathscr{P}_v\Big(W_{\tau_{\partial [v]_{\alpha' - \varepsilon_N}}} \hspace{-2mm}= z\Big) a(z - v').
        \end{equation}
        Since $\rho \leq \alpha/2 < \alpha - \varepsilon_N < \alpha' - \varepsilon_N$ by hypothesis, the boxes $[v]_{\alpha - \varepsilon_N}$ and $[v]_{\alpha' - \varepsilon_N}$ are not cut off by $\partial V_N$ for $N$ large enough. Furthermore, $\alpha' \leq b_N$ implies that $\|v - v'\|_{\infty} \leq N^{1-\alpha'}$, so it is easily seen that $N^{1-\lambda} \leq \|z - v'\|_2 \leq 4\sqrt{2} N^{1 - \lambda}$ for all $z\in \partial [v]_{\lambda - \varepsilon_N}$ and $\lambda\in \{\alpha,\alpha'\}$.
        Then, \eqref{eq:covariance.estimate.1.to.prove.3} follows immediately by using \eqref{eq:lawler.estime.noyau.potentiel.1} in \eqref{eq:lawler.green.function.difference}.
        This proves Equation \eqref{eq:covariance.estimate.0}.
    \end{proof}

    \begin{proof}[Proof of Equation \eqref{eq:covariance.estimate.1}]
        Let $v,v'\in A_{N,\rho}$ be such that
        \begin{equation}
            \alpha' - 2\varepsilon_N \leq b_N \leq 1 \wedge (\alpha' + 2\varepsilon_N).
        \end{equation}
        Define $\widetilde{\alpha}' \circeq \alpha' - 4\varepsilon_N$.
        For $N$ large enough (independent from $v,v'$ and $\rho$), we have $\lambda_{i-1} \leq \alpha < \widetilde{\alpha}' < \alpha' \leq \lambda_i$ and $1 \wedge (\widetilde{\alpha}' + 2\varepsilon_N) \leq b_N \leq 1$. From Equation \eqref{eq:covariance.estimate.0},
        \begin{equation}
            \big|\mathbb{E}\big[\psi_v(\alpha,\widetilde{\alpha}') \psi_{v'}\big] - (\alpha' - \alpha - 4\varepsilon_N) \sigma_i^2 \log N\big| \leq C_1 \sqrt{\log N},
        \end{equation}
        and from the Cauchy-Schwarz inequality and Lemma \ref{lem:IGFF.variance.uniform.bound},
        \begin{equation}
            \big|\mathbb{E}\big[\psi_v(\widetilde{\alpha}',\alpha') \psi_{v'}\big]\big| \leq C \sqrt{\varepsilon_N} \log N.
        \end{equation}
        This proves Equation \eqref{eq:covariance.estimate.1}.
    \end{proof}

    \begin{proof}[Proof of Equation \eqref{eq:covariance.estimate.2}]
        Let $v,v'\in A_{N,\rho}$ be such that $\alpha + 2\varepsilon_N \leq b_N \leq \alpha' - 2\varepsilon_N$.
        From $(1)-(3)$ in Lemma \ref{lem:IGFF.psi.Markov} :
        \begin{equation}
            \begin{aligned}
            &(1) : ~~\mathbb{E}\big[\psi_v(b_N + \varepsilon_N,\alpha') \psi_{v'}(b_N + \varepsilon_N,1)\big] = 0, \\[0mm]
            &(2) : ~~\mathbb{E}\big[\psi_v(\alpha,b_N - \varepsilon_N) \psi_{v'}(b_N + \varepsilon_N,1)\big] = 0, \\[0mm]
            &(2) : ~~\mathbb{E}\big[\psi_v(b_N + \varepsilon_N,\alpha') \psi_{v'}(\alpha,b_N - \varepsilon_N)\big] = 0, \\[0mm]
            &(2) : ~~\mathbb{E}\big[\psi_v(b_N + \varepsilon_N,\alpha') \psi_{v'}(\alpha - \varepsilon_N)\big] = 0, \\[0mm]
            &(3) : ~~\mathbb{E}\big[\psi_v(\alpha,b_N - \varepsilon_N) \psi_{v'}(\alpha - \varepsilon_N)\big] = 0.
            \end{aligned}
        \end{equation}
        Moreover, by the Cauchy-Schwarz inequality and Lemma \ref{lem:IGFF.variance.uniform.bound},
        \begin{align}
            \left.
            \begin{array}{ll}
                \vspace{2mm}&\big|\mathbb{E}\big[\psi_v(b_N - \varepsilon_N,b_N + \varepsilon_N) \psi_{v'}(b_N + \varepsilon_N,1)\big]\big| \\[0mm]
                \vspace{2mm}&\big|\mathbb{E}\big[\psi_v(\alpha,\alpha') \psi_{v'}(b_N,b_N + \varepsilon_N)\big]\big| \\[0mm]
                \vspace{2mm}&\big|\mathbb{E}\big[\psi_v(b_N,\alpha') \psi_{v'}(b_N - \varepsilon_N,b_N)\big]\big| \\[0mm]
                \vspace{2mm}&\big|\mathbb{E}\big[\psi_v(b_N,b_N + \varepsilon_N) \psi_{v'}(\alpha,b_N - \varepsilon_N)\big]\big| \\[0mm]
                \vspace{2mm}&\big|\mathbb{E}\big[\psi_v(\alpha,\alpha') \psi_{v'}(\alpha - \varepsilon_N,\alpha)\big]\big| \\[0mm]
                &\big|\mathbb{E}\big[\psi_v(b_N - \varepsilon_N,b_N + \varepsilon_N) \psi_{v'}(\alpha - \varepsilon_N)\big]\big|
            \end{array}
            \hspace{-1mm}\right\} \leq C \sqrt{\varepsilon_N} \log N.
        \end{align}
        From the last eleven equations, it thus suffices to prove
        \begin{equation}
            \big|\mathbb{E}\big[\psi_v(\alpha,b_N) \psi_{v'}(\alpha,b_N)\big] - \left(b_N - \alpha\right) \sigma_i^2 \log N\big| \leq C \sqrt{\log N}.
        \end{equation}
        The conclusion follows from the exact same argument used after \eqref{eq:covariance.estimate.1.to.prove} in the proof of Equation \eqref{eq:covariance.estimate.0}, with $b_N$ replacing $\alpha'$ everywhere.
    \end{proof}

    \begin{proof}[Proof of Equation \eqref{eq:covariance.estimate.3}]
        Let $v,v'\in V_N$ be such that $b_N \leq \alpha + 2\varepsilon_N \leq \alpha' - 2\varepsilon_N$.
        From $(1)-(3)$ in Lemma \ref{lem:IGFF.psi.Markov} :
        \begin{equation}
            \begin{aligned}
            &(1) : ~~\mathbb{E}\big[\psi_v(\alpha + 3\varepsilon_N,\alpha') \psi_{v'}(\alpha + 3\varepsilon_N,1)\big] = 0, \\[0mm]
            &(1) : ~~\mathbb{E}\big[\psi_v(\alpha + 3\varepsilon_N,\alpha') \psi_{v'}(\alpha \wedge (b_N + \varepsilon_N),\alpha)\big] = 0, \\[0mm]
            &(2) : ~~\mathbb{E}\big[\psi_v(\alpha + 3\varepsilon_N,\alpha') \psi_{v'}(\alpha \wedge (0 \vee (b_N - \varepsilon_N)))\big] = 0.
            \end{aligned}
        \end{equation}
        Moreover, by the Cauchy-Schwarz inequality and Lemma \ref{lem:IGFF.variance.uniform.bound},
        \begin{align*}
            \hspace{-2mm}\left.
            \begin{array}{ll}
                \vspace{2mm}&\big|\mathbb{E}\big[\psi_v(\alpha + 3\varepsilon_N,\alpha') \psi_{v'}(\alpha,\alpha + 3\varepsilon_N)\big]\big| \\[0mm]
                \vspace{2mm}&\big|\mathbb{E}\big[\psi_v(\alpha + 3\varepsilon_N,\alpha') \psi_{v'}(\alpha \wedge (0 \vee (b_N - \varepsilon_N)),\alpha \wedge (b_N + \varepsilon_N))\big]\big| \\[0mm]
                &\big|\mathbb{E}\big[\psi_v(\alpha,\alpha + 3\varepsilon_N) \psi_{v'}\big]\big|
            \end{array}
            \hspace{-1mm}\right\} \leq C \sqrt{\varepsilon_N} \log N.
        \end{align*}
        The last six equations together yield Equation \eqref{eq:covariance.estimate.3}.
    \end{proof}

    We summarize the results of the previous lemma and extend the statement to include all combinations of scales $\alpha < \alpha'$ and all $\rho\in (0,1]$.

    \begin{lemma}\label{lem:covariance.estimates.2}
        Let $0 \hspace{-0.5mm}\leq \hspace{-0.5mm}\alpha \hspace{-0.5mm}< \hspace{-0.5mm}\alpha' \hspace{-0.5mm}\leq \hspace{-0.5mm}1$ and let $\rho\in (0,1]$. Then, for $N$ large enough (dependent on $\alpha$ and $\alpha'$\hspace{-0.5mm}, but independent from $\rho$ (except when $\alpha = 0$)),
        \begin{align}\label{eq:covariance.estimate.2.1}
            &\max_{v,v'\in A_{N,\rho}} \hspace{-1mm}\big|\mathbb{E}\big[\psi_v(\alpha,\alpha') \psi_{v'}\big] - \mathcal{J}_{\sigma^2}(\alpha \wedge b_N,\alpha' \wedge b_N) \log N\big| \notag \\
            &\hspace{30mm}\leq C_5\para \sqrt{\log N} + C_6(\alpha,\alpha',\boldsymbol{\sigma},\boldsymbol{\lambda})\, \rho \log N.
        \end{align}
    \end{lemma}

    \begin{proof}
        If $\alpha \neq 0$ and $\rho \leq \alpha/2$, then write the decomposition from \eqref{eq:martingale.transform.decomposition},
        \begin{equation}
            \psi_v(\alpha,\alpha') = \hspace{-11mm} \sum_{\substack{1 \leq i \leq M : \\ \alpha \leq \lambda_{i-1} < \alpha' \text{ or } \alpha < \lambda_i \leq \alpha' \\ \text{or } \lambda_{i-1} \leq \alpha < \alpha' \leq \lambda_i}} \hspace{-11.5mm} \psi_v(\alpha \vee \lambda_{i-1},\alpha' \wedge \lambda_i),
        \end{equation}
        and apply Lemma \ref{lem:covariance.estimates} to each increment ($C_6 = 0$).
        If $\alpha \neq 0$ and $\rho > \alpha/2$, or if $\alpha = 0$ and $\rho \geq \alpha'/2$, then simply choose $C_6$ big enough (depending on $\alpha$ or $\alpha'$) that \eqref{eq:covariance.estimate.2.1} is satisfied. This is always possible since $\mathcal{J}_{\sigma^2}(\cdot,\cdot)$ is bounded and since
        \begin{equation}
            \max_{v,v'\in V_N} \hspace{-1mm}\frac{\big|\mathbb{E}\big[\psi_v(\alpha,\alpha') \psi_{v'}\big]\big|}{\log N} \leq C,
        \end{equation}
        by Lemma \ref{lem:IGFF.variance.uniform.bound}. Finally, if $\alpha = 0$ and $\rho < \alpha'/2$, then define $\widetilde{\alpha} \circeq 2\rho$ and apply \eqref{eq:covariance.estimate.2.1} in the first case ($0 \neq \widetilde{\alpha} < \alpha'$ and $\rho \leq \widetilde{\alpha}/2$), we have
        \begin{equation}\label{eq:lem:covariance.estimates.2.end}
            \max_{v,v'\in A_{N,\rho}} \hspace{-1mm}\big|\mathbb{E}\big[\psi_v(\widetilde{\alpha},\alpha') \psi_{v'}\big] - \mathcal{J}_{\sigma^2}(\widetilde{\alpha} \wedge b_N,\alpha' \wedge b_N) \log N\big| \leq C_5\para \sqrt{\log N}.
        \end{equation}
        On the other hand, if we ``cut'' the increments with small covariance contributions like we did multiple times in the proof of the previous lemma (using Lemma \ref{lem:IGFF.psi.Markov}, Lemma \ref{lem:IGFF.variance.uniform.bound} and the Cauchy-Schwarz inequality), then
        \begin{align}\label{eq:lem:covariance.estimates.2.end.2}
            \max_{v,v'\in V_N} \big|\mathbb{E}\big[\psi_v(\widetilde{\alpha}) \psi_{v'}\big]\big|
            &\leq \max_{v,v'\in V_N} \big|\mathbb{E}\big[\psi_v(\widetilde{\alpha} \wedge b_N) \psi_{v'}(\widetilde{\alpha} \wedge b_N)\big]\big| + C \sqrt{\varepsilon_N} \log N \notag \\
            &\leq \widetilde{C}\, (\widetilde{\alpha} \wedge b_N) \log N + C_0 + C \sqrt{\varepsilon_N} \log N \notag \\
            &\leq \widetilde{C}\, \rho \log N + C \sqrt{\log N}.
        \end{align}
        Combining \eqref{eq:lem:covariance.estimates.2.end} and \eqref{eq:lem:covariance.estimates.2.end.2} proves \eqref{eq:covariance.estimate.2.1} in the last case.
    \end{proof}

    The following corollary gives estimates on the increments of overlaps.
    For convenience, we recall their definition from \eqref{eq:overlap.increments} :
    \begin{equation}
        q_{\alpha,\alpha'}^N(v,v') \circeq \frac{\mathbb{E}\big[\psi_v(\alpha,\alpha')\psi_{v'}\big]}{\mathcal{J}_{\sigma^2}(1) \log N + C_0}, \quad v,v'\in V_N,
    \end{equation}
    where $C_0$ is the constant introduced in Lemma \ref{lem:IGFF.variance.uniform.bound}.
    The estimates are crucial in Section \ref{sec:bovier.kurkova.technique.general} to adapt the Bovier-Kurkova technique and prove Proposition \ref{prop:BV.technique}.

    \begin{corollary}\label{cor:covariance.estimates.3}
        Let $0 \hspace{-0.5mm}\leq \hspace{-0.5mm}\alpha \hspace{-0.5mm}< \hspace{-0.5mm}\alpha' \hspace{-0.5mm}\leq \hspace{-0.5mm}1$ and let $\rho\in (0,1]$. Then, for $N$ large enough (dependent on $\alpha$ and $\alpha'$\hspace{-0.5mm}, but independent from $\rho$ (except when $\alpha = 0$)),
        \begin{align}\label{eq:covariance.estimate.3.2}
            &\max_{v,v'\in A_{N,\rho}} \hspace{-1mm}\big|q_{\alpha,\alpha'}^N(v,v') - \bar{\mathcal{J}}_{\sigma^2}(\alpha \wedge b_N,\alpha' \wedge b_N)\big| \notag \\
            &\hspace{30mm}\leq \frac{C_7\para}{\sqrt{\log N}} + C_8(\alpha,\alpha',\boldsymbol{\sigma},\boldsymbol{\lambda}) \, \rho.
        \end{align}
    \end{corollary}

\newpage
\section{Technical lemmas}\label{sec:technical.lemmas}

    \begin{lemma}\label{lem:tech.lemma.1}
        The function $\mathcal{E} : [0,\gamma^{\star}] \rightarrow \R$ defined in \eqref{eq:entropy} is in $C^1([0,\gamma^{\star}])$.
    \end{lemma}

    \begin{proof}
        The function $\mathcal{E}$ is clearly continuously differentiable at $\gamma\in [0,\gamma^{\star}] \backslash \{\gamma^l\}_{l=0}^m$.
        Furthermore, for $0 < h < \gamma^1$,
        \begin{equation}
            \lim_{h\rightarrow 0^+} \frac{\mathcal{E}(h) - \mathcal{E}(0)}{h} = \lim_{h\rightarrow 0^+} \frac{-h}{\mathcal{J}_{\sigma^2}(1)} = 0 = \lim_{h\rightarrow 0^+} \frac{-2 h}{\mathcal{J}_{\sigma^2}(1)} = \lim_{h\rightarrow 0^+} \mathcal{E}'(h).
        \end{equation}
        Therefore, $\mathcal{E}$ is continuously differentiable at $\gamma = 0 \circeq \gamma^0$ (from the right).

        For $\gamma = \gamma^{\star}$, we can write
        \begin{equation}
            \gamma^{\star} = \gamma^m = \mathcal{J}_{\sigma^2\hspace{-0.3mm} / \bar{\sigma}}(\lambda^{m-1}) + \frac{\mathcal{J}_{\sigma^2}(\lambda^{m-1},1)}{\bar{\sigma}_m},
        \end{equation}
        where $\mathcal{J}_{\sigma^2}(\lambda^{m-1},1) = \bar{\sigma}_m^2 \nabla \lambda^m$.
        Thus, for $0 < h < \nabla \gamma^m$,
        \begin{equation}
            \lim_{h\rightarrow 0^+} \frac{\mathcal{E}(\gamma^{\star} - h) - \mathcal{E}(\gamma^{\star})}{-h} = \lim_{h\rightarrow 0^+} \frac{-1}{h} \left[\nabla \lambda^m - \frac{(\bar{\sigma}_m \nabla \lambda^m - h)^2}{\bar{\sigma}_m^2 \nabla \lambda^m}\right] = \frac{-2}{\bar{\sigma}_m},
        \end{equation}
        and
        \begin{equation}
            \lim_{h\rightarrow 0^+} \mathcal{E}'(\gamma^{\star} - h) = \lim_{h\rightarrow 0^+} \frac{-2(\bar{\sigma}_m \nabla \lambda^m - h)}{\bar{\sigma}_m^2 \nabla \lambda^m} = \frac{-2}{\bar{\sigma}_m}.
        \end{equation}
        Therefore, $\mathcal{E}$ is continuously differentiable at $\gamma = \gamma^{\star}$ (from the left).

        For the remaining points $\gamma = \gamma^l$, fix $l\in \{1,...,m-1\}$. The critical level $\gamma^l$ from \eqref{eq:critic.levels} can be expressed in two ways :
        \begin{align}
            \gamma^l
            &= \mathcal{J}_{\sigma^2\hspace{-0.3mm} / \bar{\sigma}}(\lambda^{l-1}) + \frac{\mathcal{J}_{\sigma^2}(\lambda^{l-1},1)}{\bar{\sigma}_l} \label{eq:critic.level.disc.1} \\
            &= \mathcal{J}_{\sigma^2\hspace{-0.3mm} / \bar{\sigma}}(\lambda^l) + \frac{\mathcal{J}_{\sigma^2}(\lambda^l,1)}{\bar{\sigma}_l} \label{eq:critic.level.disc.2} \ .
        \end{align}
        Also, note that
        \begin{align}
            \mathcal{E}(\gamma^l)
            &\stackrel{\eqref{eq:critic.level.disc.1}}{=} (1 - \lambda^{l-1}) - \frac{\mathcal{J}_{\sigma^2}(\lambda^{l-1},1)}{\bar{\sigma}_l^2} \label{eq:entropy.eq.1} \\
            &\stackrel{\phantom{\eqref{eq:critic.level.disc.1}}}{=} (1 - \lambda^l) - \frac{\mathcal{J}_{\sigma^2}(\lambda^l,1)}{\bar{\sigma}_l^2} \label{eq:entropy.eq.2},
        \end{align}
        where the last equality follows from $\mathcal{J}_{\sigma^2}(\lambda^{l-1},\lambda^l) = \bar{\sigma}_l^2 \nabla \lambda^l$. For $0 < h < \min_j \nabla \gamma^j$,
        \begin{align}
            &\lim_{h\rightarrow 0^+} \frac{\mathcal{E}(\eqref{eq:critic.level.disc.1} - h) - \mathcal{E}(\eqref{eq:critic.level.disc.1})}{-h}
            \stackrel{\eqref{eq:entropy.eq.1}}{=} \lim_{h\rightarrow 0^+} \frac{+h}{\mathcal{J}_{\sigma^2}(\lambda^{l-1},1)} + \frac{-2}{\bar{\sigma}_l} = \frac{-2}{\bar{\sigma}_l}, \\
            &\lim_{h\rightarrow 0^+} \frac{\mathcal{E}(\eqref{eq:critic.level.disc.2} + h) - \mathcal{E}(\eqref{eq:critic.level.disc.2})}{h}
            \stackrel{\eqref{eq:entropy.eq.2}}{=} \lim_{h\rightarrow 0^+} \frac{-h}{\mathcal{J}_{\sigma^2}(\lambda^l,1)} + \frac{-2}{\bar{\sigma}_l} = \frac{-2}{\bar{\sigma}_l}.
        \end{align}
        and
        \begin{align}
            &\lim_{h\rightarrow 0^+} \mathcal{E}'(\gamma^l - h) \stackrel{\eqref{eq:critic.level.disc.1}}{=} \lim_{h\rightarrow 0^+} \frac{-2(\frac{\mathcal{J}_{\sigma^2}(\lambda^{l-1},1)}{\bar{\sigma}_l} - h)}{\mathcal{J}_{\sigma^2}(\lambda^{l-1},1)} = \frac{-2}{\bar{\sigma}_l} \\
            &\lim_{h\rightarrow 0^+} \mathcal{E}'(\gamma^l + h) \stackrel{\eqref{eq:critic.level.disc.2}}{=} \lim_{h\rightarrow 0^+} \frac{-2(\frac{\mathcal{J}_{\sigma^2}(\lambda^l,1)}{\bar{\sigma}_l} + h)}{\mathcal{J}_{\sigma^2}(\lambda^l,1)} = \frac{-2}{\bar{\sigma}_l}.
        \end{align}
        Hence, $\mathcal{E}$ is continuously differentiable at $\gamma = \gamma^l$, for all $l\in \{1,...,m-1\}$.
        This ends the proof of the lemma.
    \end{proof}

    \begin{lemma}\label{lem:tech.lemma.2}
        Let $\beta > 0$. Define $P_{\beta}(\gamma) \circeq \beta \gamma + \mathcal{E}(\gamma)$, and recall
        \begin{equation}\label{eq:l.beta.2}
            l_{\beta} \circeq
            \left\{\hspace{-1mm}
            \begin{array}{ll}
                \min\{l\in \{1,...,m\} : \beta \leq \beta_c(\bar{\sigma}_l) \circeq 2 / \bar{\sigma}_l\}, &\mbox{if } \beta \leq 2 / \bar{\sigma}_m, \\
                m+1, &\mbox{otherwise},
            \end{array}
            \right.
        \end{equation}
        from \eqref{eq:l.beta}.
        Then,
        \begin{equation}\label{eq:expression.free.energy}
            \max_{\gamma\in [0,\gamma^{\star}]} P_{\beta}(\gamma) = \sum_{j=1}^{l_{\beta}-1} \left\{2 \frac{\beta}{(2 / \bar{\sigma}_j)}\right\} \nabla \lambda^j + \sum_{j = l_{\beta}}^m \left\{1 + \frac{\beta^2}{(2 / \bar{\sigma}_j)^2}\right\} \nabla \lambda^j \circeq f^{\psi}(\beta).
        \end{equation}
    \end{lemma}

    \begin{proof}
    We consider three cases :
    \begin{equation*}
        (1) ~\, l_{\beta} = m+1; \qquad (2) ~\, l_{\beta} = 1; \qquad (3) ~\, l_{\beta} \in\{2,...,m\}.
    \end{equation*}
    Since $\bar{\sigma}_1 > \bar{\sigma}_2 > ... > \bar{\sigma}_m$, these three cases imply (respectively) :
    \begin{itemize}
        \item[\quad\quad(i)] $\beta > 2 / \bar{\sigma}_j$ for all $j\in \{1,...,m\}$;
        \item[\quad\quad(ii)] $\beta \leq 2 / \bar{\sigma}_j$ for all $j\in \{1,...,m\}$;
        \item[\quad\quad(iii)] $\beta \in (2 / \bar{\sigma}_{l_{\beta}-1},2 / \bar{\sigma}_{l_{\beta}}]$.
    \end{itemize}

    \vspace{3mm}
    \noindent \textbf{Case (1) :}
        For any $\gamma\in (\gamma^{l-1},\gamma^l] \backslash\{\gamma^{\star}\}$, we have
        \begin{equation}\label{eq:entropy.derivative}
            P_{\beta}'(\gamma) = \beta - 2 \frac{(\gamma - \mathcal{J}_{\sigma^2\hspace{-0.3mm} / \bar{\sigma}}(\lambda^{l-1}))}{\mathcal{J}_{\sigma^2}(\lambda^{l-1},1)}.
        \end{equation}
        Any solution to $P_{\beta}'(\gamma) = 0$ must satisfy
        \begin{equation}
            \gamma
            \stackrel{\phantom{(i)}}{=} \mathcal{J}_{\sigma^2\hspace{-0.3mm} / \bar{\sigma}}(\lambda^{l-1}) + \frac{\beta}{2} \mathcal{J}_{\sigma^2}(\lambda^{l-1},1)
            \stackrel{(i)}{>} \mathcal{J}_{\sigma^2\hspace{-0.3mm} / \bar{\sigma}}(\lambda^{l-1}) + \frac{\mathcal{J}_{\sigma^2}(\lambda^{l-1},1)}{\bar{\sigma}_l} \stackrel{\eqref{eq:critic.level.disc.1}}{=} \gamma^l,
        \end{equation}
        which is impossible. Therefore, the maximum $\max_{\gamma\in [0,\gamma^{\star}]} P_{\beta}(\gamma)$ must be achieved at the boundary of $[0,\gamma^{\star}]$.
        We have
        \begin{equation}\label{eq:P.0.smaller.gamma.star}
            P_{\beta}(\gamma^{\star}) \circeq \beta \gamma^{\star} + 0 = \sum_{j=1}^m \left\{2 \frac{\beta}{(2 / \bar{\sigma}_j)}\right\} \nabla \lambda^j \stackrel{(i)}{>} 2 > \beta \cdot 0 + 1 \circeq P_{\beta}(0),
        \end{equation}
        which proves \eqref{eq:expression.free.energy} when $l_{\beta} = m+1$.

    \vspace{3mm}
    \noindent \textbf{Case (2) :}
        From \eqref{eq:entropy.derivative}, any solution $\gamma\in (\gamma^{l-1},\gamma^l] \backslash\{\gamma^{\star}\}$ to $P_{\beta}'(\gamma) = 0$ must satisfy
        \begin{equation}
            \gamma = \mathcal{J}_{\sigma^2\hspace{-0.3mm} / \bar{\sigma}}(\lambda^{l-1}) + \frac{\beta}{2} \mathcal{J}_{\sigma^2}(\lambda^{l-1},1) \quad\text{and}\quad l = 1,
        \end{equation}
        because $l \geq 2$ and the restriction $(ii)$ would otherwise imply $\gamma \leq \gamma^{l-1}$, from \eqref{eq:critic.level.disc.2}.
        In other words, the maximum $\max_{\gamma\in [0,\gamma^{\star}]} P_{\beta}(\gamma)$ must be achieved at the boundary of $[0,\gamma^{\star}]$ or at $\bar{\gamma} \circeq \frac{\beta}{2} \mathcal{J}_{\sigma^2}(1) \in (0,\gamma^1]$.
        Since $\beta > 0$, we have
        \begin{equation}
            P_{\beta}(\bar{\gamma}) = \frac{\beta^2}{2} \mathcal{J}_{\sigma^2}(1) + 1 - \frac{\beta^2}{4} \mathcal{J}_{\sigma^2}(1) = 1 + \frac{\beta^2}{4} \mathcal{J}_{\sigma^2}(1) > 1 = P_{\beta}(0),
        \end{equation}
        and the identity $1 + x^2 \geq 2x$ yields
        \begin{equation}
            P_{\beta}(\bar{\gamma}) = \sum_{j=1}^m \left\{1 + \frac{\beta^2}{(2 / \bar{\sigma}_j)^2}\right\} \nabla \lambda^j \geq \sum_{j=1}^m \left\{2 \frac{\beta}{(2 / \bar{\sigma}_j)}\right\} \nabla \lambda^j = P_{\beta}(\gamma^{\star}).
        \end{equation}
        This proves \eqref{eq:expression.free.energy} when $l_{\beta} = 1$.

    \vspace{3mm}
    \noindent \textbf{Case (3) :}
        From \eqref{eq:entropy.derivative}, any solution $\gamma\in (\gamma^{l-1},\gamma^l] \backslash\{\gamma^{\star}\}$ to $P_{\beta}'(\gamma) = 0$ must satisfy
        \begin{equation}
            \gamma = \mathcal{J}_{\sigma^2\hspace{-0.3mm} / \bar{\sigma}}(\lambda^{l-1}) + \frac{\beta}{2} \mathcal{J}_{\sigma^2}(\lambda^{l-1},1) \quad\text{and}\quad l = l_{\beta}.
        \end{equation}
        We must have the restriction $l = l_{\beta}$ since $\gamma\in (\gamma^{l-1},\gamma^l]$ and $\beta\in (2 / \bar{\sigma}_{l_{\beta}-1},2 / \bar{\sigma}_{l_{\beta}}]$ from (iii) imply
        \begin{align}
            &~\left\{\hspace{-1mm}
            \begin{array}{l}
                \vspace{1mm}\mathcal{J}_{\sigma^2\hspace{-0.3mm} / \bar{\sigma}}(\lambda^{l-1}) + \frac{\mathcal{J}_{\sigma^2}(\lambda^{l-1},1)}{\bar{\sigma}_{l-1}} \stackrel{\eqref{eq:critic.level.disc.2}}{=} \gamma^{l-1} < \gamma \stackrel{(iii)}{\leq} \mathcal{J}_{\sigma^2\hspace{-0.3mm} / \bar{\sigma}}(\lambda^{l-1}) + \frac{\mathcal{J}_{\sigma^2}(\lambda^{l-1},1)}{\bar{\sigma}_{l_{\beta}}} \\
                \mathcal{J}_{\sigma^2\hspace{-0.3mm} / \bar{\sigma}}(\lambda^{l-1}) + \frac{\mathcal{J}_{\sigma^2}(\lambda^{l-1},1)}{\bar{\sigma}_{l_{\beta}-1}} \stackrel{(iii)}{<} \gamma \leq \gamma^l \stackrel{\eqref{eq:critic.level.disc.1}}{=} \mathcal{J}_{\sigma^2\hspace{-0.3mm} / \bar{\sigma}}(\lambda^{l-1}) + \frac{\mathcal{J}_{\sigma^2}(\lambda^{l-1},1)}{\bar{\sigma}_l}
            \end{array}\hspace{-1mm}
            \right\} \notag \\[4pt]
            &\Longrightarrow
            ~\left\{\hspace{-1mm}
            \begin{array}{l}
                \bar{\sigma}_{l_{\beta}} < \bar{\sigma}_{l-1} \\
                \bar{\sigma}_l < \bar{\sigma}_{l_{\beta}-1}
            \end{array}
            \hspace{-1.3mm}\right\} \notag \\[4pt]
            &\Longrightarrow
            \hspace{1.7mm}\{~l = l_{\beta}\hspace{0.5mm}\},
        \end{align}
        where the last implication holds because $\bar{\sigma}_1 > \bar{\sigma}_2 > ... > \bar{\sigma}_m$.

        When we evaluate $P_{\beta}$ at $\bar{\gamma} \circeq \mathcal{J}_{\sigma^2\hspace{-0.3mm} / \bar{\sigma}}(\lambda^{l_{\beta}-1}) + \frac{\beta}{2} \mathcal{J}_{\sigma^2}(\lambda^{l_{\beta}-1},1)$, we get
        \begin{align}\label{eq:P.bar.gamma}
            P_{\beta}(\bar{\gamma})
            &= \beta \mathcal{J}_{\sigma^2\hspace{-0.3mm} / \bar{\sigma}}(\lambda^{l_{\beta}-1}) + \frac{\beta^2}{2} \mathcal{J}_{\sigma^2}(\lambda^{l_{\beta}-1},1) \notag \\
            &\hspace{10mm}+ (1 - \lambda^{l_{\beta}-1}) - \frac{\beta^2}{4} \mathcal{J}_{\sigma^2}(\lambda^{l_{\beta}-1},1) \notag \\
            &= \beta \mathcal{J}_{\sigma^2\hspace{-0.3mm} / \bar{\sigma}}(\lambda^{l_{\beta}-1}) + \left\{(1 - \lambda^{l_{\beta}-1}) + \frac{\beta^2}{4} \mathcal{J}_{\sigma^2}(\lambda^{l_{\beta}-1},1)\right\} \notag \\
            &= \sum_{j=1}^{l_{\beta}-1} \left\{2 \frac{\beta}{(2 / \bar{\sigma}_j)}\right\} \left[\frac{\nabla \mathcal{J}_{\sigma^2\hspace{-0.3mm} / \bar{\sigma}}(\lambda^j)}{\bar{\sigma}_j \nabla \lambda^j}\right] \nabla \lambda^j \notag \\
            &\hspace{10mm}+ \sum_{j=l_{\beta}}^m \left\{1 + \frac{\beta^2}{(2 / \bar{\sigma}_j)^2} \left[\frac{\nabla \mathcal{J}_{\sigma^2}(\lambda^j)}{\bar{\sigma}_j^2 \nabla \lambda^j}\right]\right\} \nabla \lambda^j \notag \\
            &= \sum_{j=1}^{l_{\beta}-1} \left\{2 \frac{\beta}{(2 / \bar{\sigma}_j)}\right\} \nabla \lambda^j + \sum_{j=l_{\beta}}^m \left\{1 + \frac{\beta^2}{(2 / \bar{\sigma}_j)^2}\right\} \nabla \lambda^j.
        \end{align}
        The last equality holds because the pairs of brackets $[\, \cdot\, ]$ on the second and third to last line are equal to $1$.
        Since $\beta > 2 / \bar{\sigma}_{l_{\beta}-1} > 0$ by (iii), we have
        \begin{equation}
            P_{\beta}(\bar{\gamma}) \stackrel{\eqref{eq:P.bar.gamma}}{>} \sum_{j=1}^{l_{\beta}-1} \left\{2\right\} \nabla \lambda^j + \sum_{j=l_{\beta}}^m \left\{1\right\} \nabla \lambda^j > 1 = P_{\beta}(0),
        \end{equation}
        and the identity $1 + x^2 \geq 2x$ yields
        \begin{equation}
            P_{\beta}(\bar{\gamma}) \stackrel{\eqref{eq:P.bar.gamma}}{\geq} \sum_{j=1}^m \left\{2 \frac{\beta}{(2 / \bar{\sigma}_j)}\right\} \nabla \lambda^j = P_{\beta}(\gamma^{\star}).
        \end{equation}
        This proves \eqref{eq:expression.free.energy} when $l_{\beta}\in \{2,...,m\}$, and end the proof of Lemma \ref{lem:tech.lemma.2}.
    \end{proof}

    We recall the definition of the perturbed field $\psi^{u}$. Let $\lambda_{i^{\star}-1} \leq \alpha < \alpha' \leq \lambda_{i^{\star}}$ for a given $i^{\star}\in \{1,...,M\}$, and let $u > -\sigma_{i^{\star}}$. Then,
    \begin{equation}
        \psi_v^u \circeq u\, \phi_v(\alpha,\alpha') + \psi_v, \quad v\in V_N.
    \end{equation}

    \begin{lemma}\label{lem:IGFF.modified.free.energy.convexity}
        Let $\beta > 0$, $\rho\in (0,1]$ and let $\lambda_{i^{\star}-1} \leq \alpha < \alpha' \leq \lambda_{i^{\star}}$ for some $i^{\star}$\hspace{-0.5mm}.
        Then, $u\mapsto f_{N,\rho}^{\psi^u}(\beta)$ is almost-surely convex and $u\mapsto \mathbb{E}\big[f_{N,\rho}^{\psi^u}(\beta)\big]$ is convex.
    \end{lemma}

    \begin{proof}
        By definition, we have
        \begin{align}
            F(u) \circeq f_{N,\rho}^{\psi^u}(\beta) = \frac{1}{\log N^2} \log\left(\int_{A_{N,\rho}} \hspace{-2mm} (g(v))^u d \mu(v)\right),
        \end{align}
        where $g(v) \circeq \exp(\beta \phi_v(\alpha,\alpha'))$ and $\mu(A) \circeq \sum_{v\in A} \exp(\beta \psi_v)$ for any $A \subseteq V_N$.
        By standard properties of logarithms, we see that $u \mapsto F(u)$ is convex almost-surely since, for all $\lambda\in [0,1]$ and all $u,u' > -\sigma_{i^{\star}}$, we have
        \begin{align}
            \hspace{8mm}&\hspace{-10mm}F(\lambda u + (1 - \lambda)u') \leq \lambda F(u) + (1-\lambda) F(u') \notag \\
            &~~\Longleftrightarrow~~
            \int_{A_{N,\rho}} \hspace{-3mm}(g(v))^{\lambda u} (g(v))^{(1-\lambda) u'} d\mu (v) \notag \\
            &\hspace{10mm}\leq \left(\int_{A_{N,\rho}} \hspace{-2mm} (g(v))^u d\mu (v)\right)^{\hspace{-1.1mm}\lambda} \left(\int_{A_{N,\rho}} \hspace{-2mm}(g(v))^{u'} d\mu (v)\right)^{\hspace{-1.3mm}1-\lambda}\hspace{-4.3mm},
        \end{align}
        and the last inequality is true by Holder's inequality ($p \circeq 1/\lambda$, $q \circeq 1 / (1 - \lambda)$ and $1/p + 1/q = 1$). The fact that $u\mapsto \esp{}{F(u)}$ is also convex follows immediately from the linearity and monotonicity of expectations.
    \end{proof}

    The parameters of $\psi^u$ can be encoded simultaneously in the left-continuous step function
    \vspace{-1mm}
    \begin{equation}
        \vspace{1mm}\sigma_u(r) \circeq
        \left\{\hspace{-1mm}
        \begin{array}{ll}
            \sigma(r), ~&\mbox{for all $r\in [0,1] \backslash (\alpha,\alpha']$}, \\
            \sigma_{i^{\star}} + u, ~&\mbox{for all $r\in (\alpha,\alpha']$}.
        \end{array}
        \right.
    \end{equation}
    Since $\mathcal{J}_{\sigma_u^2}(\cdot)$ is an increasing polygonal line, there exists a unique non-increasing left-continuous step function $r \mapsto \bar{\sigma}_u(r)$ such that the {\it concavification} of $\mathcal{J}_{\sigma_u^2}$ can be expressed as the integral of $r \mapsto \bar{\sigma}_u^2(r)$ :
    \begin{equation}
        \hat{\mathcal{J}}_{\sigma_u^2}(s) = \mathcal{J}_{\bar{\sigma}_u^2}(s) = \int_0^s \bar{\sigma}_u^2(r)\, dr \  \text{ for all $s\in (0,1]$\hspace{0.3mm}. }
    \end{equation}
    As for the field $\psi$,
    \begin{itemize}
        \item[$\bullet$] $\bar{\sigma}_{u,j}, ~1 \leq j \leq m_u,$ denote the heights of the steps of $r \mapsto \bar{\sigma}_u(r)$,
        \item[$\bullet$] $m_u$ denotes the number of steps,
        \item[$\bullet$] $\lambda_u^j$ denote the scales at which $r \mapsto \bar{\sigma}_u(r)$ jumps.
    \end{itemize}
    Recall $l_{\beta}$ from \eqref{eq:l.beta.2} and define the analogue for $\psi^u$ :
    \begin{equation}
        l_{\beta,u} \circeq
        \left\{\hspace{-1mm}
        \begin{array}{ll}
            \min\{l\in \{1,...,m_u\} : \beta \leq \beta_c(\bar{\sigma}_{u,l}) \circeq 2 / \bar{\sigma}_{u,l}\}, &\mbox{if } \beta \leq 2 / \bar{\sigma}_{u,m_u}, \\
            m_u + 1, &\mbox{otherwise}.
        \end{array}
        \right.
    \end{equation}
    The following lemma studies the differentiability of the limiting free energy of $\psi^u$ with respect to the perturbation parameter $u$.

    \begin{lemma}\label{lem:IGFF.modified.free.energy.derivative}
        Let $\beta > 0$ and let $\lambda^{j^{\star}-1} \hspace{-0.5mm}\leq \lambda_{i^{\star}-1} \hspace{-0.5mm}\leq \hspace{-0.5mm}\alpha \hspace{-0.5mm} < \hspace{-0.5mm}\alpha' \hspace{-0.5mm}\leq \lambda_{i^{\star}} \hspace{-0.5mm}\leq \lambda^{j^{\star}}$ for some $i^{\star}, j^{\star}$.
        Whenever $\beta \neq 2 / \bar{\sigma}_{j^{\star}}$, there exists $\delta = \delta(\beta,\alpha,\alpha',\boldsymbol{\sigma},\boldsymbol{\lambda}) > 0$ such that $u \mapsto f^{\psi^u}(\beta)$ is differentiable on $(-\delta,\delta)$. The derivative at $u=0$ is given by
        \begin{equation}\label{eq:lem:IGFF.modified.free.energy.derivative}
            \frac{\partial}{\partial u} f^{\psi^0}(\beta) =
            \left\{\hspace{-2mm}
            \begin{array}{ll}
                \frac{\beta \sigma_{i^{\star}} (\alpha' - \alpha)}{\bar{\sigma}_{j^{\star}}}, &\mbox{if } j^{\star} \leq l_{\beta} - 1, \\
                \frac{\beta^2 \sigma_{i^{\star}} (\alpha' - \alpha)}{2}, &\mbox{if } j^{\star} \geq l_{\beta}.
            \end{array}
            \right.
        \end{equation}
        When $\beta = 2 / \bar{\sigma}_{j^{\star}}$, there exists $\delta = \delta(\beta,\alpha,\alpha',\boldsymbol{\sigma},\boldsymbol{\lambda}) > 0$ such that $u \mapsto f^{\psi^u}(\beta)$ is differentiable on $(-\delta,\delta)\backslash\{0\}$, but the derivative at $u=0$ does not exist.
    \end{lemma}

    \newpage
    \begin{proof}
        We separate the proof in two cases :
        \begin{itemize}
            \item[] \textbf{Case (i) : } $\mathcal{J}_{\sigma^2}(r) < \mathcal{J}_{\bar{\sigma}^2}(r)$ for all $r\in (\lambda^{j^{\star}-1},\lambda^{j^{\star}})$;
            \item[] \textbf{Case (ii) : } $\exists r\in (\lambda^{j^{\star}-1},\lambda^{j^{\star}})$ such that $\mathcal{J}_{\sigma^2}(r) = \mathcal{J}_{\bar{\sigma}^2}(r)$.
        \end{itemize}

    \noindent \textbf{Case (i) :}
        The function $u \mapsto \bar{\sigma}_u(r)$ is continuous, uniformly in $r\in [0,1]$.
        Hence, we can choose $\delta = \delta(\alpha,\alpha',\boldsymbol{\sigma},\boldsymbol{\lambda}) > 0$ small enough that for all $u\in (-\delta,\delta)$ :
        \begin{itemize}
            \item $\bar{\sigma}_{u,j} = \bar{\sigma}_j$ for all $j \neq j^{\star}$;
            \item $\lambda_u^j = \lambda^j$ for all $j\in \{1,...,m_u\}$;
            \item $m_u = m$.
        \end{itemize}
        Figure \ref{fig:free.energy.derivative.u.condition} below illustrates this point more clearly.

        Note that $l_{\beta,0} = l_{\beta}$ and also $2 / \bar{\sigma}_{l_{\beta}-1} < \beta \leq 2 / \bar{\sigma}_{l_{\beta}}$ ($\beta > 2 / \bar{\sigma}_m$ when $l_{\beta} = m+1$).
        We can choose $\delta = \delta(\beta,\alpha,\alpha',\boldsymbol{\sigma},\boldsymbol{\lambda}) > 0$ small enough that
        \begin{align}
            &\text{when } \beta \neq 2 / \bar{\sigma}_{j^{\star}} (\Longrightarrow ~1 \leq j^{\star}\hspace{-0.5mm} \leq m), ~~ \text{then } l_{\beta,u} = l_{\beta} ~\text{for all } u\in (-\delta,\delta), \\
            &\text{when } \beta = 2 / \bar{\sigma}_{j^{\star}} (\Longrightarrow ~j^{\star}\hspace{-0.5mm} = l_{\beta}), ~~ \text{then } l_{\beta,u} =
            \left\{\hspace{-1mm}
            \begin{array}{ll}
                l_{\beta}, &\mbox{if } u\in (-\delta,0], \\
                l_{\beta} + 1, &\mbox{if } u\in (0,\delta).
            \end{array}
            \right.
        \end{align}
        From \eqref{eq:expression.free.energy},
        \vspace{-2mm}
        \begin{align}\label{eq:expression.free.energy.expand}
            &\hspace{-1mm}f^{\psi^u}(\beta) - f^{\psi}(\beta)
            = \left\{\hspace{-1mm}
            \begin{array}{ll}
                \beta (\bar{\sigma}_{u,j^{\star}} - \bar{\sigma}_{j^{\star}}) \nabla \lambda^{j^{\star}}\hspace{-1mm}, &\mbox{if } j^{\star}\hspace{-0.5mm} \leq l_{\beta} - 1 \mbox{ and } l_{\beta,u} = l_{\beta}, \\[2mm]
                \beta (\bar{\sigma}_{u,j^{\star}} - \frac{\beta}{4}\bar{\sigma}_{j^{\star}}^2) \nabla \lambda^{j^{\star}}\hspace{-1mm}, &\mbox{if } j^{\star}\hspace{-0.5mm} = l_{\beta} \mbox{ and } l_{\beta,u} = l_{\beta} + 1, \\[2mm]
                \frac{\beta^2}{4} (\bar{\sigma}_{u,j^{\star}}^2 - \bar{\sigma}_{j^{\star}}^2) \nabla \lambda^{j^{\star}}\hspace{-1mm}, &\mbox{if } j^{\star}\hspace{-0.5mm} \geq l_{\beta} \mbox{ and } l_{\beta,u} = l_{\beta},
            \end{array}
            \right. \notag \\[3mm]
            &\hspace{-1mm}= \left\{\hspace{-1mm}
            \begin{array}{ll}
                \beta \left(\sqrt{\bar{\sigma}_{u,j^{\star}}^2 \nabla \lambda^{j^{\star}}} - \sqrt{\bar{\sigma}_{j^{\star}}^2 \nabla \lambda^{j^{\star}}}\right) \sqrt{\nabla \lambda^{j^{\star}}}, &\mbox{if } j^{\star}\hspace{-0.5mm} \leq l_{\beta} - 1 \mbox{ and } l_{\beta,u} = l_{\beta}, \\[2mm]
                \beta (\bar{\sigma}_{u,j^{\star}} - \bar{\sigma}_{j^{\star}}) \nabla \lambda^{j^{\star}} + \beta \bar{\sigma}_{j^{\star}} (1 - \frac{\beta}{4}\bar{\sigma}_{j^{\star}}) \nabla \lambda^{j^{\star}}\hspace{-1mm}, &\mbox{if } j^{\star}\hspace{-0.5mm} = l_{\beta} \mbox{ and } l_{\beta,u} = l_{\beta} + 1, \\[2mm]
                \frac{\beta^2}{4} (\bar{\sigma}_{u,j^{\star}}^2 - \bar{\sigma}_{j^{\star}}^2) \nabla \lambda^{j^{\star}}\hspace{-1mm}, &\mbox{if } j^{\star}\hspace{-0.5mm} \geq l_{\beta} \mbox{ and } l_{\beta,u} = l_{\beta},
            \end{array}
            \right. \notag \\[3mm]
            &\hspace{-1mm}= \left\{\hspace{-1mm}
            \begin{array}{ll}
                (1^*) :~ \beta \left\{\frac{(2 u \sigma_{i^{\star}} + u^2) (\alpha' - \alpha)}{2 \bar{\sigma}_{j^{\star}}} + O(u^2)\right\}, &\mbox{if } j^{\star}\hspace{-0.5mm} \leq l_{\beta} - 1 \mbox{ and } l_{\beta,u} = l_{\beta}, \\[2mm]
                (2^*) :~ \beta \left\{\frac{(2 u \sigma_{i^{\star}} + u^2) (\alpha' - \alpha)}{2 \bar{\sigma}_{j^{\star}}} + O(u^2)\right\} & \\[2mm]
                \hspace{23mm}+ \beta \bar{\sigma}_{j^{\star}} (1 - \frac{\beta}{4}\bar{\sigma}_{j^{\star}}) \nabla \lambda^{j^{\star}}\hspace{-1mm}, &\mbox{if } j^{\star}\hspace{-0.5mm} = l_{\beta} \mbox{ and } l_{\beta,u} = l_{\beta} + 1, \\[2mm]
                (3^*) :~ \frac{\beta^2}{4} (2 u \sigma_{i^{\star}} + u^2) (\alpha' - \alpha), &\mbox{if } j^{\star}\hspace{-0.5mm} \geq l_{\beta} \mbox{ and } l_{\beta,u} = l_{\beta}.
            \end{array}
            \right.
        \end{align}
        To get the last equality, we used
        \begin{equation}
            (\bar{\sigma}_{u,j^{\star}}^2 - \bar{\sigma}_{j^{\star}}^2) \nabla \lambda^{j^{\star}} = ((\sigma_{i^{\star}} + u)^2 - \sigma_{i^{\star}}^2) (\alpha' - \alpha).
        \end{equation}
        The function $u\mapsto f^{\psi^u}(\beta)$ is always differentiable on $(-\delta,\delta)\backslash\{0\}$.
        Furthermore,
        \begin{align}
            &\text{when } \beta > 2 / \bar{\sigma}_{j^{\star}}, ~~ \frac{\partial}{\partial u^-} f^{\psi^0}(\beta) \stackrel{(1^*)}{=} \frac{\beta \sigma_{i^{\star}} (\alpha' - \alpha)}{\bar{\sigma}_{j^{\star}}} \stackrel{(1^*)}{=} \frac{\partial}{\partial u^+} f^{\psi^0}(\beta), \\
            &\text{when } \beta = 2 / \bar{\sigma}_{j^{\star}}, ~~ \frac{\partial}{\partial u^-} f^{\psi^0}(\beta) \stackrel{(3^*)}{=} \frac{\beta^2 \sigma_{i^{\star}} (\alpha' - \alpha)}{2} \neq +\infty \stackrel{(2^*)}{=} \frac{\partial}{\partial u^+} f^{\psi^0}(\beta), \\
            &\text{when } \beta < 2 / \bar{\sigma}_{j^{\star}}, ~~ \frac{\partial}{\partial u^-} f^{\psi^0}(\beta) \stackrel{(3^*)}{=} \frac{\beta^2 \sigma_{i^{\star}} (\alpha' - \alpha)}{2} \stackrel{(3^*)}{=} \frac{\partial}{\partial u^+} f^{\psi^0}(\beta).
        \end{align}
        Thus, $u\mapsto f^{\psi^u}(\beta)$ is also differentiable at $u=0$, except when $\beta = 2 / \bar{\sigma}_{j^{\star}}$.

    \vspace{5mm}
    \begin{figure}[H]
        \centering
        \includegraphics[scale=0.78]{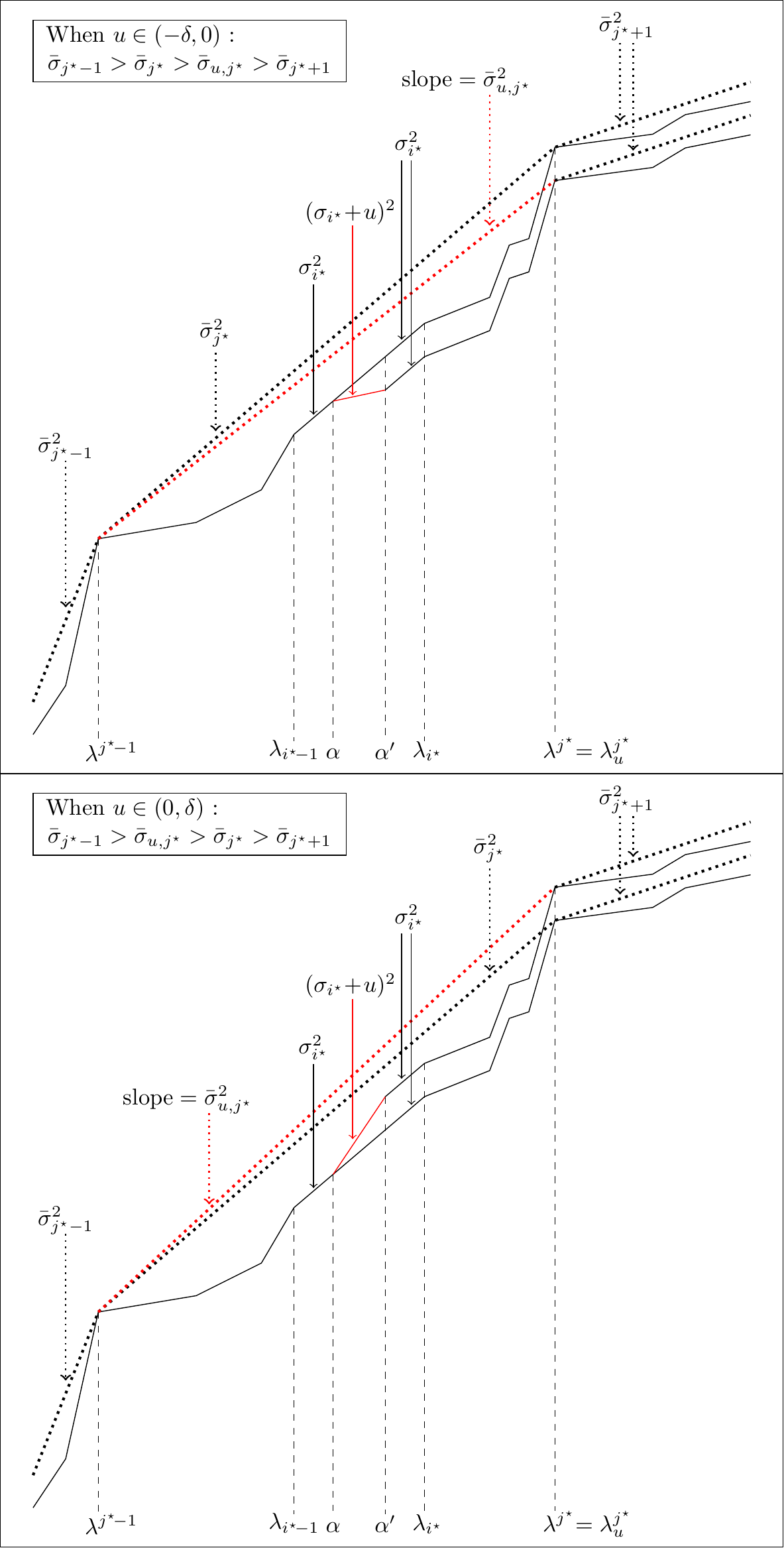}
        \captionsetup{width=0.8\textwidth}
        \caption{The dotted paths represent $\mathcal{J}_{\bar{\sigma}^2}$ and $\mathcal{J}_{\bar{\sigma}_u^2}$. The closed paths represent $\mathcal{J}_{\sigma^2}$ and $\mathcal{J}_{\sigma_u^2}$. The paths containing a {\color{red} red part} are the ones for the perturbed field $\psi^u$.}
        \label{fig:free.energy.derivative.u.condition}
    \end{figure}

    \newpage
    \noindent \textbf{Case (ii) :}
        Here are all the possible subcases of Case (ii) : \vspace{2mm}
        \begin{itemize}
            \item[\quad(ii.1)]
                \begin{itemize}
                    \item[$\bullet$] $\mathcal{J}_{\sigma^2}(r) < \mathcal{J}_{\bar{\sigma}^2}(r)$ for all $r\in [\alpha',\lambda^{j^{\star}})$;
                    \item[$\bullet$] $\exists s\in (\lambda^{j^{\star}-1},\alpha]$ such that $\mathcal{J}_{\sigma^2}(s) = \mathcal{J}_{\bar{\sigma}^2}(s)$; \vspace{2mm}
                \end{itemize}
            \item[\quad(ii.2)]
                \begin{itemize}
                    \item[$\bullet$] $\mathcal{J}_{\sigma^2}(r) < \mathcal{J}_{\bar{\sigma}^2}(r)$ for all $r\in (\lambda^{j^{\star}-1},\alpha]$;
                    \item[$\bullet$] $\exists t\in [\alpha',\lambda^{j^{\star}})$ such that $\mathcal{J}_{\sigma^2}(t) = \mathcal{J}_{\bar{\sigma}^2}(t)$; \vspace{2mm}
                \end{itemize}
            \item[\quad(ii.3)]
                \begin{itemize}
                    \item[$\bullet$] $\exists s\in (\lambda^{j^{\star}-1},\alpha]$ such that $\mathcal{J}_{\sigma^2}(s) = \mathcal{J}_{\bar{\sigma}^2}(s)$;
                    \item[$\bullet$] $\exists t\in [\alpha',\lambda^{j^{\star}})$ such that $\mathcal{J}_{\sigma^2}(t) = \mathcal{J}_{\bar{\sigma}^2}(t)$. \vspace{2mm}
                \end{itemize}
        \end{itemize}
        Denote
        \begin{align}
            &s^{\star} \circeq \max\big\{r\in [\lambda^{j^{\star}-1},\alpha] : \mathcal{J}_{\sigma^2}(r) = \mathcal{J}_{\bar{\sigma}^2}(r)\big\}, \\
            &t^{\star} \circeq \min\big\{r\in [\alpha',\lambda^{j^{\star}}] : \mathcal{J}_{\sigma^2}(r) = \mathcal{J}_{\bar{\sigma}^2}(r)\big\}.
        \end{align}

        Again, the function $u \mapsto \bar{\sigma}_u(r)$ is continuous, uniformly in $r\in [0,1]$.
        Hence, we can choose $\delta = \delta(\alpha,\alpha',\boldsymbol{\sigma},\boldsymbol{\lambda}) > 0$ small enough that for all $u\in (-\delta,\delta)$ :

        \begin{equation*}
            \centering
            \begin{array}{lll}
                \toprule
                [.2\normalbaselineskip] \multicolumn{1}{l}{\text{Case}} & \multicolumn{1}{c}{u < 0} & \multicolumn{1}{c}{u > 0} \\
                \midrule
                \vspace{-3mm}&&\\
                \multicolumn{1}{c}{\text{(ii.1)}}
                \vspace{2mm}&\tabitem \bar{\sigma}_{u,j} =
                    \left\{\hspace{-1.5mm}
                    \begin{array}{ll}
                        \bar{\sigma}_j &\hspace{-2mm}\mbox{for } j\leq j^{\star} \\
                        \bar{\sigma}_{j-1} &\hspace{-2mm}\mbox{for } j\geq j^{\star}+2
                    \end{array}
                    \right.
                &\tabitem \bar{\sigma}_{u,j} = \bar{\sigma}_j ~\text{for } j\neq j^{\star} \\
                \vspace{1mm}&\tabitem \lambda_u^j =
                    \left\{\hspace{-1.5mm}
                    \begin{array}{ll}
                        \lambda^j &\hspace{-2mm}\mbox{for } j\leq j^{\star}-1 \\
                        s^{\star} &\hspace{-2mm}\mbox{for } j = j^{\star} \\
                        \lambda^{j-1} &\hspace{-2mm}\mbox{for } j\geq j^{\star}+1
                    \end{array}
                    \right.
                &\tabitem \lambda_u^j = \lambda^j ~\text{for all } j \\
                \vspace{1mm}&\tabitem m_u = m+1
                &\tabitem m_u = m \\
                \midrule
                \vspace{-3mm}&&\\
                \multicolumn{1}{c}{\text{(ii.2)}}
                \vspace{2mm}&\tabitem \bar{\sigma}_{u,j} = \bar{\sigma}_j ~\text{for } j\neq j^{\star}
                &\tabitem \bar{\sigma}_{u,j} =
                    \left\{\hspace{-1.5mm}
                    \begin{array}{ll}
                        \bar{\sigma}_j &\hspace{-2mm}\mbox{for } j\leq j^{\star}-1 \\
                        \bar{\sigma}_{j-1} &\hspace{-2mm}\mbox{for } j\geq j^{\star}+1
                    \end{array}
                    \right. \\
                \vspace{1mm}&\tabitem \lambda_u^j = \lambda^j ~\text{for all } j
                &\tabitem \lambda_u^j =
                    \left\{\hspace{-1.5mm}
                    \begin{array}{ll}
                        \lambda^j &\hspace{-2mm}\mbox{for } j\leq j^{\star}-1 \\
                        t^{\star} &\hspace{-2mm}\mbox{for } j = j^{\star} \\
                        \lambda^{j-1} &\hspace{-2mm}\mbox{for } j\geq j^{\star}+1
                    \end{array}
                    \right. \\
                \vspace{1mm}&\tabitem m_u = m
                &\tabitem m_u = m+1 \\
                \midrule
                \vspace{-3mm}&&\\
                \multicolumn{1}{c}{\text{(ii.3)}}
                \vspace{2mm}&\tabitem \bar{\sigma}_{u,j} =
                    \left\{\hspace{-1.5mm}
                    \begin{array}{ll}
                        \bar{\sigma}_j &\hspace{-2mm}\mbox{for } j\leq j^{\star} \\
                        \bar{\sigma}_{j-1} &\hspace{-2mm}\mbox{for } j\geq j^{\star}+2
                    \end{array}
                    \right.
                &\tabitem \bar{\sigma}_{u,j} =
                    \left\{\hspace{-1.5mm}
                    \begin{array}{ll}
                        \bar{\sigma}_j &\hspace{-2mm}\mbox{for } j\leq j^{\star}-1 \\
                        \bar{\sigma}_{j-1} &\hspace{-2mm}\mbox{for } j\geq j^{\star}+1
                    \end{array}
                    \right. \\
                \vspace{1mm}&\tabitem \lambda_u^j =
                    \left\{\hspace{-1.5mm}
                    \begin{array}{ll}
                        \lambda^j &\hspace{-2mm}\mbox{for } j\leq j^{\star}-1 \\
                        s^{\star} &\hspace{-2mm}\mbox{for } j = j^{\star} \\
                        \lambda^{j-1} &\hspace{-2mm}\mbox{for } j\geq j^{\star}+1
                    \end{array}
                    \right.
                &\tabitem \lambda_u^j =
                    \left\{\hspace{-1.5mm}
                    \begin{array}{ll}
                        \lambda^j &\hspace{-2mm}\mbox{for } j\leq j^{\star}-1 \\
                        t^{\star} &\hspace{-2mm}\mbox{for } j = j^{\star} \\
                        \lambda^{j-1} &\hspace{-2mm}\mbox{for } j\geq j^{\star}+1
                    \end{array}
                    \right. \\
                \vspace{1mm}&\tabitem m_u = m+1
                &\tabitem m_u = m+1 \\
                \midrule
            \end{array}
        \end{equation*}

        \vspace{2mm}
        In other words, the parameter $\delta$ is chosen small enough that, on $(\lambda^{j^{\star}-1},\lambda^{j^{\star}}]$, the field $\psi^u$ has either one or two effective variance parameters (depending on the subcase) and they remain strictly between $\bar{\sigma}_{j^{\star}-1}$ and $\bar{\sigma}_{j^{\star}+1}$.
        If there is only one effective slope, then $\lambda_u^{j^{\star}} = \lambda^{j^{\star}}$.
        If there are two effective slopes, the segments meet at $\lambda_u^{j^{\star}}\in \{s^{\star},t^{\star}\}$.
        Figure \ref{fig:free.energy.derivative.u.table} on the next page (analogous to Figure \ref{fig:free.energy.derivative.u.condition}) illustrates this point more clearly.

        \newpage
        \begin{landscape}
        \begin{figure}[ht]
            \centering
            \hspace{-2mm}
            \includegraphics[scale=1.19]{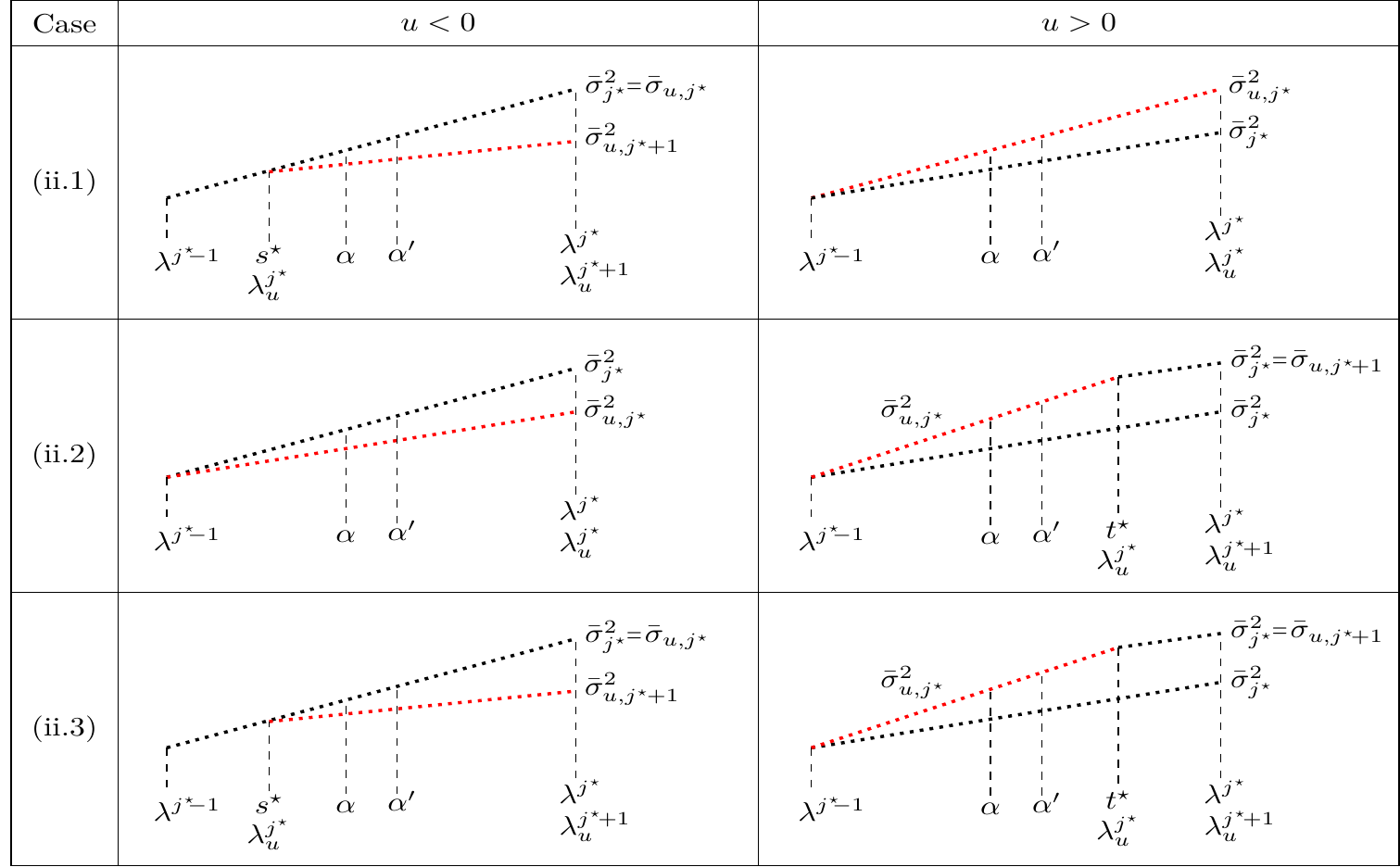}
            \captionsetup{width=1.4\textwidth}
            \caption{General form of $\mathcal{J}_{\bar{\sigma}^2}$ and $\mathcal{J}_{\bar{\sigma}_u^2}$ on $(\lambda^{j^{\star}-1},\lambda^{j^{\star}}]$. The effective slopes of $\psi$ and $\psi^u$ are the quantities $\bar{\sigma}_j^2$ and $\bar{\sigma}_{u,j}^2$, respectively. The dotted paths containing a {\color{red} red part} are the ones for the perturbed field $\psi^u$.}
            \label{fig:free.energy.derivative.u.table}
        \end{figure}
        \end{landscape}

        \newpage
        \noindent \textbf{Case (ii.1) :}
        In this case, $s^{\star}\in (\lambda^{j^{\star}-1},\alpha]$.
        We can choose $\delta = \delta(\beta,\alpha,\alpha',\boldsymbol{\sigma},\boldsymbol{\lambda}) > 0$ small enough that
        \begin{align}
            &\text{when } \beta > 2 / \bar{\sigma}_{j^{\star}} (\Longrightarrow j^{\star}\hspace{-0.5mm} \leq l_{\beta} - 1), ~~ \text{then } l_{\beta,u} =
            \left\{\hspace{-1mm}
            \begin{array}{ll}
                l_{\beta} + 1, &\mbox{if } u\in (-\delta,0), \\
                l_{\beta}, &\mbox{if } u\in [0,\delta),
            \end{array}
            \right. \\
            &\text{when } \beta = 2 / \bar{\sigma}_{j^{\star}} (\Longrightarrow j^{\star}\hspace{-0.5mm} = l_{\beta}), ~~ \text{then } l_{\beta,u} =
            \left\{\hspace{-1mm}
            \begin{array}{ll}
                l_{\beta}, &\mbox{if } u\in (-\delta,0], \\
                l_{\beta} + 1, &\mbox{if } u\in (0,\delta),
            \end{array}
            \right. \\[2mm]
            &\text{when } \beta < 2 / \bar{\sigma}_{j^{\star}} (\Longrightarrow j^{\star}\hspace{-0.5mm} \geq l_{\beta}), ~~ \text{then } l_{\beta,u} = l_{\beta} ~\text{for all } u\in (-\delta,\delta).
        \end{align}
        When $u < 0$,
        \begin{align}\label{eq:expression.free.energy.expand.2.1.u.negative}
            &f^{\psi^u}(\beta) - f^{\psi}(\beta) \notag \\
            &= \left\{\hspace{-1mm}
            \begin{array}{ll}
                \beta (\bar{\sigma}_{u,j^{\star}+1} - \bar{\sigma}_{j^{\star}}) (\lambda^{j^{\star}} - s^{\star}), &\mbox{if } j^{\star}\hspace{-0.5mm} \leq l_{\beta} - 1 \mbox{ and } l_{\beta,u} = l_{\beta} + 1, \\[2mm]
                \frac{\beta^2}{4} (\bar{\sigma}_{u,j^{\star}+1}^2 - \bar{\sigma}_{j^{\star}}^2) (\lambda^{j^{\star}} - s^{\star}), &\mbox{if } j^{\star}\hspace{-0.5mm} \geq l_{\beta} \mbox{ and } l_{\beta,u} = l_{\beta},
            \end{array}
            \right. \notag \\[3mm]
            &= \left\{\hspace{-1mm}
            \begin{array}{ll}
                \beta \left(\sqrt{\bar{\sigma}_{u,j^{\star}+1}^2 (\lambda^{j^{\star}} - s^{\star})} - \sqrt{\bar{\sigma}_{j^{\star}}^2 (\lambda^{j^{\star}} - s^{\star})}\right) \sqrt{\lambda^{j^{\star}} - s^{\star}}, & \\[2mm]
                &\hspace{-24mm}\mbox{if } j^{\star}\hspace{-0.5mm} \leq l_{\beta} - 1 \mbox{ and } l_{\beta,u} = l_{\beta} + 1, \\[2mm]
                \frac{\beta^2}{4} (\bar{\sigma}_{u,j^{\star}+1}^2 - \bar{\sigma}_{j^{\star}}^2) (\lambda^{j^{\star}} - s^{\star}), &\hspace{-24mm}\mbox{if } j^{\star}\hspace{-0.5mm} \geq l_{\beta} \mbox{ and } l_{\beta,u} = l_{\beta},
            \end{array}
            \right. \notag \\[3mm]
            &= \left\{\hspace{-1mm}
            \begin{array}{ll}
                (1^-) :~ \beta \left\{\frac{(2 u \sigma_{i^{\star}} + u^2) (\alpha' - \alpha)}{2 \bar{\sigma}_{j^{\star}}} + O(u^2)\right\}, &\mbox{if } j^{\star}\hspace{-0.5mm} \leq l_{\beta} - 1 \mbox{ and } l_{\beta,u} = l_{\beta} + 1, \\[2mm]
                (2^-) :~ \frac{\beta^2}{4} (2 u \sigma_{i^{\star}} + u^2) (\alpha' - \alpha), &\mbox{if } j^{\star}\hspace{-0.5mm} \geq l_{\beta} \mbox{ and } l_{\beta,u} = l_{\beta}.
            \end{array}
            \right.
        \end{align}
        To get the last equality, we used
        \begin{equation}
            (\bar{\sigma}_{u,j^{\star}+1}^2 - \bar{\sigma}_{j^{\star}}^2) (\lambda^{j^{\star}} - s^{\star}) = ((\sigma_{i^{\star}} + u)^2 - \sigma_{i^{\star}}^2) (\alpha' - \alpha).
        \end{equation}
        When $u > 0$, it is the same as in \eqref{eq:expression.free.energy.expand} :
        \begin{align}\label{eq:expression.free.energy.expand.2.1.u.positive}
            &f^{\psi^u}(\beta) - f^{\psi}(\beta) \notag \\
            &= \left\{\hspace{-1mm}
            \begin{array}{ll}
                (1^+) :~ \beta \left\{\frac{(2 u \sigma_{i^{\star}} + u^2) (\alpha' - \alpha)}{2 \bar{\sigma}_{j^{\star}}} + O(u^2)\right\}, &\mbox{if } j^{\star}\hspace{-0.5mm} \leq l_{\beta} - 1 \mbox{ and } l_{\beta,u} = l_{\beta}, \\[2mm]
                (2^+) :~ \beta \left\{\frac{(2 u \sigma_{i^{\star}} + u^2) (\alpha' - \alpha)}{2 \bar{\sigma}_{j^{\star}}} + O(u^2)\right\} & \\[2mm]
                \hspace{23.5mm}+ \beta \bar{\sigma}_{j^{\star}} (1 - \frac{\beta}{4}\bar{\sigma}_{j^{\star}}) \nabla \lambda^{j^{\star}}\hspace{-1mm}, &\mbox{if } j^{\star}\hspace{-0.5mm} = l_{\beta} \mbox{ and } l_{\beta,u} = l_{\beta} + 1, \\[2mm]
                (3^+) :~ \frac{\beta^2}{4} (2 u \sigma_{i^{\star}} + u^2) (\alpha' - \alpha), &\mbox{if } j^{\star}\hspace{-0.5mm} \geq l_{\beta} \mbox{ and } l_{\beta,u} = l_{\beta}.
            \end{array}
            \right.
        \end{align}
        The function $u\mapsto f^{\psi^u}(\beta)$ is always differentiable on $(-\delta,\delta)\backslash\{0\}$.
        Furthermore,
        \begin{align}
            &\text{when } \beta > 2 / \bar{\sigma}_{j^{\star}}, ~~ \frac{\partial}{\partial u^-} f^{\psi^0}(\beta) \stackrel{(1^-)}{=} \frac{\beta \sigma_{i^{\star}} (\alpha' - \alpha)}{\bar{\sigma}_{j^{\star}}} \stackrel{(1^+)}{=} \frac{\partial}{\partial u^+} f^{\psi^0}(\beta), \\
            &\text{when } \beta = 2 / \bar{\sigma}_{j^{\star}}, ~~ \frac{\partial}{\partial u^-} f^{\psi^0}(\beta) \stackrel{(2^-)}{=} \frac{\beta^2 \sigma_{i^{\star}} (\alpha' - \alpha)}{2} \neq +\infty \stackrel{(2^+)}{=} \frac{\partial}{\partial u^+} f^{\psi^0}(\beta), \\
            &\text{when } \beta < 2 / \bar{\sigma}_{j^{\star}}, ~~ \frac{\partial}{\partial u^-} f^{\psi^0}(\beta) \stackrel{(2^-)}{=} \frac{\beta^2 \sigma_{i^{\star}} (\alpha' - \alpha)}{2} \stackrel{(3^+)}{=} \frac{\partial}{\partial u^+} f^{\psi^0}(\beta).
        \end{align}
        Thus, $u\mapsto f^{\psi^u}(\beta)$ is also differentiable at $u=0$, except when $\beta = 2 / \bar{\sigma}_{j^{\star}}$.

        \vspace{5mm}
        \noindent \textbf{Case (ii.2) :}
        In this case, $t^{\star}\in [\alpha',\lambda^{j^{\star}})$.
        We can choose $\delta = \delta(\beta,\alpha,\alpha',\boldsymbol{\sigma},\boldsymbol{\lambda}) > 0$ small enough that
        \begin{align}
            &\text{when } \beta > 2 / \bar{\sigma}_{j^{\star}} (\Longrightarrow j^{\star}\hspace{-0.5mm} \leq l_{\beta} - 1), ~~ \text{then } l_{\beta,u} =
            \left\{\hspace{-1mm}
            \begin{array}{ll}
                l_{\beta}, &\mbox{if } u\in (-\delta,0], \\
                l_{\beta} + 1, &\mbox{if } u\in (0,\delta),
            \end{array}
            \right. \\
            &\text{when } \beta = 2 / \bar{\sigma}_{j^{\star}} (\Longrightarrow j^{\star}\hspace{-0.5mm} = l_{\beta}), ~~ \text{then } l_{\beta,u} =
            \left\{\hspace{-1mm}
            \begin{array}{ll}
                l_{\beta}, &\mbox{if } u\in (-\delta,0], \\
                l_{\beta} + 1, &\mbox{if } u\in (0,\delta),
            \end{array}
            \right. \\[2mm]
            &\text{when } \beta < 2 / \bar{\sigma}_{j^{\star}} (\Longrightarrow j^{\star}\hspace{-0.5mm} \geq l_{\beta}), ~~ \text{then } l_{\beta,u} = l_{\beta} ~\text{for all } u\in (-\delta,\delta).
        \end{align}
        When $u < 0$, it is the same as in \eqref{eq:expression.free.energy.expand} (without $(2^*)$) :
        \begin{align}\label{eq:expression.free.energy.expand.2.2.u.negative}
            &f^{\psi^u}(\beta) - f^{\psi}(\beta) \notag \\
            &= \left\{\hspace{-1mm}
            \begin{array}{ll}
                (1^-) :~ \beta \left\{\frac{(2 u \sigma_{i^{\star}} + u^2) (\alpha' - \alpha)}{2 \bar{\sigma}_{j^{\star}}} + O(u^2)\right\}, &\mbox{if } j^{\star}\hspace{-0.5mm} \leq l_{\beta} - 1 \mbox{ and } l_{\beta,u} = l_{\beta}, \\[2mm]
                (2^-) :~ \frac{\beta^2}{4} (2 u \sigma_{i^{\star}} + u^2) (\alpha' - \alpha), &\mbox{if } j^{\star}\hspace{-0.5mm} \geq l_{\beta} \mbox{ and } l_{\beta,u} = l_{\beta}.
            \end{array}
            \right.
        \end{align}
        When $u > 0$,
        \begin{align}\label{eq:expression.free.energy.expand.2.2.u.positive}
            &f^{\psi^u}(\beta) - f^{\psi}(\beta) \notag \\[3mm]
            &= \left\{\hspace{-1mm}
            \begin{array}{ll}
                \beta (\bar{\sigma}_{u,j^{\star}} - \bar{\sigma}_{j^{\star}}) (t^{\star} - \lambda^{j^{\star}-1}), &\mbox{if } j^{\star}\hspace{-0.5mm} \leq l_{\beta} - 1 \mbox{ and } l_{\beta,u} = l_{\beta} + 1, \\[2mm]
                \beta (\bar{\sigma}_{u,j^{\star}} - \frac{\beta}{4}\bar{\sigma}_{j^{\star}}^2) (t^{\star} - \lambda^{j^{\star}-1}), &\mbox{if } j^{\star}\hspace{-0.5mm} = l_{\beta} \mbox{ and } l_{\beta,u} = l_{\beta} + 1, \\[2mm]
                \frac{\beta^2}{4} (\bar{\sigma}_{u,j^{\star}+1}^2 - \bar{\sigma}_{j^{\star}}^2) (t^{\star} - \lambda^{j^{\star}-1}), &\mbox{if } j^{\star}\hspace{-0.5mm} \geq l_{\beta} \mbox{ and } l_{\beta,u} = l_{\beta},
            \end{array}
            \right. \notag \\[5mm]
            &= \left\{\hspace{-1mm}
            \begin{array}{ll}
                \beta \left(\sqrt{\bar{\sigma}_{u,j^{\star}}^2 (t^{\star} - \lambda^{j^{\star}-1})} - \sqrt{\bar{\sigma}_{j^{\star}}^2 (t^{\star} - \lambda^{j^{\star}-1})}\right) \sqrt{t^{\star} - \lambda^{j^{\star}-1}}, \hspace{-30mm}& \\[2mm]
                &\mbox{if } j^{\star}\hspace{-0.5mm} \leq l_{\beta} - 1 \mbox{ and } l_{\beta,u} = l_{\beta} + 1, \\[2mm]
                \beta (\bar{\sigma}_{u,j^{\star}} - \bar{\sigma}_{j^{\star}}) (t^{\star} - \lambda^{j^{\star}-1}) & \\[2mm]
                \hspace{12mm}+ \beta \bar{\sigma}_{j^{\star}}(1 - \frac{\beta}{4}\bar{\sigma}_{j^{\star}}) (t^{\star} - \lambda^{j^{\star}-1}), &\mbox{if } j^{\star}\hspace{-0.5mm} = l_{\beta} \mbox{ and } l_{\beta,u} = l_{\beta} + 1, \\[2mm]
                \frac{\beta^2}{4} (\bar{\sigma}_{u,j^{\star}+1}^2 - \bar{\sigma}_{j^{\star}}^2) (t^{\star} - \lambda^{j^{\star}-1}), &\mbox{if } j^{\star}\hspace{-0.5mm} \geq l_{\beta} \mbox{ and } l_{\beta,u} = l_{\beta},
            \end{array}
            \right. \notag \\[5mm]
            &= \left\{\hspace{-1mm}
            \begin{array}{ll}
                (1^+) :~ \beta \left\{\frac{(2 u \sigma_{i^{\star}} + u^2) (\alpha' - \alpha)}{2 \bar{\sigma}_{j^{\star}}} + O(u^2)\right\}, &\mbox{if } j^{\star}\hspace{-0.5mm} \leq l_{\beta} - 1 \mbox{ and } l_{\beta,u} = l_{\beta} + 1, \\[2mm]
                (2^+) :~ \beta \left\{\frac{(2 u \sigma_{i^{\star}} + u^2) (\alpha' - \alpha)}{2 \bar{\sigma}_{j^{\star}}} + O(u^2)\right\} & \\[2mm]
                \hspace{12.5mm}+ \beta \bar{\sigma}_{j^{\star}} (1 - \frac{\beta}{4} \bar{\sigma}_{j^{\star}}) (t^{\star} - \lambda^{j^{\star}-1}), &\mbox{if } j^{\star}\hspace{-0.5mm} = l_{\beta} \mbox{ and } l_{\beta,u} = l_{\beta} + 1, \\[2mm]
                (3^+) :~ \frac{\beta^2}{4} (2 u \sigma_{i^{\star}} + u^2) (\alpha' - \alpha), &\mbox{if } j^{\star}\hspace{-0.5mm} \geq l_{\beta} \mbox{ and } l_{\beta,u} = l_{\beta}.
            \end{array}
            \right.
        \end{align}
        To get the last equality, we used
        \begin{equation}
            (\bar{\sigma}_{u,j^{\star}}^2 - \bar{\sigma}_{j^{\star}}^2) (t^{\star} - \lambda^{j^{\star}-1}) = ((\sigma_{i^{\star}} + u)^2 - \sigma_{i^{\star}}^2) (\alpha' - \alpha).
        \end{equation}
        The function $u\mapsto f^{\psi^u}(\beta)$ is always differentiable on $(-\delta,\delta)\backslash\{0\}$.
        Furthermore,
        \begin{align}
            &\text{when } \beta > 2 / \bar{\sigma}_{j^{\star}}, ~~ \frac{\partial}{\partial u^-} f^{\psi^0}(\beta) \stackrel{(1^-)}{=} \frac{\beta \sigma_{i^{\star}} (\alpha' - \alpha)}{\bar{\sigma}_{j^{\star}}} \stackrel{(1^+)}{=} \frac{\partial}{\partial u^+} f^{\psi^0}(\beta), \\
            &\text{when } \beta = 2 / \bar{\sigma}_{j^{\star}}, ~~ \frac{\partial}{\partial u^-} f^{\psi^0}(\beta) \stackrel{(2^-)}{=} \frac{\beta^2 \sigma_{i^{\star}} (\alpha' - \alpha)}{2} \neq +\infty \stackrel{(2^+)}{=} \frac{\partial}{\partial u^+} f^{\psi^0}(\beta), \\
            &\text{when } \beta < 2 / \bar{\sigma}_{j^{\star}}, ~~ \frac{\partial}{\partial u^-} f^{\psi^0}(\beta) \stackrel{(2^-)}{=} \frac{\beta^2 \sigma_{i^{\star}} (\alpha' - \alpha)}{2} \stackrel{(3^+)}{=} \frac{\partial}{\partial u^+} f^{\psi^0}(\beta).
        \end{align}
        Thus, $u\mapsto f^{\psi^u}(\beta)$ is also differentiable at $u=0$, except when $\beta = 2 / \bar{\sigma}_{j^{\star}}$.

        \vspace{5mm}
        \noindent \textbf{Case (ii.3) :}
        In this case, $s^{\star}\in (\lambda^{j^{\star}-1},\alpha]$ and $t^{\star}\in [\alpha',\lambda^{j^{\star}})$.
        We can choose $\delta = \delta(\beta,\alpha,\alpha',\boldsymbol{\sigma},\boldsymbol{\lambda}) > 0$ small enough that
        \begin{align}
            &\text{when } \beta > 2 / \bar{\sigma}_{j^{\star}} (\Longrightarrow j^{\star}\hspace{-0.5mm} \leq l_{\beta} - 1), ~~ \text{then } l_{\beta,u} =
            \left\{\hspace{-1mm}
            \begin{array}{ll}
                l_{\beta} + 1, &\mbox{if } u\in (-\delta,0], \\
                l_{\beta} + 1, &\mbox{if } u\in (0,\delta),
            \end{array}
            \right. \\
            &\text{when } \beta = 2 / \bar{\sigma}_{j^{\star}} (\Longrightarrow j^{\star}\hspace{-0.5mm} = l_{\beta}), ~~ \text{then } l_{\beta,u} =
            \left\{\hspace{-1mm}
            \begin{array}{ll}
                l_{\beta}, &\mbox{if } u\in (-\delta,0], \\
                l_{\beta} + 1, &\mbox{if } u\in (0,\delta),
            \end{array}
            \right. \\[2mm]
            &\text{when } \beta < 2 / \bar{\sigma}_{j^{\star}} (\Longrightarrow j^{\star}\hspace{-0.5mm} \geq l_{\beta}), ~~ \text{then } l_{\beta,u} = l_{\beta} ~\text{for all } u\in (-\delta,\delta).
        \end{align}
        When $u < 0$, it is the same as in \eqref{eq:expression.free.energy.expand.2.1.u.negative} :
        \begin{align}\label{eq:expression.free.energy.expand.2.3.u.negative}
            &f^{\psi^u}(\beta) - f^{\psi}(\beta) \notag \\
            &= \left\{\hspace{-1mm}
            \begin{array}{ll}
                (1^-) :~ \beta \left\{\frac{(2 u \sigma_{i^{\star}} + u^2) (\alpha' - \alpha)}{2 \bar{\sigma}_{j^{\star}}} + O(u^2)\right\}, &\mbox{if } j^{\star}\hspace{-0.5mm} \leq l_{\beta} - 1 \mbox{ and } l_{\beta,u} = l_{\beta} + 1, \\[2mm]
                (2^-) :~ \frac{\beta^2}{4} (2 u \sigma_{i^{\star}} + u^2) (\alpha' - \alpha), &\mbox{if } j^{\star}\hspace{-0.5mm} \geq l_{\beta} \mbox{ and } l_{\beta,u} = l_{\beta}.
            \end{array}
            \right.
        \end{align}
        When $u > 0$, it is the same as in \eqref{eq:expression.free.energy.expand.2.2.u.positive} :
        \begin{align}\label{eq:expression.free.energy.expand.2.3.u.positive}
            &f^{\psi^u}(\beta) - f^{\psi}(\beta) \notag \\
            &= \left\{\hspace{-1mm}
            \begin{array}{ll}
                (1^+) :~ \beta \left\{\frac{(2 u \sigma_{i^{\star}} + u^2) (\alpha' - \alpha)}{2 \bar{\sigma}_{j^{\star}}} + O(u^2)\right\}, &\mbox{if } j^{\star}\hspace{-0.5mm} \leq l_{\beta} - 1 \mbox{ and } l_{\beta,u} = l_{\beta} + 1, \\[2mm]
                (2^+) :~ \beta \left\{\frac{(2 u \sigma_{i^{\star}} + u^2) (\alpha' - \alpha)}{2 \bar{\sigma}_{j^{\star}}} + O(u^2)\right\} & \\[2mm]
                \hspace{12.5mm}+ \beta \bar{\sigma}_{j^{\star}} (1 - \frac{\beta}{4} \bar{\sigma}_{j^{\star}}) (t^{\star} - \lambda^{j^{\star}-1}), &\mbox{if } j^{\star}\hspace{-0.5mm} = l_{\beta} \mbox{ and } l_{\beta,u} = l_{\beta} + 1, \\[2mm]
                (3^+) :~ \frac{\beta^2}{4} (2 u \sigma_{i^{\star}} + u^2) (\alpha' - \alpha), &\mbox{if } j^{\star}\hspace{-0.5mm} \geq l_{\beta} \mbox{ and } l_{\beta,u} = l_{\beta}.
            \end{array}
            \right.
        \end{align}
        The function $u\mapsto f^{\psi^u}(\beta)$ is always differentiable on $(-\delta,\delta)\backslash\{0\}$.
        Furthermore,
        \begin{align}
            &\text{when } \beta > 2 / \bar{\sigma}_{j^{\star}}, ~~ \frac{\partial}{\partial u^-} f^{\psi^0}(\beta) \stackrel{(1^-)}{=} \frac{\beta \sigma_{i^{\star}} (\alpha' - \alpha)}{\bar{\sigma}_{j^{\star}}} \stackrel{(1^+)}{=} \frac{\partial}{\partial u^+} f^{\psi^0}(\beta), \\
            &\text{when } \beta = 2 / \bar{\sigma}_{j^{\star}}, ~~ \frac{\partial}{\partial u^-} f^{\psi^0}(\beta) \stackrel{(2^-)}{=} \frac{\beta^2 \sigma_{i^{\star}} (\alpha' - \alpha)}{2} \neq +\infty \stackrel{(2^+)}{=} \frac{\partial}{\partial u^+} f^{\psi^0}(\beta), \\
            &\text{when } \beta < 2 / \bar{\sigma}_{j^{\star}}, ~~ \frac{\partial}{\partial u^-} f^{\psi^0}(\beta) \stackrel{(2^-)}{=} \frac{\beta^2 \sigma_{i^{\star}} (\alpha' - \alpha)}{2} \stackrel{(3^+)}{=} \frac{\partial}{\partial u^+} f^{\psi^0}(\beta).
        \end{align}
        Thus, $u\mapsto f^{\psi^u}(\beta)$ is also differentiable at $u=0$, except when $\beta = 2 / \bar{\sigma}_{j^{\star}}$.
        This ends the proof of Lemma \ref{lem:IGFF.modified.free.energy.derivative}.
    \end{proof}

\section*{Acknowledgements}

First, I would like to thank an anonymous referee and my advisor, Louis-Pierre Arguin, for their valuable comments that led to improvements in the presentation of this paper.
I also thank Yan Fyodorov for helpful correspondences and for pointing out the literature on the multi-scale log-REM model.
This work is supported by a NSERC Doctoral Program Alexander Graham Bell scholarship.

%
%

\bibliographystyle{alea2}
\bibliography{Ouimet_2017_bib}

\begin{thebibliography}{74}
\providecommand{\natexlab}[1]{#1}
\providecommand{\url}[1]{\texttt{#1}}
\providecommand{\urlprefix}{URL }
\expandafter\ifx\csname urlstyle\endcsname\relax
  \providecommand{\doi}[1]{doi:\discretionary{}{}{}#1}\else
  \providecommand{\doi}{doi:\discretionary{}{}{}\begingroup
  \urlstyle{rm}\Url}\fi
\providecommand{\eprint}[2][]{\url{#2}}

\bibitem[{Abe(2014)}]{MR3263552}
Y.~Abe.
\newblock Cover times for sequences of reversible {M}arkov chains on random
  graphs.
\newblock \emph{Kyoto J. Math.} \textbf{54}~(3), 555--576 (2014).
\newblock \href{http://www.ams.org/mathscinet-getitem?mr=MR3263552}{MR3263552}.

\bibitem[{Aizenman and Contucci(1998)}]{MR1657840}
M.~Aizenman and P.~Contucci.
\newblock On the stability of the quenched state in mean-field spin-glass
  models.
\newblock \emph{J. Statist. Phys.} \textbf{92}~(5-6), 765--783 (1998).
\newblock \href{http://www.ams.org/mathscinet-getitem?mr=MR1657840}{MR1657840}.

\bibitem[{Arguin and Aizenman(2009)}]{MR2537550}
L.-P. Arguin and M.~Aizenman.
\newblock On the structure of quasi-stationary competing particle systems.
\newblock \emph{Ann. Probab.} \textbf{37}~(3), 1080--1113 (2009).
\newblock \href{http://www.ams.org/mathscinet-getitem?mr=MR2537550}{MR2537550}.

\bibitem[{Arguin et~al.(2017{\natexlab{a}})Arguin, Belius and
  Bourgade}]{MR3594368}
L.-P. Arguin, D.~Belius and P.~Bourgade.
\newblock Maximum of the characteristic polynomial of random unitary matrices.
\newblock \emph{Comm. Math. Phys.} \textbf{349}~(2), 703--751
  (2017{\natexlab{a}}).
\newblock \href{http://www.ams.org/mathscinet-getitem?mr=MR3594368}{MR3594368}.

\bibitem[{Arguin et~al.(2017{\natexlab{b}})Arguin, Belius, Bourgade, Radziwill
  and Soundararajan}]{arXiv:1612.08575}
L.-P. Arguin, D.~Belius, P.~Bourgade, M.~Radziwill and K.~Soundararajan.
\newblock Maximum of the {R}iemann zeta function on a short interval of the
  critical line.
\newblock \emph{ArXiv Mathematics e-prints,} pages 1--24 (2017{\natexlab{b}}).
\newblock \href{http://arxiv.org/abs/1612.08575}{arXiv:1612.08575}.

\bibitem[{Arguin et~al.(2017{\natexlab{c}})Arguin, Belius and
  Harper}]{MR3619786}
L.-P. Arguin, D.~Belius and A.~J. Harper.
\newblock Maxima of a randomized {R}iemann zeta function, and branching random
  walks.
\newblock \emph{Ann. Appl. Probab.} \textbf{27}~(1), 178--215
  (2017{\natexlab{c}}).
\newblock \href{http://www.ams.org/mathscinet-getitem?mr=MR3619786}{MR3619786}.

\bibitem[{Arguin and Chatterjee(2013)}]{MR3055263}
L.-P. Arguin and S.~Chatterjee.
\newblock Random overlap structures: properties and applications to spin
  glasses.
\newblock \emph{Probab. Theory Related Fields} \textbf{156}~(1-2), 375--413
  (2013).
\newblock \href{http://www.ams.org/mathscinet-getitem?mr=MR3055263}{MR3055263}.

\bibitem[{Arguin and Ouimet(2016)}]{MR3541850}
L.-P. Arguin and F.~Ouimet.
\newblock Extremes of the two-dimensional {G}aussian free field with
  scale-dependent variance.
\newblock \emph{ALEA Lat. Am. J. Probab. Math. Stat.} \textbf{13}~(2), 779--808
  (2016).
\newblock \href{http://www.ams.org/mathscinet-getitem?mr=MR3541850}{MR3541850}.

\bibitem[{Arguin and Tai(2017)}]{arXiv:1706.08462}
L.-P. Arguin and W.~Tai.
\newblock Is the {R}iemann zeta function in a short interval a 1-{R}{S}{B} spin
  glass ?
\newblock \emph{ArXiv Mathematics e-prints,} pages 1--20 (2017).
\newblock \href{http://arxiv.org/abs/1706.08462}{arXiv:1706.08462}.

\bibitem[{Arguin and Zindy(2014)}]{MR3211001}
L.-P. Arguin and O.~Zindy.
\newblock Poisson-{D}irichlet statistics for the extremes of a log-correlated
  {G}aussian field.
\newblock \emph{Ann. Appl. Probab.} \textbf{24}~(4), 1446--1481 (2014).
\newblock \href{http://www.ams.org/mathscinet-getitem?mr=MR3211001}{MR3211001}.

\bibitem[{Arguin and Zindy(2015)}]{MR3354619}
L.-P. Arguin and O.~Zindy.
\newblock Poisson-{D}irichlet statistics for the extremes of the
  two-dimensional discrete {G}aussian free field.
\newblock \emph{Electron. J. Probab.} \textbf{20}, no. 59, 19 (2015).
\newblock \href{http://www.ams.org/mathscinet-getitem?mr=MR3354619}{MR3354619}.

\bibitem[{Belius(2013)}]{MR3129800}
D.~Belius.
\newblock Gumbel fluctuations for cover times in the discrete torus.
\newblock \emph{Probab. Theory Related Fields} \textbf{157}~(3-4), 635--689
  (2013).
\newblock \href{http://www.ams.org/mathscinet-getitem?mr=MR3129800}{MR3129800}.

\bibitem[{Belius and Kistler(2017)}]{MR3602852}
D.~Belius and N.~Kistler.
\newblock The subleading order of two dimensional cover times.
\newblock \emph{Probab. Theory Related Fields} \textbf{167}~(1-2), 461--552
  (2017).
\newblock \href{http://www.ams.org/mathscinet-getitem?mr=MR3602852}{MR3602852}.

\bibitem[{Berestycki(2015)}]{Berestycki2015ln}
N.~Berestycki.
\newblock Introduction to the {G}aussian {F}ree {F}ield and {L}iouville
  {Q}uantum {G}ravity.
\newblock \emph{Draft lecture notes}  (2015).
\newblock \newline[URL]~
  \url{http://www.statslab.cam.ac.uk/~beresty/Articles/oxford.pdf}.

\bibitem[{Bolthausen and Kistler(2006)}]{MR2209333}
E.~Bolthausen and N.~Kistler.
\newblock On a nonhierarchical version of the generalized random energy model.
\newblock \emph{Ann. Appl. Probab.} \textbf{16}~(1), 1--14 (2006).
\newblock \href{http://www.ams.org/mathscinet-getitem?mr=MR2209333}{MR2209333}.

\bibitem[{Bolthausen and Kistler(2009)}]{MR2531095}
E.~Bolthausen and N.~Kistler.
\newblock On a nonhierarchical version of the generalized random energy model.
  {II}. {U}ltrametricity.
\newblock \emph{Stochastic Process. Appl.} \textbf{119}~(7), 2357--2386 (2009).
\newblock \href{http://www.ams.org/mathscinet-getitem?mr=MR2531095}{MR2531095}.

\bibitem[{Bolthausen and Kistler(2012)}]{MR3372858}
E.~Bolthausen and N.~Kistler.
\newblock A quenched large deviation principle and a {P}arisi formula for a
  perceptron version of the {GREM}.
\newblock In \emph{Probability in complex physical systems}, volume~11 of
  \emph{Springer Proc. Math.}, pages 425--442. Springer, Heidelberg (2012).
\newblock \href{http://www.ams.org/mathscinet-getitem?mr=MR3372858}{MR3372858}.

\bibitem[{Bolthausen and Sznitman(1998)}]{MR1652734}
E.~Bolthausen and A.-S. Sznitman.
\newblock On {R}uelle's probability cascades and an abstract cavity method.
\newblock \emph{Comm. Math. Phys.} \textbf{197}~(2), 247--276 (1998).
\newblock \href{http://www.ams.org/mathscinet-getitem?mr=MR1652734}{MR1652734}.

\bibitem[{Bolthausen and Sznitman(2002)}]{MR1890289}
E.~Bolthausen and A.-S. Sznitman.
\newblock \emph{Ten {L}ectures on {R}andom {M}edia}, volume~32 of \emph{DMV
  Seminar}.
\newblock Birkh\"auser Verlag, Basel (2002).
\newblock ISBN 978-3-7643-6703-2.
\newblock \href{http://www.ams.org/mathscinet-getitem?mr=MR1890289}{MR1890289}.

\bibitem[{Bovier(2006)}]{MR2252929}
A.~Bovier.
\newblock \emph{Statistical {M}echanics of {D}isordered {S}ystems: A
  {M}athematical {P}erspective}, volume~18 of \emph{Cambridge Series in
  Statistical and Probabilistic Mathematics}.
\newblock Cambridge University Press, Cambridge (2006).
\newblock ISBN 978-0-521-84991-3.
\newblock \href{http://www.ams.org/mathscinet-getitem?mr=MR2252929}{MR2252929}.

\bibitem[{Bovier and Hartung(2014)}]{MR3164771}
A.~Bovier and L.~Hartung.
\newblock The extremal process of two-speed branching {B}rownian motion.
\newblock \emph{Electron. J. Probab.} \textbf{19}, no. 18, 28 (2014).
\newblock \href{http://www.ams.org/mathscinet-getitem?mr=MR3164771}{MR3164771}.

\bibitem[{Bovier and Hartung(2015)}]{MR3351476}
A.~Bovier and L.~Hartung.
\newblock Variable speed branching {B}rownian motion 1. {E}xtremal processes in
  the weak correlation regime.
\newblock \emph{ALEA Lat. Am. J. Probab. Math. Stat.} \textbf{12}~(1), 261--291
  (2015).
\newblock \href{http://www.ams.org/mathscinet-getitem?mr=MR3351476}{MR3351476}.

\bibitem[{Bovier and Kurkova(2004{\natexlab{a}})}]{MR2070334}
A.~Bovier and I.~Kurkova.
\newblock Derrida's generalised random energy models. {I}. {M}odels with
  finitely many hierarchies.
\newblock \emph{Ann. Inst. H. Poincar\'e Probab. Statist.} \textbf{40}~(4),
  439--480 (2004{\natexlab{a}}).
\newblock \href{http://www.ams.org/mathscinet-getitem?mr=MR2070334}{MR2070334}.

\bibitem[{Bovier and Kurkova(2004{\natexlab{b}})}]{MR2070335}
A.~Bovier and I.~Kurkova.
\newblock Derrida's generalized random energy models. {II}. {M}odels with
  continuous hierarchies.
\newblock \emph{Ann. Inst. H. Poincar\'e Probab. Statist.} \textbf{40}~(4),
  481--495 (2004{\natexlab{b}}).
\newblock \href{http://www.ams.org/mathscinet-getitem?mr=MR2070335}{MR2070335}.

\bibitem[{Cao et~al.(2016)Cao, Fyodorov and Le~Doussal}]{SciPostPhys.1.2.011}
X.~Cao, Y.~V. Fyodorov and P.~Le~Doussal.
\newblock {One step replica symmetry breaking and extreme order statistics of
  logarithmic REMs}.
\newblock \emph{SciPost Phys.} \textbf{1}~(2), 011 (2016).
\newblock \doi{10.21468/SciPostPhys.1.2.011}.
\newblock \newline[URL]~
  \url{https://scipost.org/10.21468/SciPostPhys.1.2.011}.

\bibitem[{Capocaccia et~al.(1987)Capocaccia, Cassandro and Picco}]{MR883541}
D.~Capocaccia, M.~Cassandro and P.~Picco.
\newblock On the existence of thermodynamics for the generalized random energy
  model.
\newblock \emph{J. Statist. Phys.} \textbf{46}~(3-4), 493--505 (1987).
\newblock \href{http://www.ams.org/mathscinet-getitem?mr=MR883541}{MR883541}.

\bibitem[{Carpentier and Le~Doussal(2001)}]{PhysRevE.63.026110}
D.~Carpentier and P.~Le~Doussal.
\newblock Glass transition of a particle in a random potential, front selection
  in nonlinear renormalization group, and entropic phenomena in {L}iouville and
  sinh-{G}ordon models.
\newblock \emph{Phys. Rev. E} \textbf{63}~(2), 026110 (2001).
\newblock \doi{10.1103/PhysRevE.63.026110}.
\newblock \newline[URL]~
  \url{https://link.aps.org/doi/10.1103/PhysRevE.63.026110}.

\bibitem[{Chhaibi et~al.(2017)Chhaibi, Madaule and Najnudel}]{arXiv:1607.00243}
R.~Chhaibi, T.~Madaule and J.~Najnudel.
\newblock On the maximum of the c$\beta$e field.
\newblock \emph{ArXiv Mathematics e-prints,} pages 1--74 (2017).
\newblock \href{http://arxiv.org/abs/1607.00243}{arXiv:1607.00243}.

\bibitem[{Comets et~al.(2013)Comets, Gallesco, Popov and
  Vachkovskaia}]{MR3126579}
F.~Comets, C.~Gallesco, S.~Popov and M.~Vachkovskaia.
\newblock On large deviations for the cover time of two-dimensional torus.
\newblock \emph{Electron. J. Probab.} \textbf{18}, no. 96, 18 (2013).
\newblock \href{http://www.ams.org/mathscinet-getitem?mr=MR3126579}{MR3126579}.

\bibitem[{Contucci and Giardin\`a(2005)}]{MR2219862}
P.~Contucci and C.~Giardin\`a.
\newblock Spin-glass stochastic stability: a rigorous proof.
\newblock \emph{Ann. Henri Poincar\'e} \textbf{6}~(5), 915--923 (2005).
\newblock \href{http://www.ams.org/mathscinet-getitem?mr=MR2219862}{MR2219862}.

\bibitem[{Dembo et~al.(2003)Dembo, Peres and Rosen}]{MR1998762}
A.~Dembo, Y.~Peres and J.~Rosen.
\newblock Brownian motion on compact manifolds: cover time and late points.
\newblock \emph{Electron. J. Probab.} \textbf{8}, no. 15, 14 (2003).
\newblock \href{http://www.ams.org/mathscinet-getitem?mr=MR1998762}{MR1998762}.

\bibitem[{Dembo et~al.(2004)Dembo, Peres, Rosen and Zeitouni}]{MR2123929}
A.~Dembo, Y.~Peres, J.~Rosen and O.~Zeitouni.
\newblock Cover times for {B}rownian motion and random walks in two dimensions.
\newblock \emph{Ann. of Math. (2)} \textbf{160}~(2), 433--464 (2004).
\newblock \href{http://www.ams.org/mathscinet-getitem?mr=MR2123929}{MR2123929}.

\bibitem[{Dembo et~al.(2006)Dembo, Peres, Rosen and Zeitouni}]{MR2206347}
A.~Dembo, Y.~Peres, J.~Rosen and O.~Zeitouni.
\newblock Late points for random walks in two dimensions.
\newblock \emph{Ann. Probab.} \textbf{34}~(1), 219--263 (2006).
\newblock \href{http://www.ams.org/mathscinet-getitem?mr=MR2206347}{MR2206347}.

\bibitem[{Derrida(1985)}]{Derrida85}
B.~Derrida.
\newblock A generalization of the random energy model which includes
  correlations between energies.
\newblock \emph{J. Physique Lett.} \textbf{46}~(9), 401--407 (1985).
\newblock \doi{10.1051/jphyslet:01985004609040100}.
\newblock \newline[URL]~
  \url{https://hal.archives-ouvertes.fr/jpa-00232535/document}.

\bibitem[{Ding(2012)}]{MR2946152}
J.~Ding.
\newblock On cover times for 2{D} lattices.
\newblock \emph{Electron. J. Probab.} \textbf{17}, no. 45, 18 (2012).
\newblock \href{http://www.ams.org/mathscinet-getitem?mr=MR2946152}{MR2946152}.

\bibitem[{Ding(2014)}]{MR3178464}
J.~Ding.
\newblock Asymptotics of cover times via {G}aussian free fields: bounded-degree
  graphs and general trees.
\newblock \emph{Ann. Probab.} \textbf{42}~(2), 464--496 (2014).
\newblock \href{http://www.ams.org/mathscinet-getitem?mr=MR3178464}{MR3178464}.

\bibitem[{Ding et~al.(2012)Ding, Lee and Peres}]{MR2912708}
J.~Ding, J.~R. Lee and Y.~Peres.
\newblock Cover times, blanket times, and majorizing measures.
\newblock \emph{Ann. of Math. (2)} \textbf{175}~(3), 1409--1471 (2012).
\newblock \href{http://www.ams.org/mathscinet-getitem?mr=MR2912708}{MR2912708}.

\bibitem[{Ding and Zeitouni(2012)}]{MR2921974}
J.~Ding and O.~Zeitouni.
\newblock A sharp estimate for cover times on binary trees.
\newblock \emph{Stochastic Process. Appl.} \textbf{122}~(5), 2117--2133 (2012).
\newblock \href{http://www.ams.org/mathscinet-getitem?mr=MR2921974}{MR2921974}.

\bibitem[{Dovbysh and Sudakov(1982)}]{MR666087}
L.~N. Dovbysh and V.~N. Sudakov.
\newblock Gram-de {F}inetti matrices.
\newblock \emph{Zap. Nauchn. Sem. Leningrad. Otdel. Mat. Inst. Steklov. (LOMI)}
  \textbf{119}, 77--86 (1982).
\newblock \href{http://www.ams.org/mathscinet-getitem?mr=MR666087}{MR666087}.

\bibitem[{Dynkin(1980)}]{MR585179}
E.~B. Dynkin.
\newblock Markov processes and random fields.
\newblock \emph{Bull. Amer. Math. Soc. (N.S.)} \textbf{3}~(3), 975--999 (1980).
\newblock \href{http://www.ams.org/mathscinet-getitem?mr=MR585179}{MR585179}.

\bibitem[{Fang and Zeitouni(2012{\natexlab{a}})}]{MR2968674}
M.~Fang and O.~Zeitouni.
\newblock Branching random walks in time inhomogeneous environments.
\newblock \emph{Electron. J. Probab.} \textbf{17}, no. 67, 18
  (2012{\natexlab{a}}).
\newblock \href{http://www.ams.org/mathscinet-getitem?mr=MR2968674}{MR2968674}.

\bibitem[{Fang and Zeitouni(2012{\natexlab{b}})}]{MR2981635}
M.~Fang and O.~Zeitouni.
\newblock Slowdown for time inhomogeneous branching {B}rownian motion.
\newblock \emph{J. Stat. Phys.} \textbf{149}~(1), 1--9 (2012{\natexlab{b}}).
\newblock \href{http://www.ams.org/mathscinet-getitem?mr=MR2981635}{MR2981635}.

\bibitem[{Fyodorov and Bouchaud(2008{\natexlab{a}})}]{MR2430565}
Y.~V. Fyodorov and J.-P. Bouchaud.
\newblock Freezing and extreme-value statistics in a random energy model with
  logarithmically correlated potential.
\newblock \emph{J. Phys. A} \textbf{41}~(37), 372001, 12 (2008{\natexlab{a}}).
\newblock \href{http://www.ams.org/mathscinet-getitem?mr=MR2430565}{MR2430565}.

\bibitem[{Fyodorov and Bouchaud(2008{\natexlab{b}})}]{MR2425780}
Y.~V. Fyodorov and J.-P. Bouchaud.
\newblock Statistical mechanics of a single particle in a multiscale random
  potential: {P}arisi landscapes in finite-dimensional {E}uclidean spaces.
\newblock \emph{J. Phys. A} \textbf{41}~(32), 324009, 25 (2008{\natexlab{b}}).
\newblock \href{http://www.ams.org/mathscinet-getitem?mr=MR2425780}{MR2425780}.

\bibitem[{Fyodorov et~al.(2009)Fyodorov, Le~Doussal and Rosso}]{MR2882779}
Y.~V. Fyodorov, P.~Le~Doussal and A.~Rosso.
\newblock Statistical mechanics of logarithmic {REM}: duality, freezing and
  extreme value statistics of {$1/f$} noises generated by {G}aussian free
  fields.
\newblock \emph{J. Stat. Mech. Theory Exp.} ~(10), P10005, 32 (2009).
\newblock \href{http://www.ams.org/mathscinet-getitem?mr=MR2882779}{MR2882779}.

\bibitem[{Ghirlanda and Guerra(1998)}]{MR1662161}
S.~Ghirlanda and F.~Guerra.
\newblock General properties of overlap probability distributions in disordered
  spin systems. {T}owards {P}arisi ultrametricity.
\newblock \emph{J. Phys. A} \textbf{31}~(46), 9149--9155 (1998).
\newblock \href{http://www.ams.org/mathscinet-getitem?mr=MR1662161}{MR1662161}.

\bibitem[{Harper(2013)}]{arXiv:1304.0677}
A.~J. Harper.
\newblock A note on the maximum of the {R}iemann zeta function, and
  log-correlated random variables.
\newblock \emph{ArXiv Mathematics e-prints,} pages 1--26 (2013).
\newblock \href{http://arxiv.org/abs/1304.0677}{arXiv:1304.0677}.

\bibitem[{Jagannath(2016)}]{MR3539644}
A.~Jagannath.
\newblock On the overlap distribution of branching random walks.
\newblock \emph{Electron. J. Probab.} \textbf{21}, Paper No. 50 16 (2016).
\newblock \href{http://www.ams.org/mathscinet-getitem?mr=MR3539644}{MR3539644}.

\bibitem[{Jagannath(2017)}]{MR3628881}
A.~Jagannath.
\newblock Approximate ultrametricity for random measures and applications to
  spin glasses.
\newblock \emph{Comm. Pure Appl. Math.} \textbf{70}~(4), 611--664 (2017).
\newblock \href{http://www.ams.org/mathscinet-getitem?mr=MR3628881}{MR3628881}.

\bibitem[{Lawler(1991)}]{MR1117680}
G.~F. Lawler.
\newblock \emph{Intersections of {R}andom {W}alks}.
\newblock Probability and its Applications. Birkh\"auser Boston, Inc., Boston,
  MA (1991).
\newblock ISBN 0-8176-3557-2.
\newblock \href{http://www.ams.org/mathscinet-getitem?mr=MR1117680}{MR1117680}.

\bibitem[{Maillard and Zeitouni(2016)}]{MR3531703}
P.~Maillard and O.~Zeitouni.
\newblock Slowdown in branching {B}rownian motion with inhomogeneous variance.
\newblock \emph{Ann. Inst. Henri Poincar\'e Probab. Stat.} \textbf{52}~(3),
  1144--1160 (2016).
\newblock \href{http://www.ams.org/mathscinet-getitem?mr=MR3531703}{MR3531703}.

\bibitem[{Mallein(2015{\natexlab{a}})}]{MR3373310}
B.~Mallein.
\newblock \hspace{-0.4mm}{M}aximal \hspace{-0.4mm}displacement
  \hspace{-0.4mm}of \hspace{-0.4mm}a \hspace{-0.4mm}branching
  \hspace{-0.4mm}random \hspace{-0.4mm}walk \hspace{-0.4mm}in
  \hspace{-0.4mm}time-inhomogeneous environment.
\newblock \emph{Stoch. Proc. Appl.} \textbf{125}~(10), 3958--4019
  (2015{\natexlab{a}}).
\newblock \href{http://www.ams.org/mathscinet-getitem?mr=MR3373310}{MR3373310}.

\bibitem[{Mallein(2015{\natexlab{b}})}]{MR3361256}
B.~Mallein.
\newblock Maximal displacement in a branching random walk through interfaces.
\newblock \emph{Electron. J. Probab.} \textbf{20}, no. 68, 40
  (2015{\natexlab{b}}).
\newblock \href{http://www.ams.org/mathscinet-getitem?mr=MR3361256}{MR3361256}.

\bibitem[{M\'ezard et~al.(1987)M\'ezard, Parisi and Virasoro}]{MR1026102}
M.~M\'ezard, G.~Parisi and M.~A. Virasoro.
\newblock \emph{Spin {G}lass {T}heory and {B}eyond : {A}n {I}ntroduction to the
  {R}eplica {M}ethod and {I}ts {A}pplications}, volume~9 of \emph{World
  Scientific Lecture Notes in Physics}.
\newblock World Scientific Publishing Co., Inc., Teaneck, NJ (1987).
\newblock ISBN 9971-50-116-3.
\newblock \href{http://www.ams.org/mathscinet-getitem?mr=MR1026102}{MR1026102}.

\bibitem[{Najnudel(2016)}]{arXiv:1611.05562}
J.~Najnudel.
\newblock On the extreme values of the {R}iemann zeta function on random
  intervals of the critical line.
\newblock \emph{ArXiv Mathematics e-prints,} pages 1--52 (2016).
\newblock \href{http://arxiv.org/abs/1611.05562}{arXiv:1611.05562}.

\bibitem[{Ouimet(2014)}]{Ouimet2014master}
F.~Ouimet.
\newblock \emph{\textit{{\'{E}}tude du maximum et des hauts points de la marche
  al\'eatoire branchante inhomog\`ene et du champ libre gaussien
  inhomog\`ene}}.
\newblock Master's thesis, Universit\'e de Montr\'eal (2014).
\newblock \newline[URL]~
  \url{https://papyrus.bib.umontreal.ca/xmlui/handle/1866/11510}.

\bibitem[{Ouimet(2017)}]{arXiv:1509.08172}
F.~Ouimet.
\newblock Maxima of branching random walks with piecewise constant variance.
\newblock \emph{Accepted. To appear in the Brazilian Journal of Probability and
  Statistics,} pages 1--28 (2017).
\newblock \href{http://arxiv.org/abs/1509.08172}{arXiv:1509.08172}.

\bibitem[{Panchenko(2010{\natexlab{a}})}]{MR2599202}
D.~Panchenko.
\newblock A connection between the {G}hirlanda-{G}uerra identities and
  ultrametricity.
\newblock \emph{Ann. Probab.} \textbf{38}~(1), 327--347 (2010{\natexlab{a}}).
\newblock \href{http://www.ams.org/mathscinet-getitem?mr=MR2599202}{MR2599202}.

\bibitem[{Panchenko(2010{\natexlab{b}})}]{MR2600075}
D.~Panchenko.
\newblock The {G}hirlanda-{G}uerra identities for mixed {$p$}-spin model.
\newblock \emph{C. R. Math. Acad. Sci. Paris} \textbf{348}~(3-4), 189--192
  (2010{\natexlab{b}}).
\newblock \href{http://www.ams.org/mathscinet-getitem?mr=MR2600075}{MR2600075}.

\bibitem[{Panchenko(2010{\natexlab{c}})}]{MR2679002}
D.~Panchenko.
\newblock On the {D}ovbysh-{S}udakov representation result.
\newblock \emph{Electron. Commun. Probab.} \textbf{15}, 330--338
  (2010{\natexlab{c}}).
\newblock \href{http://www.ams.org/mathscinet-getitem?mr=MR2679002}{MR2679002}.

\bibitem[{Panchenko(2011)}]{MR2825947}
D.~Panchenko.
\newblock Ghirlanda-{G}uerra identities and ultrametricity: an elementary proof
  in the discrete case.
\newblock \emph{C. R. Math. Acad. Sci. Paris} \textbf{349}~(13-14), 813--816
  (2011).
\newblock \href{http://www.ams.org/mathscinet-getitem?mr=MR2825947}{MR2825947}.

\bibitem[{Panchenko(2012)}]{MR2945621}
D.~Panchenko.
\newblock A unified stability property in spin glasses.
\newblock \emph{Comm. Math. Phys.} \textbf{313}~(3), 781--790 (2012).
\newblock \href{http://www.ams.org/mathscinet-getitem?mr=MR2945621}{MR2945621}.

\bibitem[{Panchenko(2013{\natexlab{a}})}]{MR2999044}
D.~Panchenko.
\newblock The {P}arisi ultrametricity conjecture.
\newblock \emph{Ann. of Math. (2)} \textbf{177}~(1), 383--393
  (2013{\natexlab{a}}).
\newblock \href{http://www.ams.org/mathscinet-getitem?mr=MR2999044}{MR2999044}.

\bibitem[{Panchenko(2013{\natexlab{b}})}]{MR3052333}
D.~Panchenko.
\newblock \emph{The {S}herrington-{K}irkpatrick {M}odel}.
\newblock Springer Monographs in Mathematics. Springer, New York
  (2013{\natexlab{b}}).
\newblock ISBN 978-1-4614-6288-0.
\newblock \href{http://www.ams.org/mathscinet-getitem?mr=MR3052333}{MR3052333}.

\bibitem[{Paquette and Zeitouni(2016)}]{arXiv:1602.08875}
E.~Paquette and O.~Zeitouni.
\newblock The maximum of the cue field.
\newblock \emph{ArXiv Mathematics e-prints,} pages 1--67 (2016).
\newblock \href{http://arxiv.org/abs/1602.08875}{arXiv:1602.08875}.

\bibitem[{Rhodes and Vargas(2014)}]{MR3274356}
R.~Rhodes and V.~Vargas.
\newblock Gaussian multiplicative chaos and applications: a review.
\newblock \emph{Probab. Surv.} \textbf{11}, 315--392 (2014).
\newblock \href{http://www.ams.org/mathscinet-getitem?mr=MR3274356}{MR3274356}.

\bibitem[{Rockafellar(1970)}]{MR0274683}
R.~T. Rockafellar.
\newblock \emph{Convex {A}nalysis}.
\newblock Princeton Mathematical Series, No. 28. Princeton University Press,
  Princeton, N.J. (1970).
\newblock ISBN 0-691-08069-0.
\newblock \href{http://www.ams.org/mathscinet-getitem?mr=MR0274683}{MR0274683}.

\bibitem[{Ruelle(1987)}]{MR875300}
D.~Ruelle.
\newblock A mathematical reformulation of {D}errida's {REM} and {GREM}.
\newblock \emph{Comm. Math. Phys.} \textbf{108}~(2), 225--239 (1987).
\newblock \href{http://www.ams.org/mathscinet-getitem?mr=MR875300}{MR875300}.

\bibitem[{Saksman and Webb(2016)}]{arXiv:1604.08378}
E.~Saksman and C.~Webb.
\newblock Multiplicative chaos measures for a random model of the {R}iemann
  zeta function.
\newblock \emph{ArXiv Mathematics e-prints,} pages 1--36 (2016).
\newblock \href{http://arxiv.org/abs/1604.08378}{arXiv:1604.08378}.

\bibitem[{Sheffield(2007)}]{MR2322706}
S.~Sheffield.
\newblock Gaussian free fields for mathematicians.
\newblock \emph{Probab. Theory Related Fields} \textbf{139}~(3-4), 521--541
  (2007).
\newblock \href{http://www.ams.org/mathscinet-getitem?mr=MR2322706}{MR2322706}.

\bibitem[{Talagrand(2010)}]{MR2678900}
M.~Talagrand.
\newblock Construction of pure states in mean field models for spin glasses.
\newblock \emph{Probab. Theory Related Fields} \textbf{148}~(3-4), 601--643
  (2010).
\newblock \href{http://www.ams.org/mathscinet-getitem?mr=MR2678900}{MR2678900}.

\bibitem[{Talagrand(2011{\natexlab{a}})}]{MR2731561}
M.~Talagrand.
\newblock \emph{Mean {F}ield {M}odels for {S}pin {G}lasses. {V}olume {I}: Basic
  {E}xamples}, volume~54 of \emph{Ergebnisse der Mathematik und ihrer
  Grenzgebiete. 3. Folge. A Series of Modern Surveys in Mathematics [Results in
  Mathematics and Related Areas. 3rd Series. A Series of Modern Surveys in
  Mathematics]}.
\newblock Springer-Verlag, Berlin (2011{\natexlab{a}}).
\newblock ISBN 978-3-642-15201-6.
\newblock \href{http://www.ams.org/mathscinet-getitem?mr=MR2731561}{MR2731561}.

\bibitem[{Talagrand(2011{\natexlab{b}})}]{MR3024566}
M.~Talagrand.
\newblock \emph{Mean {F}ield {M}odels for {S}pin {G}lasses. {V}olume {II}:
  {A}dvanced {R}eplica-{S}ymmetry and {L}ow {T}emperature}, volume~55 of
  \emph{Ergebnisse der Mathematik und ihrer Grenzgebiete. 3. Folge. A Series of
  Modern Surveys in Mathematics [Results in Mathematics and Related Areas. 3rd
  Series. A Series of Modern Surveys in Mathematics]}.
\newblock Springer, Heidelberg (2011{\natexlab{b}}).
\newblock ISBN 978-3-642-22253-5.
\newblock \href{http://www.ams.org/mathscinet-getitem?mr=MR3024566}{MR3024566}.

\bibitem[{Zeitouni(2014)}]{Zeitouni2014ln}
O.~Zeitouni.
\newblock Gaussian fields.
\newblock \emph{Notes for Lectures}  (2014).
\newblock \newline[URL]~
  \url{http://www.cims.nyu.edu/~zeitouni/notesGauss.pdf}.

\end{thebibliography}

\end{document}